\documentclass[10pt]{amsart}

\usepackage{amsmath,amssymb}

\usepackage{mathrsfs}

\usepackage[active]{srcltx}
\usepackage{graphicx,color}

\usepackage{stackengine}
\usepackage{scalerel}

\usepackage{epic}
\usepackage{pstricks}
\usepackage{comment}
\usepackage{amsthm}
\def\plus{\raisebox{0.95pt}{\scalebox{0.5}{\ensuremath{+}}}}

\newtheorem{theorem}{Theorem}[section]
\newtheorem{proposition}[theorem]{Proposition}

\newtheorem{lemma}[theorem]{Lemma}

\newtheorem{corollary}[theorem]{Corollary}
\theoremstyle{definition}

\newtheorem{remark}[theorem]{Remark}

\def\R{\mathbb{R}}

\def\N{\mathbb{N}}
\def\C{\mathbb{C}}
\def\ZZ{\mathbb{Z}}

\def\la{\lambda}

\def \supp{\operatorname{supp}}
\def\esssup{\textrm{ess\,sup}}

\def\cA{{\mathcal A}}

\def\cD{{\mathcal D}}
\def\cX{{\mathcal X}}
\def\cF{{\mathcal F}}

\def\cK{{\mathcal K}}
\def\cL{{\mathcal L}}
\def\cM{{\mathcal M}}

\def\cR{{\mathcal R}}
\def\cS{{\mathcal S}}

\def\cE{{\mathcal E}}

\def\cY{{\mathcal Y}}
\def\cX{{\mathcal X}}

\def\cT{{\mathcal T}}
\def\cM{{\mathcal M}}

\def\cP{{\mathcal P}}
\def\cZ{{\mathcal Z}}

\def\bK{\mathbf{K}}

\def\bfx{\mathbf{x}}

\newcommand{\ud}{\mathrm{d}}

\newcommand{\X}{{X}}
\newcommand{\A}{{A}}

\newcommand{\wX}{\widetilde{X}}
\newcommand{\wT}{\widetilde{T}}
\newcommand{\wA}{\widetilde{A}}

\newcommand{\es}{\mathrm{ess\,sup}}

\newcommand{\widecheck}{\check}

\newcommand{\xo}{\vec{X}}

\definecolor{darkred}{rgb}{0.7,0.1,0.1}

\title[]{Maximal regularity estimates for the abstract Cauchy problems}

\author[]{Sebastian Kr\'ol \& Mieczysław Mastyło \& Jarosław Sarnowski}

\address{Sebastian Kr{\'o}l, Faculty of Mathematics and Computer Science, Adam Mickiewicz University in Pozna{\'n}, ul. Uniwersytetu Pozna{\'n}skiego 4, 61-614 Pozna{\'n}, Poland}
\email{sebastian.krol@amu.edu.pl}
 
\address{Mieczysław Mastyło, Faculty of Mathematics and Computer Science, Adam Mickiewicz University in Pozna{\'n}, ul. Uniwersytetu Pozna{\'n}skiego 4, 61-614 Pozna{\'n}, Poland}
\email{mieczyslaw.mastylo@amu.edu.pl}

\address{Jarosław Sarnowski, Faculty of Mathematics and Computer Science, Nicolaus Copernicus University in Toru{\'n}, ul. Chopina 12/18, 87-100 Toru{\'n}, Poland}
\email{jsarnowski@doktorant.umk.pl}

\begin{document}

\thanks{The first and third author were partially supported by the NAWA/NSF grant BPN/NSF/2023/1/00001. The second author was supported by the National Science Centre, Poland, Project
no. 2019/33/B/ST1/00165.} 



\begin{abstract} 
In this work, we extend the Da Prato–Grisvard theory of maximal regularity estimates for sectorial operators in interpolation spaces. Specifically, for any generator $-\mathcal{A}$ of an analytic semigroup on a Banach space $\mathcal{X}$, we identify the interpolation spaces between $\mathcal{X}$ and the domain $D_{\mathcal{A}}$ of $\mathcal{A}$ in which the part of $\mathcal{A}$ satisfies certain maximal regularity estimates. We also establish several new results concerning both homogeneous and inhomogeneous $L^1$-maximal regularity estimates, extending and completing recent findings in the literature. These results are motivated not only by applications to problems in areas such as fluid mechanics but also by the intrinsic theoretical interest of the subject. In particular, we address the optimal choice of data spaces for the Cauchy problem associated with $\mathcal{A}$, ensuring the existence of strong solutions with global-in-time control of their derivatives. This control is measured via the homogeneous parts of the interpolation norms in the spatial variable and weighted Lebesgue norms over the time interval. Furthermore, we characterize weighted $L^1$-estimates and establish their relationship with unweighted estimates. Additionally, we reformulate the characterization condition for $L^1$-maximal regularity due to Kalton and Portal in {\it a priori} terms that do not rely on semigroup operators. Finally, we introduce a new interpolation framework for $L^p$-maximal regularity estimates, where $p \in (1, \infty)$, within interpolation spaces generated by non-classical interpolation functors.
\end{abstract}

\renewcommand{\subjclassname}{\textup{2020} Mathematics Subject Classification}

\maketitle

\medskip

\section{Introduction}

In this paper, we study the solvability and maximal regularity estimates for solutions to the following abstract inhomogeneous Cauchy problem:
\begin{equation}
u' + A u = f \, \textrm{ on } (0,\tau)\quad  \textrm{ and }\quad  u(0)= x, \tag*{$(CP)_{A,f,x}$}
\end{equation}
where $A$ denotes a closed, linear operator on a complex Banach space $(X,\|\cdot\|_X)$, 
 $f$~is the forcing term,  $x\in X$ represents the initial value, and $\tau\in (0,\infty]$ is given. 
To outline our main results we briefly present some facts that motivated our study. 

The basic concept we deal with is the notion of {\it $L^p$-maximal regularity estimates}; see (MRE) for $E = L^p$, $p\in [1,\infty]$, in Subsection \ref{sec mr def} for the corresponding definition.  
In the case when an operator $A$ satisfy such $L^p$-estimates, we say that it has the {\it $L^p$-maximal regularity on $(0,\tau)$} (in short, $L^p$-m.r.\,on $(0,\tau)$). This concept has a rich history, and its interest stems from the fact that it is an important tool in the theory of non-linear evolution equations arising, e.g. in fluid mechanics. We refer the reader to, e.g. \cite{Am95}, \cite{KuWe04}, \cite{Lu09}, \cite{PrSi16} and \cite{HNVW24}, for the abstract theory and its application; see also references therein. The developed theory combines techniques and methods from harmonic analysis (singular integral operators, Fourier multipliers),  the geometry of Banach spaces, the interpolation theory, and  abstract evolution equations (semigroup theory). Some of these connections are further explored in this paper.

In this paper, we primarily study $L^p$-maximal regularity estimates when $p= 1$ and $p=\infty$. Such estimates are more restrictive for the geometry of $X$ than those for $p\in (1,\infty)$. Indeed, the classical result by Baillon \cite{Ba80} (resp. by Guerre-Delabri\`ere \cite{GuDe95}) asserts that the existence of an unbounded generator of an analytic semigroup on a Banach space $X$ which has $L^\infty$-m.r. (resp. $L^1$-m.r.) is equivalent to the fact that  $X$ contains an isomorphic copy of $c_0$ (resp. complemented copy of  
$\ell^1$). In particular, when $X$ is reflexive, no unbounded operator $A$ does have $L^1$- and $L^\infty$-m.r. 
Even in the case when $X$ contains such copies, there exist
examples of concrete differential operators $A$, which play a crucial role in applications to PDEs, that do not satisfy the corresponding estimates. For instance, 
the generators of the heat and Poisson semigroups on $X := C_0 (\R^n)$ or $X = L^1(\R^n)$; see, e.g. \cite[Subsection 17.4.a, p.657]{HNVW24}.

One way to overcome these limitations is by employing interpolation techniques. The well-known result by Da Prato and Grisvard \cite{DaPrGr75, DaPrGr79} states that for any generator $-\cA$ of an analytic semigroup $\cT$ on any Banach space $\cX$, the part $A$ of $\cA$ in the classical real interpolation spaces $X:=(\cX,D_\cA)_{\theta,p}$ with $\theta\in (0,1)$ and $p\in [1,\infty]$ has $L^p$-maximal regularity on finite intervals. Recall that by the definition of the part of an operator, $D(A) := \{x \in D(\cA) \cap X : \cA x \in X\}$ and $Ax := \cA x$ for $x \in D(A)$. Here, $D_\cA$ stands for the domain $D(\cA)$ of $\cA$ equipped with the graph norm, and  $(\,\cdot\,)_{\theta, p}$ is the real interpolation functor. 
An adaptation of the Kalton-Lancien result \cite[Theorem 3.3]{KaLa00} shows that, in general, we do not have $L^p$-maximal regularity estimates over $\R_+$ in the Da Prato-Grisvard theorem; see Remark \ref{rem on l1}(b). 
While this example may be considered  artificial, we demonstrate that such a lack of  $L^1$-estimates over $\R_+$ occurs in the case of the part $A$ of the negative Laplacian $\cA = -\Delta$ on $\cX: = L^1(\R^n)$; see Theorem \ref{thm main}(iii). In addition, we study the measuring of the global time-regularity via weighted $L^1$ norms; cf. the assertions (iii), (v) and (vi) of Theorem \ref{thm main}. For, we provide a weighted counterpart of the Kalton-Portal result \cite[Theorem 3.6]{KaPo08} in Theorem \ref{weighted L1 mr} and explore its connection to the unweighted one in Corollary \ref{L1 for decr} and Proposition \ref{L1 extrapol var 2}. In particular, Proposition \ref{L1 extrapol var 2} provides the converse implication to \cite[Proposition 17.2.36(2), p.610]{HNVW24}, which has been left open therein. We treat this example as motivation for studying estimates over $\R_+$ with respect to weaker (semi)norms than the original interpolation norms.

In this article, we examine the possibility of extending the Da Prato-Grisvard result to interpolation functors other than the classical real interpolation functors $(\cdot)_{\theta, p}$. 
We restrict our studies mainly to a class of functors $K_\Phi$ of $K$-method (our abstract results, Theorem \ref{interp result} and Corollary \ref{cor1} also refer to other class of functors). 
The class of admissible parameters $\Phi$, which could {\it a priori} range among all Banach function spaces over $(\R_+, \ud t)$, depends essentially on $p$ and/or the geometry of an underlying Banach space $\cX$. 
In the case of $p=1,\infty$, on the one hand, we have to take into account the results of Baillon and Guerre-Delabri\`{e}re mentioned above. It connects the problem of describing admissible $\Phi$ to  L\v{e}vy's type results from the interpolation theory (see, e.g. \cite[Theorems 4.6.25 and 4.6.28]{BrKr91}).
On the other hand we have to ensure an appropriate, working representation of the corresponding interpolation norm $\|\cdot\|_{K_\Phi(\cX,D_\cA)}$; see Proposition \ref{Komastu rep}. 
 It naturally imposes certain constrains on the class of $\Phi$, and consequently we deal here with weighted Lebesgue spaces $L^1_w$ and $L^\infty_w$ for arbitrary weight functions $w$; see also complementary results Propositions \ref{represent for K} and \ref{KJ}.

In Subsection \ref{subs H}, we discuss equivalent representations of the {\it homogeneous part} (see \eqref{homoge part}) of the classical real interpolation norms to the one provided in Proposition \ref{Komastu rep}. Our Proposition \ref{Haase equi} extends the result due to Haase \cite[Theorem 6.4.2]{Ha06} and shows the sharpness of the assumptions of some its ingredients. Furthermore, Proposition \ref{equiv for l1} allows to express the characterization condition for the $L^1$-m.r. due to Kalton and Portal \cite{KaPo08} in {\it a priori} terms of $A$ (i.e. its resolvent). Moreover, in Theorem \ref{optimal for l1}, for the negative generator $\cA$ of an analytic semigroup on a Banach space $\cX$, we introduce a subspace $X$ of $\cX$, which becomes a Banach space with respect to a natural norm and such that the part $A$ of $\cA$ in $X$ satisfies $L^1$-m.r. estimates on whole $\R_+$. 
The space $X$ can be naturally seen as an end-point of the real interpolation scale $(\cX, D_\cA)_{\theta, 1}$ with $\theta\in (0,1)$; cf. Remark \ref{rem on X1}. In comparison with the Da Prato-Grisvard theorem, a remarkable phenomenon occurs here.

We also study an optimal choice of the space of forcing terms $f$ in $(CP)_{A,f,x}$, which guarantees the existence of strong solutions with prescribed global-in-time control of their derivatives, measured via the homogeneous parts of the corresponding interpolation norms (in the space variable) and weighted Banach function norms (in the time variable); see Theorems \ref{thm main}(ii) and \ref{thm main 2}(ii), Corollary \ref{hom ext 1-p}, Propositions \ref{hom ext p} and \ref{hom ext infty}, as well as Remark \ref{force term}. Our assumptions on $f$ for deriving the homogeneous estimates \eqref{hom est1}, \eqref{hom est oo} and \eqref{E mr est} essentially relax those stated, e.g. in \cite[Proposition 2.14 and Theorem 2.20]{DaHiMuTo}.

In Section 5, we apply our abstract results to study the maximal regularity estimates for the negative Laplace operator in homogeneous Besov spaces. With the help of the main preparatory result, Proposition \ref{equivalent norm on besov}, we establish the maximal regularity of the negative homogeneous Laplacian on homogeneous Besov spaces in Corollary \ref{final application}(i). Additionally, part (ii) of this corollary identifies a \textit{limiting} space of force terms $f$, which, on the one hand, ensures the existence of strong solutions to the Cauchy problem associated with the Laplace operator on $L^p(\mathbb{R}^n)$, where $p \in [1, \infty)$, or $C_0(\mathbb{R}^n)$. On the other hand, compared to the results in Section 3, it provides additional control over the time behavior of the homogeneous norms of the solutions $u$ on $\mathbb{R}_+$.

\section{Preliminaries} \label{sect prelim}

In this section we provide auxiliary results for the proofs of our main results in the following sections. Some of these preparatory results may be of independent interest in the context of the general theory; see e.g. Lemma \ref{regularity of u}, Propositions \ref{weak L1}, \ref{L1 extrapol var 2}, \ref{represent for K}, \ref{KJ}, \ref{Haase equi}, \ref{equiv for l1}, and Theorem \ref{optimal for l1}.
The subsections 2.1-2.3 are devoted to the maximal regularity. The subsections 2.5-2.8 deal with interpolation spaces and their corresponding (semi)-norms. 
The particular context of these subsections is well-reflected by their corresponding titles.

\subsection{The maximal regularity}\label{sec mr def}

Consider the abstract Cauchy problem $(CP)_{A,f,x}$ associated with a closed, linear operator $A$ on a complex Banach space $X$.  
A~function $u\in W^{1,1}_{loc}([0,\tau);X)$ is called a {\it strong solution} to $(CP)_{A,f,x}$ if $u(0) = x$,  $u(t)\in D(A)$ and $u'(t) + A u(t) = f(t)$ for a.e. $t\in (0,\tau)$.

In this paper we also consider the notion of {\it maximal regularity on $I=(0,\tau)$} with respect to function spaces $E$ other than the Lebesgue $L^p$ spaces. The standard notion of $L^p$-maximal regularity is reformulated verbatim if $E$ represents, for instance,  Banach function space such as a weighted $L^p$ space or more general symmetric space (e.g. Lorentz or Orlicz spaces).

For completeness, recall that a Banach space $E = (E, \|\cdot\|_E)$ is called a {\it Banach function space over $\R_+ := (0,\infty)$} if $E$ is an order ideal in the space $L^0 := L^0(\R_+, \ud t)$ of all Lebesgue measurable, complex functions on $\R_+$, equipped with the topology of convergence in Lebesgue measure, and if the norm $\|\cdot\|_E$ is monotone. This means that if $|h| \leq |g|$ for $g \in E$ and $h \in L^0$, then $h \in E$ and $\|h\|_E \leq \|g\|_E$.

Suppose that $1 \leq p < \infty$ and that $v$ is a {\it weight function} on $\R_+$ ({\it weight} for short), that is, a positive, measurable function on $\R_+$.
We denote by $L^{p}_v$ (or $L^p(v)$) the (weighted) Lebesgue space $L^p(\R_+, v\ud t)$ with the standard norm. In the case $p = \infty$, we denote by $L_v^\infty$ (or $L^\infty(v)$) 
the Banach function space over $\R_+$ of all functions $f \in L^0$ 
such that $fv \in L^\infty(\R_+, \ud t)$, equipped with the norm
\[
\|f\|_{L_v^\infty} := \esssup_{t \in \R_+} |f(t)v(t)|.
\]

If $(X,\|\cdot\|_X)$ is a Banach space, then the $X$-valued variant of $E$ is defined to be the space $E(\R_+; X)$ ($E(X)$ for short) of all strongly measurable functions $f \colon \R_+ \to X$ such that $\|f(\cdot)\|_X \in E$. This is a~Banach space equipped with the norm
\[
\|f\|_{E(X)} := \| \|f(\cdot)\|_X \|_E.
\]
where, as usual, we identify the space of measurable functions with the space of equivalence classes of measurable functions.
For the interval $I=(0,\tau)$, $\tau \in (0,\infty)$, let $E(I;X)$ (or $E(0,\tau;X)$) denote the space of all functions from $I$ into $X$ such that their extensions by $0\in X$ to functions on $\R_+$ belong to $E(X)$. As usual, one can identify $E(I;X)$ with a subspace of $E(X)$.
We write  $L^1_{loc}$ to denote  $L^1_{loc}([0,\infty), \ud t)$, that is,  the space of all locally integrable, complex functions on $[0,\infty)$. For simplicity, we often write $\|f\|_X$ instead of $\|f(\cdot)\|_X$ below.

For a Banach function space $E$ which embeds continuously into $L^1_{loc}$ and $I=(0,\tau)$, $\tau\in (0,\infty]$,  we say that the operator $A$ has the {\it $E$-maximal regularity on $I$} (for short, $E$-m.r.\,on $I$), if for every $f\in E(I; X)$ there exists a unique {\it strong} solution $u$ to $(CP)_{A,f,0}$  such that $u', A u \in E(I; X)$. In this case, by the closedness of $A$, the following {\it $E$-maximal regularity estimate} holds:
\begin{equation}\label{m.r est}
\big\| u' \big\|_{E(I;X)} + \big\| A u \|_{E(I;X)} \leq C \big\| f \big\|_{E(I;X)} \tag{MRE}
\end{equation}
for some constant $C$ independent of $f\in E(I; X)$. We say that $A$ has '$E$-m.r.' if $A$ has $E$-m.r.\,on $(0,\tau)$ for some $\tau \in (0,\infty]$. 

Some caution is required concerning this notion in the case $I=\R_+$.  Note that we do not require that $u\in E(\R_+;X)$. 
In the case when $E = L^p$, $p\in [1,\infty]$, the definition stated above is consistent with the one from, e.g. Kalton and Portal \cite{KaPo08} or from the recent monograph on the subject \cite[Chapter 17]{HNVW24}, 
but differs to that considered, e.g. by Dore in \cite{Do00}, where $u \in W^{1,p}(\R_+;X)$ is assumed. This variant is more adequate in some situations, 
for instance, in the case when $A$ is not invertible, but the estimates on $\R_+$ of the type \eqref{m.r est} play a role; see e.g. \cite{DaHiMuTo, OgSh20} and the references therein. 
Examples of this arise naturally from differential operators on unbounded regions, such as the Laplace operator or the Stokes operator on exterior domains; cf. Giga and Sohr \cite{GiSa11}.

To clarify the novelty of our results and their proofs, we provide the following remark, which reformulates Dore's results on $L^p$-maximal regularity and adapts them to our framework.

\begin{remark}\label{Dore results}
As mentioned above, the notion of $L^p$-maximal regularity on $\R_+$ in Dore \cite{Do00} differs from the one that we make use of in this work. However, by Lebesgue’s integral representation of the strong solutions, it is readily seen that, in the case of finite intervals, both definitions coincide.

In particular, for all $\tau \in (0,\infty]$,  the $L^p$-m.r.\,on $(0,\tau)$ for a given operator $A$ (the notion we work with here), yields $u\in L^p_{loc}([0,\tau); X)$ for the corresponding solution $u$ of $(CP)_{A,f,x}$ with $f\in L^p(0,\tau;X)$. 
Furthermore, by \cite[Corollary 5.4]{Do00}, if $A$ has the $L^p$-m.r.\,on $(0,\tau)$ for some $\tau$, then it has $L^p$-m.r.\,on each finite interval, that is, for all $\tau\in (0,\infty)$. 

Consequently, for each $\tau$ and $p\in [1,\infty]$, by \cite[Corollary 4.4]{Do00}, the $L^p$-m.r. on $(0,\tau)$ implies that $-A$ is the generator of an analytic (but not necessarily strongly continuous at $0$) semigroup $T=(T(t))_{t> 0}$ on $X$. In addition, in the case when $\tau =\infty$, an analysis analysis of the proof of \cite[Theorem 4.3]{Do00}  yields that the generated semigroup is uniformly bounded, i.e. $\sup_{t> 0}\|T(t)\|_{\cL(X)}<\infty$; see also \cite[Theorem 17.2.15(2), p. 586]{HNVW24}.

Following \cite[p.\,128]{ABHN01}, for a Banach space $X$ and a strongly continuous function $T : (0,\infty) \to \cL(X)$, we say that $T$  is a {\it semigroup on $X$}, if $T(t+s) = T(t)T(s)$ for all $s,t > 0$, $\sup_{t\in (0,1)} \|T(t)\|_{\cL(X)} <\infty$, and $T(t)x = 0$ for all  $t>0$ implies $x=0$. 
From \cite[Proposition 3.2.4, p.\,126]{ABHN01} it follows that there exists an operator $A$ on $X$, such that $(-\infty, \omega) \subset \rho(A)$ for some $\omega \in \R$ and 
\[
(\la + A)^{-1} = \int_0^{\infty} e^{-\lambda t}T(t) \ud t, \qquad \lambda > \omega.
\]
The operator $-A$ is called a {\it generator} of $T$. For the generator $-A$ of a semigroup $T$ on $X$, we set for every $f\in L^1_{loc}([0,\infty);X)$
\begin{equation}\label{op U}
U f(t) := U_A f(t) := \int_0^t T(t-s) f(s) \ud s, \qquad t\geq 0. 
\end{equation}
We call $U$ the {\it solution operator} corresponding to  the operator $A$.  
The above-mentioned independence of the property of the $L^p$-m.r.\,on $(0, \tau)$ for  $\tau \in (0,\infty)$, shows that for any operator $A$ with $L^p$-m.r.\,on a finite interval, the corresponding solution operator $U$ is continuous on $L^p_{loc}([0,\infty);X)$ and, in addition, for every $f\in L^p_{loc}([0,\infty);X)$ 
\begin{equation}\label{loc est}
U f (t) \in D(A)\quad  \textrm{ for a.e. } t>0 \quad \textrm{and}\quad  A U f, \, U f \in L^p_{loc}([0,\infty);X).  
\end{equation}

In addition, according to Dore \cite[Theorem 4.1]{Do00}, if $U$ is bounded on $L^p(\R_+; X)$, then $0\in \rho(A)$, or, equivalently, the semigroup $T$ has the negative exponential growth bound. Conversely, if $A$ has $L^p$-m.r. on a finite interval and $0\in \rho(A)$, then $U$ is bounded on $L^p(\R_+; X)$. In particular, if we a priori know that $-A$ generates a bounded, analytic semigroup on $X$ and $0\in \rho(A)$, then the both definitions of the $L^p$-m.r. on $\R_+$ coincide; see \cite[Theorem 5.2]{Do00}.
\end{remark}

In addition, according to Dore \cite[Theorem 4.1]{Do00}, if $U$ is bounded on $L^p(\R_+; X)$, then $0 \in \rho(A)$, or, equivalently, the semigroup $T$ has the negative exponential growth bound. Conversely, if $A$ has $L^p$-m.r.\,on a finite interval and $0 \in \rho(A)$, then $U$ is bounded on $L^p(\R_+; X)$. In particular, if we a priori know that $-A$ generates a bounded, analytic semigroup on $X$ and $0 \in \rho(A)$, then both definitions of the $L^p$-m.r. on $\R_+$ coincide; see \cite[Theorem 5.2]{Do00}.

\subsection{The mild solutions}

For any densely defined generator $-A$ of a semigroup $T$ (not necessarily analytic), it is a standard fact of the semigroup theory (following from the infinitesimal characterization of $A$), that the regularity of the function $U f$ (see \eqref{op U}) is in a correspondence with the condition: $U f(t)\in D(A)$ and $AU f(t) = (U f)' (t) + f(t)$ for a.e. $t>0$; see e.g. Pazy \cite[Theorem 2.9, p.109]{Pa83}. 
In the case of non-densely defined $-A$ (which naturally arise in the context of the Da Prato-Grisvard result for the functor $(\, \cdot\,)_{\theta, \infty}$; see Lemma \ref{dense dom}) a corresponding statement can be found in \cite[Proposition 17.1.3, p.572]{HNVW24}. 

Let $\tau >0$, following \cite[p.130]{ABHN01}, we say that a function $u \in C([0,\tau];X)$ is a mild solution to $(CP)_{A,f,x}$, if $\int_0^t u(s) \ud s \in D(A)$ and 
\[
u(t) = A \int_0^t u(s) \ud s + x + \int_0^t f(s) \ud s, \qquad t \in [0,\tau].
\]  
In the case when $\tau = \infty$, the definition is analogous.

The following lemma is a refined version of \cite[Proposition 17.1.3, p.572]{HNVW24} and \cite[Theorem 2.9, p.109]{Pa83}. In particular, the statement (iii) of Lemma \ref{regularity of u}, which says that for any generator of a semigroup the function $v(t): =T(t)x + U f(t)$, $t>0$, is always a mild solution to $(CP)_{A,f,x}$ if $x\in \overline{D(A)}$, seems to be not stated in the literature on the semigroup theory, and may be of an independent interest. In particular, it extends the Da Prato-Sinestrari result, see \cite[Proposition 3.7.22, p.162]{ABHN01}, by showing that the assumption of the analyticity of $T$ is redundant. 
Furthermore, notice that assertion (ii) enables an equivalent reformulation of the definition of $L^p$-m.r., provided we already know that $-A$ is the generator of a semigroup.
\begin{lemma}\label{regularity of u}
Let $-A$ be the generator of a semigroup $T$ on a Banach space $X$. 
Let $f\in L^1_{loc}([0,\infty);X)$ and set 
\begin{equation}\label{funct u}
u (t) := U f(t) = \int_0^t T(t-s)f(s)\ud s, \qquad t>0.
\end{equation}
The following assertions hold.
\begin{itemize}
\item [(i)] If $u$ has the strong derivative at Lebesgue's point $r>0$ of the function $f$, then $u(r) \in D(A)$ and $Au(r) =  f(r) - u'(r)$. 

\item [(ii)] For every $r>0$,  $\int_0^r u(s) \ud s\in D(A)$ and  
\[
u(r) + A\int_0^r u(s) \ud s  = \int_0^r f(s) \ud s.
\]
In particular, if $u(t) \in D(A)$ for a.e. $t>0$ and $Au \in L^1_{loc}([0,\infty);X)$, then for each point $r$ which is Lebesgue's point of $f$ and $Au$, the function $u$ has a strong derivative at $r$ and $u'(r) = f(r) - Au(r)$.

\item [(iii)] For every $x\in X$ the function $v(t): = T(t)x + u(t)$, $t>0$, is continuous and bounded on each interval $(0,\tau)$, $\tau>0$, and for every  $r>0$,    
\[
\int_0^r v(s) \ud s\in D(A)\quad \textrm{and}\quad v(r) + A\int_0^r v(s) \ud s  = \int_0^r f(s) \ud s.
\] In particular, if $x \in  \overline{D(A)}$, then $v$ is a mild solution to $(CP)_{A,f,x}$, that is, in addition, $\lim_{t\rightarrow 0\plus}v(t) = x$. 
\end{itemize}
\end{lemma}

The proof of (i) makes use of the basic properties of the vector-valued Riemann-Stieltjes  integral; see, e.g. \cite[Chapter 1.9]{ABHN01}. It clearly shows how the definition of the generator $A$ via Laplace's transform still allows one to involve the continuity of the generated semigroup only on $(0,\infty)$. 

For the proof of (ii) one could follow the idea of that of (i), but details are more technical. We obtained it as a direct consequence of the (vector) Fubini theorem and the following auxiliary result, which also plays a crucial role in the proofs of some results presented in the next sections; see, e.g. Propositions \ref{Komastu rep}.

\begin{lemma}[{\cite[Proposition A.8.2, p.298]{Ha06}}]\label{lem on density} Let $-A$ be the generator of a semigroup $T$ on a Banach space $X$. Then, for every $x\in X$ and $t>0$ we have that 
\[
\int_0^t T(s)x\ud s \in D(A) \quad \textrm{ and } \quad A \int_0^t T(s)x\ud s = x - T(t)x.
\]
\end{lemma}

For any generator $-A$ of a $C_0$-semigroup the conclusion of Lemma \ref{lem on density} is the well-known fact from the semigroup theory, which simply follows  from the infinitesimal characterization of $-A$.

\begin{proof}[Proof of Lemma $\ref{regularity of u}$] 
The proof of the both statements is based on the following characterization of the elements from the graph of $A$: for each $x,y\in X$ and $\la \in \rho(A)$, the equality 
\begin{equation}\label{graph eq}
\la (\la + A)^{-1}x - x = (\la + A)^{-1} y
\end{equation}
is equivalent to $x$ belonging to $D(A)$ with $Ax = - y$.

Fix $\la > \omega$, where $\omega>0$ is such that $\sup_{t\geq 0} e^{-\omega t}\|T(t)\|_{\cL(X)}<\infty$. 
Then,  $(\la + A)^{-1} = \int_0^\infty e^{-t\la} T(t) \ud t$ with the convergence in the strong operator topology of $X$.

For (i), it is sufficient to show \eqref{graph eq} for $x$ and $y$ replaced with $u(r)$ and $u'(r) - f(r)$, respectively. Note that the integration by parts formula for the Riemann-Stieltjes integral we get that
\begin{align*}
\la (\la + A)^{-1}u(r) - u(r) 
& = \lim_{\epsilon \rightarrow 0\plus} \int_\epsilon^{1/\epsilon} \Bigl( \int_0^r T(t+r  -s)f(s) \ud s - u(r) \Bigr) \ud (-e^{-\la t})\\
& = \lim_{\epsilon\rightarrow 0\plus} e^{-\la/\epsilon} \Bigl(\int_0^r T(1/\epsilon + r -s)f(s) \ud s - u(r)  \Bigr)\\
& \qquad - \lim_{\epsilon\rightarrow 0\plus} e^{-\la\epsilon} \Bigl(\int_0^r T(\epsilon + r -s)f(s) \ud s - u(r)  \Bigr) \\
& \qquad + \lim_{\epsilon\rightarrow 0\plus} \int_\epsilon^{1/\epsilon} e^{-\la t} {\ud}\left(\int_0^r T(t+r-s)f(s) \ud s\right).
\end{align*} 
It is readily seen that the first two limits are equal to zero (with the help of Lebesgue's domination theorem for the second one). 

Therefore, by, e.g. \cite[Proposition 1.9.11, p.65]{ABHN01}, it is sufficient to show that the function $G(t): = \int_0^r T(t+r-s)f(s) \ud s$, $t>0$, is an antiderivative of $g(t):= T(t)u'(r) - T(t)f(r)$, $t>0$, which is continuous, and consequently Bochner integrable on finite intervals. Then, since the function $e^{-\lambda(\cdot)} g(\cdot)$ is Bochner integrable on $\R_+$, we get that
\[
\lim_{\epsilon\rightarrow 0\plus} \int_\epsilon^{1/\epsilon} e^{-\la t} g(t) {\ud} t = (\lambda + A)^{-1} (u'(r) - f(r)).
\]
Moreover, it is easily seen that by the continuity of $g$, if we show that the right derivative of $G$ exists and is equal to $g$ on $(0,\infty)$, then its (full) derivative does.
For, first note that for all $h>0$ we have
\begin{align*}
\int_0^{r} & T(t+r + h -s)f(s) \ud s - \int_0^{r}  T(t+r  -s)f(s) \ud s  \\
& = T(t) [u(r+h) - u(r)] - \int_r^{r+h}  T(t+r + h - s)f(s) \ud s\\
& = T(t) [u(r+h) - u(r)] - \int_r^{r+h}  T(t+r-s)f(s) \ud s\\
&\quad  - \int_r^{r+h}  [T(t+r + h - s)- T(t + r - s)] f(s)\ud s 
\end{align*} 
Since $\lim_{h\rightarrow 0}\frac{1}{h}\int_r^{r+h} \|f(s) - f(r)\|_X \ud s = 0$, by the strong continuity of $T$ on $(0,\infty)$, we easily get
\[
\lim_{h\rightarrow 0} \frac{1}{h}\int_r^{r+h}  T(t+r - s)f(s) \ud s = T(t)f(r)
\]
\[
\lim_{h\rightarrow 0} \frac{1}{h} \int_r^{r+h}  [T(t+r + h - s)- T(t + r - s)] f(s)\ud s  = 0. 
\]
Therefore, it gives that $G'_+(t) = g(t)$, and as was noted above, concludes the proof of (i).

For (ii), let $r>0$. By Fubini's theorem we get
\begin{align*}
\int_0^r U f(s) \ud s  & = \int_0^r \int_0^s T(s-\tau) f(\tau) \ud \tau \ud s\\
& = \int_0^r \int_0^{r-\tau} T(s) f(\tau) \ud s \ud \tau.
\end{align*}
Note that the function $\phi: [0,r]\ni \tau \mapsto \int_0^{r-\tau} T(s) f(\tau) \ud s \in X$ is in $L^1(0,r; X)$. Moreover, Lemma \ref{lem on density} gives that $\phi(\tau) \in D(A)$ with $A\phi(\tau) = f(\tau) - T(r-\tau)f(\tau)$, $\tau\in [0,r]$. Therefore, by Hille's theorem (see, e.g. \cite[Proposition 1.1.7]{ABHN01}) we get that $\int_0^r\phi(\tau) \ud \tau \in D(A)$ and 
\begin{align*}
A\int_0^r\phi(\tau) \ud \tau & = \int^r_0 A \phi(\tau) \ud \tau  = \int^r_0 ( f(\tau) - T(r-\tau)f(\tau)) \ud \tau
\end{align*}
It gives the first statement of (ii). 
 The second one is a direct consequence of Hille's theorem  and \cite[Proposition 1.2.2]{ABHN01}. 

 For (iii), by \cite[Proposition A.8.2, p.298]{Ha06}, for every $x\in X$ and $r>0$ we have that 
 \[
 \int_0^rT(s)x\ud s \in D(A)\qquad \textrm{and} \qquad A\int_0^rT(s)x\ud s = x - T(r)x.
 \]
Combining it with (ii) we easily see that the function $v$ has the desired properties stated in (iii); for the last statement therein see, for instance, \cite[Lemma 3.3.12, p.149]{ABHN01}. 
Therefore, the proof is complete.
\end{proof}

For further references we extract useful consequences in the following lemma.

\begin{lemma}\label{lem on hom est} 
Let $-A$ be the generator of an analytic semigroup $T$ on  a Banach space $X$. Then the following statements hold. 
\begin{itemize}
\item [(i)]  For every $f\in L^1_{loc}([0,\infty); X)$ and every $t>0$ such that $U f(t) \in D(A)$ we have that      
\begin{equation}\label{hom est in X}
\bigl\| A U f(t) \big\|_X \leq  \int_0^t \|AT(t-s) f(s)\|_X \ud s. 
\end{equation}

\item [(ii)] If $A$ has $L^1$-m.r., then for any $f\in L^1_{loc}([0,\infty); X)$ and for a.e. $t>0$ the function $(0,t)\ni s\mapsto AT(t-s)f(s)\in X$ is Bochner integrable. 
  
\item [(iii)] Let $\tau>0$. If for every $f\in L^1(0,\tau; X)$ and for a.e. $t\in (0,\tau)$ the function $(0,t)\ni s\mapsto AT(t-s)f(s)\in X$ is Bochner integrable on $(0,t)$ with $\|\int_0^t AT(t-s)f(s) \ud s\|_{L^1(0,\tau;X)} < \infty$, then $A$ has $L^1$-m.r.\,on $(0,\tau)$ $($on each finite interval$)$. 
\end{itemize}
\end{lemma}

\begin{proof}
(i) For all $t>0$, the function $(0,t)\ni s\mapsto T(t-s)f(s) \in X$ is Bochner integrable and $(0,t)\ni s\mapsto AT(t-s)f(s) \in X$ is strongly measurable (see, e.g. \cite[Proposition 1.3.4]{ABHN01}). 
By Bochner's theorem, for $t>0$ such that $(0,t)\ni s\mapsto \|AT(t-s)f(s)\|_X$ is Lebesgue integrable, we have 
\[
\Big\| \int_0^t AT(t-s) f(s) \ud s \Big\|_X \leq \int_0^t \|AT(t-s) f(s)\|_X \ud s.
\]
Moreover, by Hille's theorem \cite[Proposition 1.1.7]{ABHN01}, we get for such $t$ that
\[
A\int_0^t T(t-s) f(s) \ud s = \int_0^t AT(t-s) f(s) \ud s. 
\]
It gives \eqref{hom est in X} in such case. Otherwise, that is, for $t>0$ such that $(0,t)\ni s\mapsto \|AT(t-s)f(s)\|_X$ is not Lebesgue integrable, i.e. $\int_0^t \|AT(t-s) f(s)\|_X \ud s = \infty$, if we know that $U f(t) \in D(A)$, or in the other words,  $AU f(t)$ is identified as a vector in $X$, then \eqref{hom est in X} holds trivially.

(ii) By Remark \ref{Dore results}(b), for $\lambda>0$ large enough the operator $A+\lambda$ has $L^1$-m.r. on $\R_+$. The Kalton-Portal characterization (see \cite[Theorem 3.6]{KaPo08}, or Theorem \ref{weighted L1 mr} with $v\equiv 1$ below) yields the existence of a constant $C_\lambda$, which may depend on $\lambda$, such that
\[
\int^\infty_0 \|(A+\lambda) e^{-\lambda s} T(s)x\|_X  \leq C_\lambda \|x\|_X , \qquad x\in \X.
\]
 In particular,  for every $\tau>0$ and $x\in \X$, we easily get 
\begin{align}\label{KP for tau}
\nonumber \int^\tau_0 \|AT(s)x\|_X \ud s & \leq e^\tau\lambda\int^\tau_0 \|(A+\lambda) e^{-\lambda s} T(s)x\|_X \ud s + \int^\tau_0 \|\lambda T(s) x\|_X\ud s\\
\nonumber & \leq e^{\tau \lambda} \int^\infty_0 \|(A+\lambda) e^{-\lambda s} T(s)x\|_X \ud s  
+ \tau \lambda \sup_{0\leq s\leq \tau}\|T(s)\|_{\cL(X)}\,\|x\|_X\\
&\leq (C_\lambda e^{\tau \lambda} + \tau \lambda \sup_{0\leq s\leq \tau}\|T(s)\|_{\cL(X)}) \,\|x\|_X.
\end{align} 
Further, since for every $f\in L^1_{loc}([0,\infty);X)$ and $\tau >0$, the function $(0, \tau)^2\ni (t,s) \mapsto \|AT(t-s)f(s)\|_X \chi_{\{0< s\leq t \leq \tau\}}$ is measurable, by the scalar Fubini theorem we have that   
\begin{align*}
\int_0^\tau \int_0^t \|AT(t-s) f(s)\|_{X} \ud s\ud t & = \int_0^\tau\int_s^\tau \|AT(t-s) f(s)\|_{X} \ud t \ud s \\
& = \int_0^\tau \int_0^{\tau -s} \|AT(t) f(s)\|_{X} \ud t \ud s \\
& \leq C \int_0^\tau  \|f(s)\|_{X} \ud s<\infty
\end{align*} for $C$ specified in \eqref{KP for tau}. Since $\tau>0$ is arbitrary here, we get the desired claim.

For (iii), Lemma \ref{regularity of u} shows that $U f$ is in $W^{1,1}(0,\tau; X)$ and $(U f)'(t) + AU f(t) = f(t)$ for a.e. $t\in (0,\tau)$. The closed graph theorem applied to the operator $AU$ gives the estimate \eqref{m.r est}, therefore $A$ has $L^1$-m.r.\,on $(0,\tau)$.  
\end{proof}

\begin{remark}\label{rem on density} Recall that for the generator $-A$ of an analytic semigroup $T$ on $X$, $T(t)X\subset D(A)$ for all $t>0$ and  for all $x\in D(A)$ the function $\R_+ \ni t \mapsto AT(t)x = T(t)Ax \in X$ is in $L^1_{loc}([0,\infty);X)$.

However,  by the Kalton-Portal characterisation of $L^1$-m.r. (on a finite interval) and the closed graph theorem, if $A$ does not have $L^1$-m.r., then there exists $x\in X$ such that the function $AT(\cdot)x$ is not Bochner integrable on $[0,t]$ for any $t >0$, more precisely,  $\int_0^t \|AT(s)x\|_X \ud s = \infty$ for all $t> 0$. 
Therefore, for some vectors $x\notin D(A)$ and all $t > 0$ the symbol $\int_0^tAT(t-s)f(s) \ud s$ for $f(s) = x$ may not have any meaning (as a Bochner integral). On the other hand, by Lemma \ref{lem on density}, for any $x\in X$ if $f(s) = x$ $(s>0)$, then  $A\int_0^t  T(t-s)f(s) \ud s = A\int_0^t T(s)x \ud s$ is well defined (cf. Lemma \ref{lem on hom est}(iii)).

\end{remark}

Combining Lemma \ref{lem on hom est}(i) with Remark \ref{Dore results}(b), if $A$ has $L^p$-m.r. for some $p\in [1,\infty]$, then for any $f\in L^p_{loc}([0,\infty); X)$ the inequality \eqref{hom est in X} holds for a.e. $t>0$.
 The following result, which can also be considered as an extension of Dore \cite[Theorem 7.1]{Do00} and Hyt\"onen \cite[Theorem 1.4]{Hy05}, in particular, shows that it holds for every $f\in L^1_{loc}([0,\infty); X)$. 
Below, for a weight function $v$ and a measurable subset $B$ of $\R_+$ we set $v(B): = \int_B v(t) \ud t$.
\begin{proposition}\label{weak L1} Let $p\in [1,\infty]$ and let $A$ be a linear operator on a Banach space $X$ with $L^p$-m.r. on a finite interval.

 Then, for every Muckenhoupt weight $v\in A_1(\R)$ and every $\tau>0$, the operator $AU$ is continuous operator from $L^1_v(0,\tau ;X)$ to $L^{1,\infty}_v(0,\tau; X)$, i.e. for every $f\in L^1_{loc}([0,\infty), v\ud t; X)$, $U f(t) \in D(A)$   for a.e. $t>0$ and  there exists a constant $C_\tau$ independent on $f$ such that for every $\la > 0$
\[
v(\{ t\in (0,\tau): \|AU f(t)\|_X \geq \la \}) \leq C_\tau \lambda^{-1} \|f\|_{L^1_v(0,\tau; X)}.
\]
\end{proposition}

\begin{proof} By Dore \cite[Corollary 4.4]{Do00}, $-A$ is the generator of an analytic semigroup $T$ on $X$. 
Suppose first, that the exponential growth bound of $T$ is negative. Then, by Dore \cite[Theorem 5.2]{Do00}, the operator $AU$ is bounded on $L^p(\R_+;X)$. 

Naturally identifying $L^p(\R_+;X)$ with a subspace of $L^p(\R;X)$, $AU$ can be consider as an operator $\cS$ on $L^p(\R;X)$, i.e. $\cS f: = e AU (rf)$, $f\in  L^p(\R;X)$, where $e$ stands for the extension by $0\in X$ from $\R_+$ onto $\R$ and $r$ for the restriction from $\R$ to $\R_+$.  
Of course, $\cS$ is bounded on $L^p(\R;X)$ and for every $f: \R\rightarrow X$ which is bounded with the compact support $\supp f$ and for every $t\notin \supp f$ we have 
\[
\cS f(t) = \int_\R \cK(t-s) f(s) \ud s,
\]
where $\cK(t):= A T(t)$ if $t>0$ and $\cK(t):=0 \in \cL(X)$ otherwise.
Since $T$ is a bounded, analytic semigroup on $X$, $\cK$ satisfies the Calder\'{o}n-Zygmund conditions. Therefore, by weighted variant of the result  by Benedek, Calder\'on, Panzone \cite{BeCaPa62}, see  \cite[Chapter V, Theorem 3.4]{GCRu85}, $\cS$ satisfies the weak $(1,1)$ estimates on $L^p(\R;X)\cap L^1_v(\R;X)$. In particular, there exists a constant $C$ such that for every $f\in L^p(\R_+;X)\cap  L^1_v(\R_+;X)$ and all $\la>0$
\begin{equation}\label{wL1}
v\left(\{t>0 : \|A U f(t)\|_X > \la\}\right) \leq C \lambda^{-1} \|f\|_{L^1_v(\R_+;X)}.
\end{equation}
First, we apply it to show that for all $f\in L^1_{loc}([0,\infty), v\ud t;X)$, $U f(t) \in D(A)$ for a.e. $t>0$. 
Suppose by contradiction that there exists $N\in \N$ such that 
$B:= \{ t\in [0,N]:  U f(t)\notin D(A) \}$ has the positive Lebesgue's measure. 
Note that $U (f\chi_{[0,N]}) = U f$ on $[0,N]$. 
 Let $f_n$, $n\in \N$, be a sequence of functions with compact supports from $L^\infty(\R_+;X)\cap  L^1_v(\R_+;X)$ which approximate $f \chi_{[0,N]}$ in $L^1_v(\R_+;X)$. By \eqref{wL1}, for every $n, m \in \N$ and $\lambda>0$ we have that
 \[
v\left(\{t>0 : \|A U f_n(t) - A U f_m(t)\|_X > \la\}\right) \leq C \lambda^{-1} \|f_n - f_m\|_{L^1_v(\R_+;X)}.
 \]
It gives an increasing sequence $(N_k)_{k\in \N}\subset \N$ such that for all $k\in \N$:
\[
v\left(B_k\right) \geq v([0,N]) - 2^{-k-1}v(B),
\]
where
\[
B_k := \{ t\in [0,N] : \|A U f_{N_{k+1}}(t) - A U f_{N_k}(t)\|_X \leq  2^{-k} \}.
\] In particular, for every $t\in 
\bigcap_{k\in \N} B_k$: 
\[
\sum_{k\in \N} \|A U f_{N_{k+1}}(t) - A U f_{N_k}(t)\|_X \leq 1.
\] Consequently, the sequence $A U f_{N_k}(t)$, $k\in \N$,  is convergent as $k\rightarrow \infty$.  
 Since $L^1_v\subset L^1_{loc}$, the sequence $U f_n(t)$, $n \in \N$,  is convergent to $U f(t)$ for every $t>0$ as $n\rightarrow \infty$. Therefore, the closedness of $A$ shows that $U f(t) \in D(A)$ for all $t\in \bigcap_{k\in \N} B_k$. But $v(\bigcap_{k\in \N} B_k) \geq v([0,N]) - v(B)/2$, so we get a contradiction. 
Therefore, for all $f\in L^1_{loc}([0,\infty), v\ud t;X)$ the function $A U f$ is well-defined and, in addition, there exists a sequence of bounded, compactly supported functions $f_n$, $n\in \N$, such that $f_n$ converges to $f$ in  $L^1_{loc}([0,\infty), v\ud t;X)$, and  $A U f_n$ converges to $A U f$ a.e. on $\R_+$. 
In particular, for each $\tau>0$ and $\epsilon>0$ we get
\[
v\left(\{t\in (0,\tau) : \|A U f_n(t) - A U f(t)\|_X > \epsilon\}\right) \rightarrow 0 \textrm{ as } n \rightarrow \infty.
\]
Combining it with \eqref{wL1} for $f_n\chi_{[0,N]}$, $n\in \N$, we easily get 
\[
v\left(\{t\in (0,\tau) : \|A U f(t)\|_X > \la\}\right) \leq  C \lambda^{-1} \|f\|_{L^1_v(0,\tau;X)}.
\]
Additionally, it shows that in the case when $f\in L^1_v(\R_+; X)$ and $(f_n)$ is convergent to $f$ in $L^1_v(\R_+; X)$, then \eqref{wL1} holds for $f$ too.

In general case, i.e. when $T$ does not have the negative exponential growth bound, the above argument works for $A+\la$ with $\la>0$ large enough, which generates the semigroup $(e^{-\la t}T(t))_{t > 0}$. Since $f\in L^1_{loc}([0,\infty), v\ud t;X)$ if and only if $e^{-\la ( \cdot)} f\in L^1_{loc}([0,\infty), v \ud t;X)$, and $e^{t\la} U_{A+\la} (e^{- \la (\cdot)} f)(t) = e^{-\la t} U_A f(t)$ for all $t>0$, the proof is complete.
\end{proof}

\begin{remark}\label{D-Y-S}  Note that, in the conclusion of Proposition \ref{weak L1}, we do not claim that for $f\in L^1_{loc}([0,\infty);X)$ the corresponding mild solution $U f$ is strongly differentiable almost everywhere on $\R_+$.  

If, in addition to the assumptions stated there, we know that $A$ is densely defined, then for every $f\in L^1_{loc}([0,\infty);X)$, $U f$ has strong right-side derivative a.e. on $\R_+$. 
Indeed, it follows directly from the infinitesimal characterization of densely defined generators of semigroups, which gives that $U f(t)\in D(A)$ implies  
$
\frac{1}{h}(T(h)U f(t) - U f(t)) \rightarrow - AU f(t)$ as $h\rightarrow 0\plus$.
Moreover, for $h, t>0$,
\[ 
\frac{1}{h}\bigl(U f(t+h) - U f(t)\bigr) = \frac{1}{h}\bigl( T(h)U f(t) - U f(t)\bigr) + \frac{1}{h} \int_{t}^{t+h} T(t+h-s) f(s) \ud s. 
\]
By the strong continuity of $T$ at $0$, for each Lebesgue's point $t$ of $f$ we get 
$(U f)'_+(t) = - AU f(t) + f(t)$. 
We do not know if the {\it full} derivative of $U f$ exists a.e. on $\R_+$.  

\end{remark}

\subsection{The characterization of weighted $L^1$-m.r.}

The following result provides a weighted counterpart  of the Kalton-Portal characterization of $L^1$-m.r. on $\R_+$; see \cite[Theorem 3.6]{KaPo08}.

\begin{theorem}\label{weighted L1 mr}
Let $-A$ be the generator of a bounded analytic semigroup $T$ on $X$. Let $v$ be a locally integrable weight function on $(0,\infty)$ and such that $L^1_v \subset L^1_{loc}$. 
Then, the following assertions are equivalent. 
\begin{itemize}
\item [(i)] The operator $A$ has $L^1_v$-m.r. on $\R_+$, that is,  for every $f\in L^1_v(\R_+; X)$ there exists a function $u\in W^{1,1}_{loc} ([0,\infty);X)$ such that $u(0) = 0$, $u(t)\in D(A)$ and $u'(t) + Au(t) = f(t)$ for a.e. $t>0$, and 
\[
\|u'\|_{L^1_v(\R_+;X)}+ \|Au\|_{L^1_v(\R_+;X)} \leq C \|f\|_{L^1_v(\R_+;X)} 
\]
for some constant $C$ independent on $f$.
\item [(ii)] There exists a constant $C$ such that for every Lebesgue's point $s\in (0,\infty)$ of $v$ and every $x\in X$ we have that 
\begin{equation}\label{L1v mr cond}
\int_0^\infty \|AT(t)x\|_X v(t+s) \ud t \leq C \|x\|_X v(s).
\end{equation}
\end{itemize} 
\end{theorem}

A direct consequence of Theorem \ref{weighted L1 mr} is the following extrapolation result; cf. Proposition \ref{weak L1}. 

\begin{corollary}\label{L1 for decr}  {\rm(i)} The $L^1$-m.r. implies $L^1_v$-m.r. for any non-increasing weight $v$ on $(0,\infty)$. 

{\rm(ii)} Suppose that $-A$ is the generator of an analytic semigroup $T$ on a Banach space $X$. Let $v$ be a non-decreasing weight on $(0, \infty)$ with $\lim_{t\rightarrow 0\plus} v(t) = 0$. Then, if \eqref{L1v mr cond} holds, then $Ax = 0$ for all $x\in D(A)$.  
\end{corollary}
\begin{proof} The first assertion 
(i) follows easily via the characterization condition of Theorem \ref{weighted L1 mr}. For (ii), by the standard rescaling argument one can assume that $T$ exponentially decreases. By \eqref{L1v mr cond} and the monotone convergence theorem we get
\[
\int_0^{\infty} \|AT(t)x\|_X v(t) \ud t \leq 0.
\]
The assumption on $v$ implies for every $\tau>0$  
\[
0= \int_0^{\infty} \|AT(t)x\|_X v(t) \ud t \geq v(\tau)  \int_{\tau}^{\infty} \|AT(t)x\|_X \ud t.
\]
Hence, for all $x \in X$ and $t > 0$ we get $AT(t)x = 0$. In particular, for $x\in D(A)$, $R(\lambda, A) Ax = 0$, or equivalently, $R(\lambda, A)x =  \frac{1}{\lambda} x$, for all $\lambda > \omega (T)$, where $\omega(T)$ stands for the exponential growth bound of $T$. Therefore, by the uniqueness theorem for the Laplace transform (see, e.g. \cite[Theorem 1.7.3, p.51]{ABHN01}), $T(t)x = x$ for every $t>0$ and $x\in D(A)$. It completes the proof of the corollary.
\end{proof}

\begin{proof}[Proof of Theorem $\ref{weighted L1 mr}$]
For (i)$\Rightarrow$(ii), let $s>0$ be a Lebesgue point of $v$. For any $\epsilon\in (0,1)$ we set 
\[
u_{\epsilon,s}(t) := \left\{\begin{array}{ll} 
0, &  t\in [0, s],\\
\min(\frac{t-s}{\epsilon},1) T(t-s)x, & t>s.
\end{array} \right.
\]
Note that $u_\epsilon(t) \in D(A)$ for all $t\geq 0$ and  $u_\epsilon \in W^{1,1}_{loc} ([0,\infty);X)$ with 
\[
u_{\epsilon,s}'(t) = \left\{\begin{array}{lll}
0, & t\in [0,s],\\
- \min(\frac{t-s}{\epsilon},1) AT(t-s) x  + \frac{1}{\epsilon}T(t-s)x, & t\in (s,s+\epsilon],\\
-AT(t-s)x, & t>s+\epsilon.
\end{array}\right.
\]
Therefore, $u_{\epsilon,s}$ is a solution of $(CP)_{A,f,0}$ with 
$f_{\epsilon,s}(t):= \frac{1}{\epsilon} T(t-s)x \chi_{[s,s+\epsilon]}(t)$, $t>0$. 
Note that 
\begin{equation}\label{bound on 01}
\int_{\R_+} \|f_\epsilon(t)\|_X v(t) \ud t = \sup_{t\in (0,1]}\|T(t)\|_{\cL(X)} \|x\|_X \frac{1}{\epsilon} \int_{s}^{s+\epsilon} v(t) \ud t.
\end{equation}
On the other side we get
\begin{align*}
\int_{\R_+} \|Au_\epsilon (t)\|_X v(t) \ud t  & \geq \int_{s+\epsilon}^\infty  \|AT(t-s) x\|_X v(t) \ud t = \\
& =  \int_{\epsilon}^\infty  \|AT(t) x\|_X v(t+s) \ud t.
\end{align*}
Therefore, letting with $\epsilon \rightarrow 0\plus$ we get \eqref{L1v mr cond}. It completes the proof of this implication.

For (ii)$\Rightarrow$(i), let $f\in L^1_v(\R_+; X) \subset L^1_{loc}([0,\infty);X)$. 
Note that the function $(0, \infty)^2 \ni (t,s) \mapsto \chi_{\{s<t\}}(t,s) \|AT(t-s) f(s)\|_X$ is measurable, therefore combining Fubini's theorem with \eqref{L1v mr cond}, we easily get
\begin{align*}
\int_{\R_+} \int_0^t\|AT(t-s)f(s)\|_X \ud s v(t) \ud t  & = 
\int_{\R_+} \int_s^\infty \|AT(t-s)f(s)\|_X v(t) \ud t \ud s\\
& = \int_{\R_+} \int_{\R_+} \|AT(t)f(s)\|_X v(t+s) \ud t \ud s \\
& \leq C  \int_{\R_+} \|f(s)\|_X v(s) \ud s  
\end{align*}
In particular, the function $(0,t)\ni s\mapsto AT(t-s)f(s)\in X$ is Bochner integrable for a.e. $t>0$. Consequently, by Hille's theorem 
$ U f(t) = \int_0^t T(t-s)f(s) \ud s \in D(A)$, and Bochner's theorem gives 
\[
\int_{\R_+} \left\|A  U f(t) \right\|_X v(t) \ud t \leq C \int_{\R_+} \|f(t)\|_X v(t) \ud t. 
\]
The remaining part of the statement (i) is a standard consequence of our assumption that $L^1_v\subset L^1_{loc}$, the above estimate and Lemma \ref{regularity of u}(ii). 
\end{proof}

\begin{remark} Note that in \eqref{bound on 01} we use only the boundedness of $T$ on $[0,1]$, but we apply the maximal regularity estimate in (i) for a family of functions supported on $[s,s+\epsilon]$, where $s$ ranges over a.a. points of $\R_+$. Such {\it uniform} estimate requires the boundedness of $T$ on $\R_+$; see Remark \ref{Dore results}(a). 
\end{remark}

We conclude this subsection with a complementary result to Proposition \ref{weak L1} and Corollary \ref{L1 for decr}, which completes also \cite[Proposition 17.2.36]{HNVW24}. More precisely, the statement (1) of \cite[Proposition 17.2.36]{HNVW24} says that for $p\in (1,\infty)$ (resp. $p=\infty$), the $L^p_v$-m.r.\,for the power weights $v(t)=t^\mu$, $\mu \in (- 1, p-1)$ (resp. $\mu \in (0,1)$) is equivalent to the (unweighted) $L^p$-m.r.  The case $p=1$ has been left open therein.  Hereby, with the help of Theorem \ref{weighted L1 mr}, we provide the affirmative answer to it. In fact, our result identifies a more general class of weights $v$ for which (weighted) $L_v^1$-m.r. is equivalent to (unweighted) $L^1$-m.r. More importantly, this result should be read in the context of a negative example presented in Theorem \ref{thm main}(iii). Roughly, it says that, in the absence of the $L^1$-m.r. estimates, the weighted $L^1_v$-m.r. estimates do not hold for a large class of weights $v$ as well. 
We start with an auxiliary result.
 
\begin{lemma}\label{log(v) convex}
 Let $v\colon  (0,\infty) \to (0,\infty)$ be such that $\log\circ v$ is a convex function. Then,  for every $t > 0$ the function $\frac{v(t+\cdot)}{v(\cdot)}$ is increasing.
\end{lemma}
\begin{proof}
Let $0 < r < s$ and $t > 0$. Since $\log\circ v$ is convex, we get 
\begin{align*}
\log(v(r+t)) + \log(v(s))&  \leq \frac{s-r}{y-s+t}\log(v(r)) + \frac{t}{s-r+t}\log(v(s+t)) \\ &+ \frac{t}{s-r+t}\log(v(r)) + \frac{s-r}{s-r+t}\log(v(s+t))\\ & = \log(v(r)) + \log(v(s+t)).
\end{align*}
Therefore, by the monotonicity of an exponential function we get
\[
v(r+t)\cdot v(s) \leq v(r)\cdot v(s+t).
\]
and this yields
\[
\frac{v(t+r)}{v(r)} \leq \frac{v(t+s)}{v(s)}.
\]
This finishes the proof of the lemma.
\end{proof}

\begin{proposition}\label{L1 extrapol var 2}
Let $v$ be a non-increasing function on $(0,\infty)$ such that $\log\circ v$ is convex and for some  $\epsilon > 0$ and for a.e. $t > 0$ 
\begin{equation}\label{cond on v}
\lim_{s \to \infty} \frac{v(t+s)}{v(s)} \geq \epsilon.
\end{equation} 
Let $-A$ be a generator of a bounded analytic semigroup on $X$.
 Then, $A$ has $L^1_v$-m.r. on $\R_{+}$ if and only if it has $L^1$-m.r. on $\R_{+}$. In particular, the statement holds for $v(t) = t^{-\theta}$ with $\theta \in (0,\infty)$.
\end{proposition}

\begin{proof}
Since $v$ is non-increasing, by Corollary \ref{L1 for decr}, the $L^1$-m.r.\,on $\R_+$ of $A$ gives its $L^1_v$-m.r.\,on $\R_+$. 
Suppose that $A$ has  $L^1_v$-m.r.\,on $\R_+$. By the condition \eqref{L1v mr cond}, there exist a constant $C>0$ such that for all $x \in X$ and $s > 0$
\[
\int_0^\infty \|AT(t)x\|_X \frac{v(t+s)}{v(s)} \ud t \leq C \|x\|_X.
\]
By the Lemma \ref{log(v) convex} we get that for all $t>0$, $\lim_{s \to \infty} \frac{v(t+s)}{v(s)} \in (0, \infty]$ exists, hence the function $\R_+\ni t\mapsto \lim_{s \to \infty} \frac{v(t+s)}{v(s)}$ is measurable and Fatou's lemma gives
\[
\int_0^\infty \|AT(t)x\|_X \ud t \leq \frac{1}{\epsilon} \int_0^\infty \|AT(t)x\|_X \lim_{s \to \infty}\frac{v(t+s)}{v(s)} \ud t \leq \frac{C}{\epsilon} \|x\|_X.
\]
Therefore, by the Kalton-Portal characterization, $A$ has $L^1$-m.r.\,on $\R_+$. 
\end{proof}

\begin{remark} 
(a) In the Proposition \ref{L1 extrapol var 2} the assumption that $-A$ generates a bounded analytic semigroup on $X$ is not necessary in the case $\theta \in (0,1)$, since it follows from $L^1_v$-m.r. of $A$, as it was noted in the proof of \cite[Proposition 17.2.36]{HNVW24}.

It is worth mentioning that in Theorem \ref{thm main}(vi) below, we show that the condition \eqref{cond on v} is, in a sense, sharp for establishing the equivalence between 
(unweighted) $L^1$-maximal regularity and (weighted) $L^1_v$-maximal regularity.

(b) In Theorem \ref{thm main}(vi) below, we show that the condition \eqref{cond on v} is, in a sense, sharp  for establishing the equivalence between (unweighted) $L^1$-maximal regularity and (weighted) $L^1_v$-maximal regularity. 
\end{remark}

For the further references we conclude this subsection with the characterization of the $L^\infty$-m.r.\,on $\R_+$ due to Kalton and Portal \cite{KaPo08}. 

\begin{lemma}[{\cite[Theorem 3.7]{KaPo08}}]\label{charact}
Let $-A$ be the generator of a bounded analytic $C_0$-semigroup $T$ on a Banach space $X$. 
Then, $A$ has $L^\infty$-m.r. on $\R_+$ if and only if  there exists a constant $C>0$ such that for every $x\in X$ 
\begin{equation}\label{L infty cond}
\|x\|_X \leq C \sup_{t>0} \|tA T(t) x\|_X + \limsup_{t\rightarrow \infty}\|T(t)x\|_{X}. 
\end{equation}
\end{lemma}
We point out that it can be shown that the condition \eqref{L infty cond} is sufficient for $L^{\infty}$-m.r. on $\mathbb{R}_+$ 
even for non-densely defined generators of bounded analytic semigroups.

\medskip

\subsection{The K-method}
Let $(X,Y)$ be an interpolation couple of Banach spaces and let $\Phi$ be a Banach function space.
Recall first the definition of the {\em $K$-functional}, that is, for every $x\in X+Y$ and $t>0$ put
\[
K(t,x):=K(t,x, X,Y):= \inf\left\{ \|a\|_X+t\|b\|_Y: a\in X, b\in Y, a+b=x \right\}.
\] 
With the help of the $K$-functional, we define the space
\[
K_\Phi(X,Y):= \big\{ x \in X+Y: \; [(0,\infty)\ni t \mapsto t^{-1} K(t,x) ]\in \Phi 
\big\},
\]
and equip it with the norm
\[
\|x\|_{K_\Phi}:=\|x\|_{K_\Phi(X,Y)}:=\left\| (\cdot)^{-1}K(\cdot, x) \right\|_\Phi\,, \quad \quad x\in {K_\Phi(X,Y)}.
\]
This definition of the functor $K_\Phi$ differs very slightly from the one in \cite[Section 3.3]{BrKr91}, where $\Phi$ runs through Banach function spaces containing the function $\min(1,\cdot)$, and $x\in K_\Phi(X,Y)$ if and only if $K(\cdot, x)\in \Phi$. The point is only the factor $\frac{1}{t}$ between the two definitions. 
For instance, by \cite[Lemma 2.2]{ChKr17}, $K_\Phi$ is an interpolation functor if and only if $\min(1,(\cdot)^{-1})\in \Phi$, in the other case, $K_\Phi(X,Y) = \{0\}$.

It is clear that for any $q\in [1,\infty)$ and $\theta\in (0,1)$, if $\Phi := L^q(\R_+, t^{q(1-\theta) - 1} \, \ud t)$, then $K_\Phi(X,Y)$ coincides with the classical interpolation space $(X,Y)_{\theta,q}$. For basic notation regarding interpolation theory and fundamental results, we refer the reader, e.g. to \cite{BrKr91, Lu09, Tr95}.

Below we consider mainly two special classes of Banach function space $\Phi$, which correspond to the $L^1$-m.r. and $L^\infty$-m.r., that is, weighted Lebesgue spaces $L^1_w$ and $L^\infty_w$, respectively. For instance, for $\Phi = L^\infty_w$ the  corresponding space $K_\Phi(X,Y)$ can be considered as a generalized Marcinkiewicz space. More precisely, note that for the couple $(L^1, L^\infty)$, the space $K_\Phi(L^1, L^\infty)$ is the Marcinkiewicz functional space $M_w$ given by the norm 
\[
\|f\|_{K_\Phi(L^1, L^\infty)}  \simeq \sup_{t>0} w(t) {f^{**}(t)}, 
\]
that is, $K(t,f) = tf^{**}(t)= \int_0^tf^*(s) \ud s$.  For $(L^p, W^k_p)$ we have 
\[
K(t,f)  \simeq \omega_{k} (f, t^{1/k})_{L^p}, 
\]
where $\omega_k (f,s)_{L^p} := \sup_{|h|\leq s}\|\delta^k_h f\|_{L^p}$.

The following result is complementary to our main result Theorems \ref{thm main} and \ref{thm main 2}, but may be of independent interest.  
Roughly, it says that interpolation parameters $L^1_w$ and $L^\infty_w$ for general weights $w$ are in fact equivalent to some parameters which are defined in terms of appropriately chosen concave functions. We start with the case of $L^\infty_w$, which is more transparent than that one of $L^1_w$. 

Recall that a function $\phi: \R_+\rightarrow \R_+$ is called \emph{quasi-concave}, if $\phi$ is non-decreasing and 
\[
\phi(t) \leq \max (1, t/s) \phi(s) \quad \textrm{ for all } s,t>0.
\]
Recall that for any quasi-concave function $\phi$ its least concave majorant $\tilde{\phi}$ satisfies $\phi\leq \tilde{\phi} \leq 2 \phi$; see e.g. \cite[Corollary 3.1.4]{BrKr91}.

\begin{proposition}\label{represent for K}  Let $(X,Y)$ be a pair of Banach spaces such that $Y$ is continuously embedded in $X$. Let $w$ be a weight function on $\R_+$. Then, the following assertions hold.
\begin{itemize}
\item [(i)] There exists a quasi-concave function $\phi$ on $\R_+$ such that $K_{L^\infty(w)}(X,Y) = K_{L^\infty(t/\phi(t))}(X,Y)$ with equal norms, i.e. for every $x\in  K_{L^\infty_w}(X,Y)$
\[
\|x\|_{ K_{L^\infty_w}(X,Y)} = \sup_{t>0} \frac{K(t,x,X,Y)}{\phi(t)}. 
\]
In particular, $\|\cdot\|_{ K_{L^\infty_w}(X,Y)}$ is equivalent to the norm  
$\sup_{t>0} \frac{K(t, \cdot)}{\hat\phi(t)}$, where $\hat \phi$ denotes the least concave majorant of $\phi$.  

\item [(ii)] 
There exists a concave function $\phi$ on $\R_+$ such that 
\[
K_{L^1_w}(X,Y) = K_{L^1(\R_+,\ud (-\phi'))}(X,Y)
\]
with equivalent norms, i.e. for every $x\in  K_{L^1_w}(X,Y)$
\[
\|x\|_{ K_{L^1_w}(X,Y)}  \simeq \int_0^\infty  K(t,x,X,Y) \ud (-\phi').
\]
Here, the integral is the Riemann-Stieltjes integral.
\end{itemize}

\end{proposition}

\begin{proof} (i)  We assume that $\min(1, (\cdot)^{-1})\in L^\infty_w$, since in the other case $K_{L^\infty_w}(X,Y) = \{0\}$  and the proof is trivial. 

Let $x\in K_{L^\infty_w}(X,Y)$ and $x\neq 0$, i.e. 
for all $s>0$ we have that 
\[
0< K(s,x) \leq \frac{s}{w(s)} \|x\|_{K_{L^\infty_w}}
\]
Since $K(\cdot, x)$ is concave we further get 
\[
K(t,x) = \inf_{s>0} \max (1, t/s) K(s, x) \leq \inf_{s>0} \max (1, t/s) \frac{s}{w(s)} \|x\|_{K_{L^\infty_w}}. 
\] 
Set 
\[
\phi(t) := \inf_{s>0} \max (1, t/s) \frac{s}{w(s)}, \quad\, t>0.
\]
Note that $\phi$ is quasi-concave and for $\phi_*(t):= \frac{t}{\phi(t)}$ for all $t>0$:
\[
\|x\|_{K_{L^\infty_{\phi_*} } }:= \sup_{t>0} \frac{1}{t}K(t,x) \phi_*(t) \leq  \|x\|_{K_{L^\infty_w}}. 
\]
Conversely, since $\phi(t) \leq t/w(t)$,  for any $x \in X$ and $t>0$ we get
\[
\frac{1}{t}K(t,x) w(t) \leq \frac{1}{t}K(t,x) \phi_*(t).   
\]
It shows that for $x\in K_{L^\infty_{\phi_*}}(X,Y)$, $\|x\|_{K_{L^\infty_w}}\leq \|x\|_{K_{L^\infty_{\phi_*}}}$.
Therefore, the proof of (i) is complete. 

For (ii), we can assume that $\min(1, (\cdot)^{-1})\in L^1_w$, since in the other case $K_{L^1_w}(X,Y) = \{0\}$  and the proof is trivial. 
Let $\ud \mu := \frac{w(t)}{t}\ud t$.
 Recall that any concave function $\phi$ has the Riemann-Stieltjes integral representation 
\[
\phi(t) = - \int_{\R_+} \min(s,t) \ud \phi'(s), \qquad t>0.
\] 

In particular, for any $x\in X$ there exists a Borel measure $\nu_x$ on $\R_+$ such that 
\[
K(t,x) = \int_{\R_+} \min(s,t) \ud \nu_x(s), \qquad t>0.
\]
 Therefore, if $x\in K_{L^1_w}(X,Y)$, then by Tonelli's theorem we get that
\begin{align*}
\|x\|_{L^1_w} & = \int_{\R_+}  \int_{\R_+} \min(s,t)  \ud \nu_x(s) \ud \mu (t)\\
& = \int_{\R_+} \int_{\R_+} \min(s,t)  \ud \mu(t) \ud \nu_x (s). \\
\end{align*}
Note that the function $\psi(s):= \int_{\R_+} \min(s,t)  \ud \mu(t)$, $s>0$, is quasi-concave and finite for all $s>0$. Let $\phi$ be the least concave majorant of $\psi$. In particular, if $\nu_\phi$ stands for the measure on $\R_+$ generated by $-\phi'$, then we get
\begin{align*}
\|x\|_{L^1_w} & \simeq \int_{\R_+} \int_{\R_+} \min(s,t)  \ud \nu_\phi(t) \ud \nu_x (s) = \int_{\R_+}  K(t,x) \ud \nu_\phi (t).
\end{align*}
 It completes the proof. 
\end{proof}
\begin{remark} Note that the concave function $\phi$ in (ii) of Proposition \ref{represent for K} is the least concave majorant of the function $\phi_\mu(\cdot, 1)$, where $\phi_\mu(\cdot, \cdot)$ is a {\it fundamental function of the functor} $K_{\Phi}$ with $\Phi:=L^1(\R_+,\ud \mu)$; 
see, e.g. \cite[p. 145]{BrKr91} for the corresponding definition. 

Indeed, recall that the function $\phi_\mu$ is defined by the condition $K_{\Phi} (t\R,s\R) = \phi_\mu(t,s) \R$ for all $s,t>0$, that is, for all $r\in \R$ we have that $\|r\|_{K_{\Phi} (t\R,s\R) } = \phi_\mu(t,s) |r|$, where $|r|$ stands for the modulus of $r$. 
Since for all $r\in \R$ and $h>0$,  $K(h, r, t\R, \R) = \min(t,h) |r|$, we obtain that 
\[
\|r\|_{K_{\Phi} (t\R, \R) } = |r| \int_{\R_+} \frac{1}{h}\min(t,h) \ud \mu(h).
\]
Therefore, it yields $\phi_\mu(t,1) = \int_{\R_+} {\frac{1}{s}} \min(t,s) \ud \mu(s)$, $t>0$, as we claimed. 
\end{remark}

\subsection{The Hardy operator}
Let $P$ denote the {\em Hardy operator} and $Q$ be its (formal) adjoint, that is, the integral operators given by 
\[
Pf(t):=\frac{1}{t}\int_0^t f (s)\,\ud s, \quad\quad  Qf(t):= \int_t^\infty f (s)\,\frac{\ud s}{s}
\]
for all $f\in\cM$ such that the respective integrals exist for all $t\in (0,\infty )$. The maximal domain of $P$, $D_{max}(P)$, is hence the space $L^1_{loc}: = L^1_{loc}([0,\infty))$, while the maximal domain of $Q$, $D_{max}(Q)$, is the set of all measurable functions $f\in\cM$ such that $f\chi_{(\tau ,\infty )} \in L^1 ( \R_+, t^{-1}\, \ud t)$ for every $\tau >0$.

Note that if $\Phi$ is a Banach function space such that the Hardy operator $P$ is bounded on $\Phi$, then $\Phi \subseteq L_{loc}^1$, and $\chi_{(0,1)}\in \Phi$ if and only if $\min(1,\frac{1}{(\cdot)})\in \Phi$. According to \cite[Lemma 2.2]{ChKr17}, in this case,  if $\chi_{(0,1)}\notin \Phi$, then $K_\Phi(X,Y) =\{0\}$.

The following characterization of the boundedness of the Hardy operator $P$ and its adjoint $Q$ on weighted $L^1$ and $L^\infty$ spaces is essentially stated in \cite{Mu72b}; see also references therein. 

Since, in \cite[Theorems 1 and 2]{Mu72b}, the weight functions are finite almost everywhere, the characterization conditions provided there, according to the convention established after \cite[Theorem 1]{Mu72b}, can be reformulated as follows: for any weight function $w \colon (0,\infty) \to [0, \infty]$, the operator $P$ (resp. $Q$) is {\it bounded} on $L^1_w$ if and only if
\begin{equation}\label{Mu con P}
\sup_{t>0} \left[\esssup_{0<r\leq t} \frac{1}{w(r)} \int_t^\infty \frac{w(s)}{s} \ud s \right] < \infty 
\end{equation}
(resp. 
\begin{equation}\label{Mu con Q}
\sup_{t>0} \left[\esssup_{r\geq t} \frac{1}{rw(r)} \int_0^t w(s) \ud s \right]< \infty).
\end{equation}
(Similar statement in the case of $L^\infty$ estimates can be extracted from \cite[Theorem 3]{Mu72b}). 
Since we are working with weight functions with values in $(0,\infty)$ and we are rather interested in the mapping properties of $P$ and $Q$, 
 i.e. $P L^1_w, QL^1_w  \subset L^1_w$, the expression 
$ \esssup_{0<r\leq t} \frac{1}{w(r)}$ (resp. $\esssup_{r\geq t} \frac{1}{rw(r)}$) can be replaced with $ \frac{1}{w(t)}$ (resp. $\frac{1}{tw(t)}$). It can be read from the proof provided in \cite{Mu72b}. For the convenience of the reader we provide a direct proof here. 

\begin{lemma}[{\cite[Theorems 1 and 2]{Mu72b}}]\label{P ineq eqiv} Let $w$ be a weight function on $(0, \infty)$. Then, the following statements are true:
\begin{itemize}
\item [(i)] The operator $P$ $($resp. $Q$$)$ is bounded on $L^1_w$ if and only if
\begin{equation}\label{P on l1}
[w]_{P, L^1}:=\es_{t>0} \frac{1}{w(t)} \int_t^\infty\frac{w(s)}{s} \ud s <\infty 
\end{equation}
$\Big($resp. 
\begin{equation}\label{Q on l1}
[w]_{Q, L^1} := \es_{t>0} \frac{1}{tw(t)}\int_0^t w(s) \ud s <\infty\Big).
\end{equation}

\item [(ii)] The operator $P$ $($resp. $Q$$)$ is bounded on $L^\infty_w$ if and only if 
\begin{equation}\label{P on l infty}
[w]_{P,L^{\infty}} := \es_{t>0} \frac{w(t)}{t} \int_0^t \frac{1}{w(s)} \ud s < \infty
\end{equation}
$\Big($resp. 
\begin{equation} \label{Q on l infty}
[w]_{Q,L^{\infty}} := \es_{t>0}  w(t) \int_t^\infty \frac{1}{sw(s)} \ud s < \infty\Big). 
\end{equation}
\end{itemize}
\end{lemma}

\begin{proof} (i) For a weight function $w$ let $\Phi :=  L^1_w$. 
Note that the K\"othe dual $\Phi'$ of $\Phi$ ($\Phi$ is considered as a Banach function space over $(\R_+, \ud t)$) is $L^\infty(w^{-1})$. 

First, note that the condition $[w]_{P,L^1}<\infty$, implies that $\Phi \subset L^1_{loc} =  D_{max}(P) $. Indeed, let $f \in \Phi$, then for each $\tau>0$, since by our assumption $w(s)>0$ for all $s>\tau$, we can write
\begin{align*}
\int_0^{\tau} |f(t)| \ud t & \leq [w]_{P,L^1}  \int_0^\tau  \left(\int_t^\infty\frac{w(s)}{s} \ud s\right)^{-1} |f(t)| w(t) \ud t\\
& \leq [w]_{P,L^1}  \left(\int_\tau^\infty\frac{w(s)}{s} \ud s\right)^{-1} \int_0^\tau   |f(t)| w(t) \ud t < \infty.
\end{align*} 

Now, for all $g\in \Phi'$ with $\|g\|_{\Phi'} = 1 $, i.e. $|g(s)| \leq w(s)$  for a.e. $s>0$, and for every $f\in \Phi$, since $(0,\infty)^2\ni (t,s) \mapsto \frac{|f(s)g(t)|}{t}\chi_{(0,t)(s)}$ is measurable, by Tonelli's theorem, we get
\begin{align*}
\int_0^\infty | Pf(t) g(t) | \ud t & \leq  \int_{0}^\infty \int_{0}^\infty \frac{|f(s)|}{t} \chi_{(0,t)}(s)  \ud s |g(t)|  \ud t\\
& = \int_{0}^\infty \int_{s}^\infty \frac{|g(t)|}{t} \ud t |f(s)|  \ud s\\
& \leq  \int_{0}^\infty \int_{s}^\infty \frac{w(t)}{t} \ud t |f(s)|  \ud s\\
& \leq  [w]_{P, L^1} \int_{0}^\infty |f(s)| w(s) \ud s.
\end{align*}  
By the Lorentz-Luxemburg theorem, $(\Phi ')' = \Phi$, therefore, $\|Pf\|_{\Phi}\leq [w]_{P,L^1} \|f\|_\Phi$ for all $f\in \Phi$.

To prove the statement for $Q$, first note that $[w]_{Q,L^1}$ yields $\Phi \subset D_{max}(Q)$. Indeed, for $\tau>0$, since $0< w(s)$ for all $s \in (0,\tau)$, we obtain that 
\begin{align*}
\int_\tau^\infty |f(t)| \frac{\ud t}{t} &  \leq [w]_{Q,L^1} \int_\tau^\infty 
\left(\int_0^t w(s) \ud s\right)^{-1} |f(t)| w(t) \ud t < \infty.
\end{align*}

Similarly as above, for any $g\in \Phi '$ with $\|g\|_{\Phi} \leq 1$ and any $f\in \Phi$, by Tonelli's theorem,  we get

 \begin{align*}
\int_0^\infty | Qf(t) g(t) | \ud t & = \int_{0}^\infty \int_{0}^\infty \frac{|f(s)|}{s} \chi_{(t, \infty)}(s)  \ud s |g(t)|  \ud t\\
& = \int_{0}^\infty \frac{1}{s} \int_{s}^\infty |g(t)| \ud t |f(s)|  \ud s\\
& \leq  \int_{0}^\infty \frac{1}{s} \int_{s}^\infty w(t) \ud t |f(s)|  \ud s\\
& \leq  [w]_{Q, L^1} \int_{0}^\infty |f(s)| w(s) \ud s.
\end{align*}  
Again, the Lorentz-Luxemburg theorem gives the desired boundedness of $Q$.

For the converse implication, recall that by \cite[Theorem 1 and 2]{Mu72b}, the boundedness of $P$ and $Q$ on $L^1_w$ implies respectively the conditions \eqref{Mu con P} and \eqref{Mu con Q} (even for any measurable function $w$ on $(0,\infty)$ with values in $[0,\infty]$). Since those conditions imply obviously the corresponding ones stated in (i), and since  in the sequel we do not use this converse implication in our proofs, its direct proof is omitted here. The proof of (i) is complete. 

 For the proof of (ii) note that the boundedness of $P$ (respectively, $Q$) on $L^1_w$ is equivalent to the boundedness of $Q$ (respectively, $P$) on $L^\infty(w^{-1})$. Therefore, (ii) follows form (i). It completes the proof.
\end{proof}

The following consequence of Lemma \ref{P ineq eqiv}  provides a simple (in its form) condition that characterizes weight functions for which the sum operator $P+Q$ (called the Calder\'on operator) is bounded on weighted $L^1$ (or $L^\infty$) spaces. We formulate in terms of these conditions our main results stated in Section \ref{sect DaP-G}.

\begin{lemma}\label{P+Q ineq} Let $w$ be a weight function on $(0,\infty)$. 
Then the following assertions hold. 
\begin{itemize}
\item [(i)] The operators $P$ and $Q$ are bounded on $L^1_w$ if and only if 
\begin{equation}\label{P+Q on L1}
[w]_{P+Q,L^1}:= \es_{t> 0} \frac{1}{w(t)}\int_0^\infty \frac{w(s)}{t+s} \ud s < \infty. 
\end{equation}
\item [(ii)] The operators $P$ and $Q$ are bounded on $L^\infty_w$ if and only if  $P$ and $Q$ are bounded on $L^1_{w^{-1}}$ if and only if 
\begin{equation}\label{P+Q on Linfty}
[w]_{P+Q,L^\infty}:= \es_{t> 0} {w(t)}\int_0^\infty \frac{1}{t+s}\frac{1}{w(s)} \ud s < \infty.  
\end{equation}
\end{itemize}
\end{lemma}
\begin{proof}
Note that 
\begin{align*}
\frac{1}{w(t)}\int_0^\infty \frac{1}{t+s} w(s) \ud s & = 
\frac{1}{w(t)} \int_0^{t} \frac{1}{t+s} w(s) \ud s + \frac{1}{w(t)} \int_{t}^\infty \frac{1}{t+s} w(s) \ud s \\
\nonumber &\leq  \frac{1}{tw(t)} \int_0^{t} w(s) \ud s + \frac{1}{w(t)} \int_{t}^\infty \frac{w(s)}{s} \ud s\\
\nonumber & \leq [w]_{P,L^1} + [w]_{Q,L^1}
\end{align*}
On the other hand, since for every $t>0$, $\inf_{s\geq t} \frac{s}{s+t} = \inf_{0<s\leq t} \frac{t}{t+s} = \frac{1}{2}$, we easily get 
\begin{align*}
\frac{1}{w(t)}\int_0^\infty \frac{1}{t+s} w(s) \ud s   & = 
\frac{1}{w(t)} \int_0^{t} \frac{1}{t+s} w(s) \ud s + \frac{1}{w(t)} \int_{t}^\infty \frac{1}{t+s} w(s) \ud s \\
& \geq \frac{1}{2} \frac{1}{tw(t)} \int_0^{t} w(s) \ud s + \frac{1}{2}\frac{1}{w(t)} \int_{t}^\infty \frac{w(s)}{s}  \ud s \\
\end{align*}

Therefore, $2 [w]_{P+Q, L^1} \geq [w]_{P,L^1} + [w]_{Q,L^1}$ Lemma \ref{P ineq eqiv}(i) completes the proof of (i). 
The proof of (ii) follows exactly the same argument. 
\end{proof}

\begin{remark} (a) Recall that the function $\mathcal{S} w(t) = \int_0^\infty \frac{w(s)}{t+s} \ud s$, $t>0$, is called the {\emph{Stieltjes transform}} of the function $w$. Therefore, the statement (i) of the above result can be rephrased as follows: the Calder\'on operator $P+Q$ is bounded on $L^1_w$ if and only if Stieltjes' transform of $w$ is up to a constant less than $w$ on $\R_+$. It is readily seen that for non-negative $f$, $\cS f \leq  Pf+Qf\leq 2  \cS f$ a.e. on $\R_+$. 
\end{remark}

The following corollary shows that for the power weights $w_\mu(t): = t^{-\mu}$ $(t>0)$ with $\mu \in (0, 1)$, the quantity $[w]_{P+Q,L^1}$ can be computed exactly. 

\begin{corollary}\label{P+Q for powers} Let $w_\mu(t): = t^{-\mu}$ $(t>0)$ with $\mu \in (0, 1)$.  Then, $[w]_{P+Q,L^1} = \frac{\pi}{\sin(\pi\mu)}$ and, consequently, $P+Q$ is bounded on $L^1_{w_\mu}$ and on $L^\infty(w_\mu^{-1})$.
\begin{proof} Let $B$ stand for the beta function. Then, the statement follows from the following
\begin{align*}
t^{\mu} \int_0^{\infty} \frac{1}{s^{\mu}(t+s)} ds & = \int_0^{\infty}\frac{s^{-\mu}}{1 + s} ds = B(1-\mu,\mu) = \frac{\pi}{\sin(\pi\mu)}
\end{align*}
and Lemma \ref{P+Q ineq}.
\end{proof}
\end{corollary}

We conclude this subsection with an extension of Proposition \ref{represent for K}(ii). Roughly, we show that under an additional assumption on the parameter $\Phi = L^1_w$ the corresponding interpolation functor $K_\Phi$ admits a further  equivalent description, which simplifies that given in Proposition \ref{represent for K}(ii). We start with necessary preparatory results on the $J$-method and refer the reader to \cite{BeLo76} or \cite{BrKr91}  for the background.

Recall that if $\rho$ is a~quasi-concave function and $\xo=(X_0, X_1)$ is a Banach couple, then
$J_{\rho, 1}(\xo)$ consists of all $x\in X_0 + X_1$ which can be represented by
\begin{align*}
x = \int_0^\infty u(t) \frac{dt}{t} \quad \text{convergence in \, $X_0 + X_1$}\,,
\end{align*}
where $u\colon (0, \infty) \to X_0 + X_1$ is strongly measurable with values in $X_0\cap X_1$ and
\[
\int_0^\infty \frac{J(t, u(t): \xo)}{\rho(t)} \frac{dt}{t} <\infty\,.
\]

Let $J_{\rho, 1}(\xo)$ be equipped with the norm
\[
\|x\|_{\rho, 1} = \inf \int_0^\infty \frac{J(t, u(t); \xo)}{\rho(t)} \frac{dt}{t}\,,
\]
where the infimum is taken over all $u$ such that $x= \int_0^\infty u(t) \frac{dt}{t}$ with the mentioned properties. Then, $J_{\rho, 1}$ is an exact interpolation functor. 
Here, $J\colon (0, \infty) \times (X_0 \cap X_1) \to [0, \infty)$ is the so called the $J$-functional defined by
\[
J(t, x; \xo):= \max(\|x\|_{X_0}, \|x\|_{X_1}), \quad\, t>0,\,\, x \in X_0 \cap X_1\,.
\]

In what follows we will use the well-known and easily verified maximal property which states that for any exact 
interpolation functor $F$ with characteristic function $\varphi$ the following continuous inclusion holds 
for any Banach couple $\xo$ with norm less than or equal to $1$
\[
J_{\rho, 1}(\xo) \hookrightarrow F(\xo)\,,
\]
where $\rho(t):= \varphi_{*}(1, t)$ and $\varphi_{*}(s, t):= 1/\varphi(1/s, 1/t)$ for all $s, t>0$.

Recall also that the {\it lower} and \emph{upper dilation indices} of a~quasi-concave function $\phi \colon \R_+ \to \R_+$ 
are given by
\[
\alpha_{\phi} := \lim_{t\to 0\plus} \frac{\ln s_{\phi}(t)}{\ln
t}, \quad\, \beta_{\phi} := \lim_{t \to \infty}\frac{\ln s_{\phi}(t)}{\ln t},
\]
respectively, where $s_{\phi}(t) := \sup_{r>0}\frac{\phi(rt)}{\phi(r)}$ for all $t>0$. Clearly, $0 \leq \alpha_\phi \leq 
\beta_\phi \leq 1$. A~quasi-concave function $\varphi$ is said to be {\it quasi-power} if  $0<\alpha_{\phi} \leq \beta_{\phi} <1$.

\medskip

To proceed, we require the following auxiliary result, which characterizes the nontriviality of the dilation indices of quasi-concave functions 
through integral conditions. While the fact that the nontriviality of the indices implies integral estimates is well known, we include a proof 
for the sake of completeness.

\begin{lemma}\label{indices}
For any quasi-concave function $\phi\colon \R_+\rightarrow~\R_+$ the following statements hold: 
\begin{itemize}
\item[{\rm(i)}] The estimate $\int_0^t \frac{\phi(s)}{s}\,\ud s \lesssim \phi(t)$ for all $t>0$ is equivalent to $\alpha_{\phi}>0$.
\item[{\rm(ii)}] The estimate $\int_t^\infty \frac{\phi(s)}{s^2}\,\ud s \lesssim \frac{\phi(t)}{t}$ for all $t>0$ is equivalent to 
$\beta_{\phi}<1$.
\end{itemize}
\end{lemma}

\begin{proof}
(i)  Suppose, for the sake of contradiction, that 
\[
\int_0^t \frac{\phi(s)}{s}\,\ud s \lesssim \phi(t), \quad t>0\,,
\]
and $\alpha_{\phi} = 0$. Since
\[
\alpha_{\phi} = \sup_{0<t<1} \frac{\ln s_\phi(t)}{\ln t}\,,
\]
it follows that for all $t\in (0, 1)$, one has $s_\phi(t)=1$. Consequently, for each positive integer $n$, 
there exists a~$r_n > 0$ such that  
\[
\frac{\phi(r_n/n)}{\phi(r_n)} \geq \frac{1}{2}\,.
\]  
This leads to the estimate  
\[
\sup_{t>0} \frac{1}{\phi(t)} \int_0^t \frac{\phi(s)}{s}\,\ud s   
\geq \frac{\phi(r_n/n)}{\phi(r_n)} \int_{r_n/n}^{r_n} \frac{\phi(s)}{s\,\phi(r_n/n)}\,\ud s
\geq \frac{1}{2} \log n
\]
for each $n\in \mathbb{N}$, which is a contradiction. Therefore, $\alpha_\phi>0$.

For the reverse implication, assume that  
$\alpha_{\phi} > 0$. By the definition, for any $\varepsilon \in (0, \alpha_{\phi})$ there exists a constant $K = K(\varepsilon) > 0$ such that  
\[
s_{\phi}(t) \leq K t^{\alpha_{\phi} - \varepsilon}, \quad\,  t \in (0,1]\,.
\]  
This implies that
\[
C_{\phi} := \int_0^1 \frac{s_{\phi}(r)}{r}\,\ud r 
\leq \frac{K}{\alpha_{\phi}- \varepsilon} < \infty\,.
\]  
Hence,  
\[
\int_0^t \frac{\phi(s)}{s}\,\ud s = \bigg(\int_0^1 \frac{\phi(tr)}{r\,\phi(t)}\,\ud r \bigg) \phi(t) \leq C_{\phi} \phi(t), \quad t > 0.
\]  
This completes the proof of (i).

\smallskip

(ii) Let $\psi(t) := t \phi(1/t)$ for all $t > 0$. A change of variable shows that the estimate  
$\int_t^\infty \frac{\phi(s)}{s^2} \, \ud s \lesssim \frac{\phi(t)}{t}$ for all $t>0$ is equivalent to the estimate  
$\int_0^{1/t} \frac{\psi(s)}{s}\,\ud s \lesssim \psi(1/t)$ for all $t > 0$, and consequently to  
\[
\int_0^t \frac{\psi(s)}{s}\,\ud s  \lesssim \psi(t), \quad\, t > 0\,.
\]  
Clearly, $\psi$ is a quasi-concave function. Thus, by (i), this estimate is equivalent to $\alpha_{\psi} > 0$. 
Since $s_{\psi}(t) = t s_{\phi}(1/t)$  for all $t > 0$, it follows that  $\alpha_{\psi} = 1 - \beta_{\phi}$. This completes the proof.  
\end{proof}

\begin{proposition} \label{KJ}
Let $\Phi:=L^1_w$ be a parameter of the $K$-method such that operators $P$ and $Q$ are bounded on $\Phi$.  Let  $\Psi:= L^1_v$, where $v(t):= \varphi(1, 1/t)$ for all $t>0$ and 
$\varphi$ is the characteristic function of the functor $K_\Phi$.

Then, the following statements are true{\rm:}
\begin{itemize}
\item [{\rm(i)}] 
For any Banach couple $\xo = (X_0, X_1)$,
\[
K_\Phi(\xo) \simeq K_\Psi(\xo)\,,
\]
up to equivalence of norms depending only on $[w]_{P, L_1}$ and $[w]_{Q, L_1}$

\item[{\rm(ii)}] The operators $P$ and $Q$ are bounded on $L^1_v$.
\item[{\rm(iii)}] The function $\phi = \varphi(\cdot, 1)$ is quasi-power.
\end{itemize}
\end{proposition}

\begin{proof} (i) Note that $\varphi(1, 1/t)= \int_0^\infty \min\big(1, \frac{s}{t}\big) \frac{w(s)}{s}\,\ud s$ for all $t>0$. Thus, by Lemma \ref{P+Q ineq}, the boundedness of $P$ and  $Q$ on $L^1_w$ yields that for almost all $t>0$ one has
\[
v(t) = \varphi(1, 1/t) = \frac{1}{t}\,\int_0^t w(s)\,\ud s + \int_t^\infty \frac{w(s)}{s}\,\ud s \leq C w(t)
\]
with $C:= [w]_{P, L^1} + [w]_{Q, L^1}$. Consequently, $
K_\Phi(\xo) \hookrightarrow K_\Psi(\xo)$
with the norm of the inclusion map less than or equal to $C$.

For the inverse inclusion, let $\rho(t):= 1/\varphi(1, 1/t)$ for all $t>0$. 
Then, based on the aforementioned maximal property, we have $
J_{\rho, 1}(\xo) \hookrightarrow K_\Phi(\xo)\,.$ Thus, it is enough to show that 
\[
K_\Psi(\xo) \hookrightarrow J_{\rho, 1}(\xo)
\]
with the norm of the inclusion map independent of $\xo$. To see this fix $x \in K_\Psi(\xo)$. Note that from the quasi-concavity of  
$\varphi(1, \cdot)$, we have $\min(1, 1/t)\varphi(1, 1) \leq \varphi(1, 1/t)$ for all $t>0$, and hence  
\[
L^1(v(t)/t) \hookrightarrow L^1\big(\min(t^{-1}, t^{-2}))\,.
\]
This implies $\min(1, 1/t)K(t, x; \xo) \to 0$ as $t\to 0$ or as $t\to \infty$. Thus, applying the fundamental lemma of interpolation theory, 
we conclude the existence of a representation $x = \int_0^\infty u(t)\frac{dt}{t}$ (convergence in $X_0 + X_1$) with 
$u\colon (0, \infty) \to X_0 + X_1$ strongly measurable with values in $X_0 \cap X_1$, such that
\[
J(t, u(t); \xo) \leq \gamma K(t, x; \xo), \quad\, t>0\,,
\]  
where $\gamma$ is a universal constant; see \cite[Lemma 3.3.2]{BeLo76} with $\gamma <4$. 
In consequence, $x\in J_{\rho, 1}(\xo)$ with $|x|_{\rho, 1} \leq \gamma |x|_{K_\Psi(\xo)}$, and this completes the proof of(i).

(ii) Taking 
$\xo := (t\mathbb{R}, \mathbb{R})$ in (i), we conclude that 
\[
\int_0^\infty s^{-1}\min(t,s) \varphi(1, 1/s)\,ds \leq C \varphi(t, 1), \quad\, t>0\,. 
\]
Hence, by  Lemma \ref{P+Q ineq}(i),  
\begin{equation}\label{ineq for v}
\int_0^t v(s)\ud s \leq C t\,v(t), \quad\, \int_t^\infty \frac{v(s)}{s}\,\ud s \leq C \,v(t), \quad\, t>0
\end{equation} 
and this proves the statement (ii).

(iii) Clearly, the inequalities \eqref{ineq for v} are equivalent to 
\[
\int_0^t \frac{\phi(s)}{s} \ud s  \leq C \phi(t), \quad\, \int_t^\infty \frac{\phi(s)}{s^2} \ud s \leq C \frac{\phi(t)}{t}, \quad\, t>0\,.
\]
Therefore, Lemma \ref{indices} completes the proof.
\end{proof}

\subsection{The representation of interpolation spaces $K_\Phi(\cX,D_\cA)$}

The next preparatory result, Proposition \ref{Komastu rep}, is well-known for the parameters $\Phi$ corresponding to the classical interpolation functors $(\,\cdot\,)_{\theta, q}$ with $\theta \in (0,1)$ and $q\in [1,\infty]$; see e.g. \cite[Chapter 6.2]{Ha06}. For a more general class of $\Phi$ it can be extracted as a special case of \cite[Theorem 3.1]{ChKr16}; see also \cite[Remarks 3.2 and 3.3]{ChKr16}. For the reader's convenience, we give a direct, self-contained, alternative proof based on Lemma \ref{lem on density}.

\begin{proposition}\label{Komastu rep} Let $-\cA$ be a generator of a bounded analytic semigroup $\cT$ on a Banach space $\cX$. Let $\Phi$ be a Banach function space over $\R_+$ such that $\min(1, \frac{1}{(\cdot)}) \in \Phi$ and  the Hardy operator $P$ is bounded on $\Phi$. 

Then, 
\[
K_\Phi(\cX, D_\cA) = \left \{ x\in \cX:   \left[ (0,\infty)\ni t\mapsto \|\cA \cT(t)x \|_{\cX} \right]\in \Phi \right\}
\]
and 
\begin{equation}\label{equi norm}
\|x\|_{K_\Phi(\cX,D_\cA)} \simeq \|x\|_\cX + \| \cA \cT(\cdot) x \|_{\Phi(\cX)},   \qquad x\in K_\Phi(\cX,D_\cA).
\end{equation}

In particular, if, in addition,  $\cA$ is invertible, i.e. $0\in \rho(\cA)$, then 
\[
\|x\|_{K_\Phi(\cX,D_\cA)} \simeq  \| \cA \cT(\cdot) x  \|_{\Phi(\cX)} , \qquad x\in K_\Phi(\cX,D_\cA).
\]

Moreover, the part of $\cA$ in $K_\Phi(\cX,D_\cA)$ is again the generator of a bounded analytic semigroup $T$ on $K_\Phi(\cX,D_\cA)$, which is the restriction of $\cT$ to $K_\Phi(\cX,D_\cA)$. 
\end{proposition}

\begin{proof}
First note that for every $x\in \cX$ and $t>0$ 
\begin{equation}\label{Komatsu}
\|\cA \cT(t) x\|_{\cX} \lesssim \frac{1}{t}K(t,x)\leq \frac{1}{t}\int_0^t \|\cA \cT(s) x\|_{\cX} \ud s + \|\cA \cT(t)x\|_{\cX} + \|\cT(t) x\|_{\cX}.
\end{equation}
In fact, for the first inequality note that for each $x = y+z \in \cX$ with $y\in \cX$ and $z\in D_\cA$ we have that 
\[
t\|\cA \cT(t) x\|_{\cX} \leq \sup_{s>0}\|s\cA \cT(s)\|_{\cL(\cX)} \|y\|_\cX + \sup_{s>0}\|\cT(s)\|_{\cL(\cX)}\, t \|\cA z\|_\cX. 
\]
For the second one, by Lemma \ref{lem on density}, for every $t>0$ and $x \in \cX$, $x - \cT(t) x = \cA \int_0^t \cT(s) x \ud s\in \cX$ and $\cT(t) x \in D_\cA$.  

Moreover, it is easily seen that for all $t>0$ and $x\in \cX$
\[
\min(1, {t^{-1}})\|x\|_\cX \leq t^{-1}K(t,x) \leq t^{-1} \|x\|_{\cX}.
\]

Combining this estimate with \eqref{Komatsu}, the uniform boundedness of $(\cT(t))_{t\geq 0}$ on $\cX$ and $P$ on $\Phi$, we get \eqref{equi norm}.  

Furthermore, if $\cA$ is invertible, then $\|\cT(t)x\|_\cX \leq \|\cA^{-1}\|_{\cL(\cX)} \|\cA \cT(t)x\|_{\cX}$ for all $t>0$. By \eqref{Komatsu}, it proves the second claim. The last one is a direct consequence of \eqref{equi norm} and the well-known characterization of generators of analytic semigroups; see e.g. \cite[Theorem 3.7.11]{ABHN01}.  
\end{proof}

Note that the expression 
\begin{equation}\label{homoge part}
[x]_{\Phi, \cA} :=  \| \cA \cT(\cdot) x  \|_{\Phi(\cX)}, \qquad x\in K_\Phi(\cX,D_\cA),
\end{equation}
is a seminorm on $K_\Phi(\cX,D_\cA)$, and in the case when $\cA$ is injective, it gives a norm, which is equivalent to $\|\cdot\|_{K_\Phi(\cX,D_\cA)}$ if  $0\in \rho(\cA)$.  We refer to $[\cdot]_{\Phi, \cA}$ as the {\it homogenous part} of the norm $\|\cdot\|_{K_\Phi(\cX,D_\cA)}$.

\subsection{Equivalent representation of homogenous part of interpolation norms }\label{subs H}

Note that \cite[Theorems 3.1 and 3.4]{ChKr16} provides further equivalent representations of the interpolation spaces $K_\Phi(\cX,D_\cA)$ and its norms.  
A close analysis of the proof of those results in fact shows that the homogeneous part \eqref{homoge part} admits equivalent expressions, which we discuss in this subsection.  

More precisely, if $\psi\neq 0$ is a holomorphic function on a sector $S_\omega:=\{z\in \C: z\neq 0, |\arg z|<\omega\}$, with $\omega>\omega_\cA$ (the sectoriality angle of $\cA$) such that one of the following conditions holds:
\begin{itemize}
\item [(A)] $\psi$, $(\cdot)^{-1} \psi\in \cE(S_\omega)$, $\lim_{z\rightarrow 0}z^{-1} \psi(z) \neq 0$, $\sup_{z\in S_\omega, s\geq 1} \left|\frac{\psi(sz)}{s\psi(z)} \right|<\infty$, and 
$P$ is bounded on $\Phi$ (\cite[Theorem 3.1]{ChKr16}), 
\item [(B)] $\psi$, $(\cdot)^{-1} \psi\in \cE(S_\omega)$, and $P$ and $Q$ are bounded on $\Phi$ (\cite[Theorem 3.4]{ChKr16}),
\end{itemize}
then 
\[
K_\Phi(\cX,D_\cA) = \{ x\in \cX: \left[(0,\infty)\ni t\mapsto t^{-1}\|\psi(t\cA)x\|_\cX \right]\in \Phi\}
\]
\begin{equation}\label{equiv}
[x]_{\Phi, \cA} :=  \| \cA \cT(\cdot) x  \|_{\Phi(\cX)} \simeq \|(\cdot)^{-1} \psi(\cdot \cA)x\|_{\Phi(\cX)}\,, \qquad x\in \cX.
\end{equation} 
Here, $\cE(S_\omega)$ denotes the subalgebra of the space $H^\infty(S_\omega)$ of bounded holomorphic functions on $S_\omega$, which is given by $\cE(S_\omega): = H^\infty_0(S_\omega) \oplus \langle(1+z)^{-1}\rangle \oplus \langle 1 \rangle$, where 
\[
H^\infty_0(S_\omega): =\{f\in H^\infty(S_\omega):\exists C, \alpha>0 \textrm{ s.t. } |f(z)|\leq C\min(|z|^\alpha, |z|^{-\alpha}) (z\in S_\omega)\}.
\]

For instance, each of the following functions  satisfies the above condition (A):
\begin{equation}\label{psi funct}
\psi_1(z)=e^{-z} - 1,\qquad \psi_2(z) = z(1+z)^{-1},\qquad \psi_3(z)= z(1+z)^{-2}.
\end{equation}

Below, we provide an extension of this statement for a special class of $\Phi$; see Proposition \ref{Haase equi} below.  We start with some preliminaries. For any Banach function space $\Phi$ over $(\R_+,\ud t)$ (not necessarily such that $P$ is bounded on it) and for any $\psi\in \cE(S_\omega)$ we put
\begin{equation}\label{cX spaces}
\cX_{\Phi, \psi}:= \cX_{\Phi, \psi, \cA}:= \{ x\in \cX: [x]_{\Phi, \psi, \cA}<\infty\},
\end{equation}
where 
\[
[x]_{\Phi,\psi}: =  [x]_{\Phi,\psi,\cA} := \|(\cdot)^{-1}\psi(\cdot \cA)x \|_{\Phi(\cX)} 
\]
Note that $[\cdot]_{\Phi, \cA}$ in \eqref{homoge part}  corresponds to $[\cdot]_{\Phi, \psi, \cA}$ with $\psi(z): = ze^{-z}$.

In the case of injective sectorial operators $\cA$ on Hilbert spaces $\cX$ and $\Phi = L^p(\R_+,t^{\mu} \ud t)$ with $p\in [1,\infty]$ and $\mu\in \R$, the corresponding spaces $\cX_{\Phi, \psi, \cA}$ for $\psi \in H^\infty_0(S_\omega)$ with some additional decaying properties at $0$ and $\infty$, have been developed and studied by McIntosh and his co-workers; see e.g. \cite{ADM96}. A systematized account of related results, rephrased already in a general setting of Banach spaces $\cX$, was provided by Haase in \cite[Section 6.4]{Ha06}, where the {\it homogeneous spaces} may be, in general, larger than $\cX_{\Phi, \psi, \cA}$ defined above. 

In particular, when $\cA$ is injective, an analysis of the proof of \cite[Theorem 6.4.2, p.150]{Ha06} shows that for $\Phi = L^p(t^\mu)$ ($\mu\in \R$, $p\in [1,\infty]$), for any $\psi$, $\gamma \in \cE(S_\omega)$ with  $z^{-\theta}\psi$, $z^{-\theta}\gamma\in H^\infty_0(S_\omega)$, where $\theta = 1-\frac{\mu+1}{p}$ if $p\in [1,\infty)$ and $\theta :=1-\mu$ if $p=\infty$, the semi-norms $[\cdot]_{\Phi, \psi}$ and $[\cdot]_{\Phi, \phi}$ are equivalent {\it on $\cX$}. It is not clear whether the general case (that is, when $\cA$ is not injective) can be reduced to the injective one by some direct argument. For this reason, below we provide a modification of the argument applied in the proof of \cite[Theorem 6.4.2]{Ha06}.

The main point of the extension of \cite[Theorem 6.4.2]{Ha06}, see   Proposition \ref{Haase equi}(i),  is to show that in the case of arbitrary sectorial operators $\cA$, and $p$, $\mu$, $\theta$ and $\Phi$ as above, for some $\psi, \gamma \in \cE(S_\omega)\setminus H^\infty_0(S_\omega)$ with $z^{-\theta} \psi$, $z^{-\theta} \gamma \in H^\infty_0(S_\omega)$  the equivalence of semi-norms $[\cdot]_{\Phi, \psi}$ and $[\cdot]_{\Phi, \gamma}$ still holds {\it on $\cX$}.

Moreover, the statement (iii) of Proposition \ref{Haase equi} shows that we should not expect such equivalence of the corresponding semi-norms if we only know $z^{-\theta} \psi$, $z^{-\theta} \gamma \in \cE(S_\omega)$ instead of $z^{-\theta} \psi$, $z^{-\theta} \gamma \in H^\infty_0(S_\omega)$. Specifically, it shows that the one of the main ingredients of the proof, \cite[Theorem 5.2.2b)]{Ha06}, is sharp, and also gives an limitation on the assumptions of \cite[Theorem 6.4.2]{Ha06}.

\begin{proposition}\label{Haase equi} Let $\cA$ be a sectorial operator on a Banach space $\cX$.
\begin{itemize} 
\item [(i)]  Let $\Phi = L^p(\R_+, t^\mu \ud t)$ for $p\in [1,\infty]$ and  $\mu \in \R$. 
Let $\psi$, $\gamma \in \cE(S_\omega)\setminus \{0\}$ $($with $\omega>\omega_\cA)$ such that $(\cdot)^{-\theta} \psi$, $(\cdot)^{-\theta} \gamma \in H^\infty_0(S_\omega)$ with  $ \theta :=  1-\frac{\mu+1}{p}$ if $p\in [1,\infty)$ or $\theta := 1-\mu$ if $p=\infty$. 
If for each $\xi \in \{\psi,\gamma\}$ we have $\xi \in H_0^\infty(S_\omega)$ or 
\begin{equation}\label{McI cond}
\cX_{\Phi, \xi}\, \subset \overline{D(\cA)}\quad  \textrm{and} \quad \xi(\cA)\overline{D(\cA)}\,\subset \cR(\cA), 
\end{equation}
then, the semi-norms $[\cdot]_{\Phi, \psi}$ and $[\cdot]_{\Phi, \gamma}$ are equivalent on $\cX$.

\item [(ii)] If, in addition,  $-\cA$ generates a bounded holomorphic semigroup $\cT$ on $\cX$, then for any $\theta \in (0,1)$, if $\Phi = L^p(\R_+, t^{p(1-\theta) - 1} \ud t)$ with $p\in [1,\infty)$ or $\Phi = L^\infty(t^{1-\theta})$, then  all semi-norms $[\cdot]_{\Phi, \psi_j}$, $j=1,2,3$, for $\psi_j$ as in  \eqref{psi funct}, are equivalent to the semi-norm \eqref{homoge part} and $K_\Phi(\cX,D_\cA) = \cX_{\Phi, \psi_j}$ (as the sets).
In the case $\Phi = L^\infty(t)$ (i.e. $\theta = 0$), $\cX_{\Phi, \psi_2} = \cX_{\Phi, \psi} = \cX$.

\item [(iii)] If $-\cA$ generates a bounded holomorphic semigroup $\cT$ on $\cX$ and $\cA y \neq 0$ for some $y\in D(\cA)$,  then the conclusion of (ii) false for $\theta = 0$ and $p\in [1,\infty)$.

More precisely, for any $p\in [1,\infty)$, if $\Phi = L^p(\R_+, t^{p- 1} \ud t)$, then the semi-norm $[\cdot]_{\Phi, \psi_2}$ is not equivalent to any of the following, mutually equivalent semi-norms: $[\cdot]_{\Phi, \psi_j}$ $(j=1,3)$ and $[\cdot]_{\Phi, \cA}$ $($see \eqref{psi funct} and \eqref{homoge part}$)$.
\end{itemize}

\end{proposition}

\begin{proof}  (i) We provide only main supplementary observation which should be made in the proof of \cite[Theorem 6.4.2]{Ha06}. Following \cite[Lemma 6.2.5]{Ha06}, set 
\[
\alpha(x) := z^n(1+z)^{-2n},\quad \beta(z) := c^{-1}  \alpha(z) \overline{\gamma(\bar{z})}, \quad \phi(z):= \alpha(z) \beta(z)  \gamma(z),  \quad\,  z\in S_\omega,
\] 
where $|\theta|+ 1 < n\in \N$ and $c$ stands for $\int_0^\infty |\alpha (t) \gamma(t)|^2 \frac{\ud t}{t}$, which is finite because $\alpha \gamma \in H^{\infty}_0(S_\omega)$. 
Obviously, $\beta, \phi\in H^{\infty}_0(S_\omega)$. 
To allow $\psi$ or $\gamma$ be in $\cE(S_\omega)\setminus  H^\infty_0(S_\omega)$ but satisfy \eqref{McI cond}, the point is to use the full power of the McIntosh's approximation result (see, e.g. \cite[Theorem 5.2.6 or Lemma 5.2.3e)]{Ha06}), not only its {\it intermediate} form \cite[Proposition 5.2.4c)]{Ha06} that was applied in the proof of \cite[Theorem 6.4.2]{Ha06} and which is sufficient in the case $\psi, \gamma \in H^\infty_0(S_\omega)$. 
More precisely, arguing as in the proof of \cite[Theorem 6.4.2]{Ha06}, we derive to the estimate: for all $0<a<b<\infty$
\begin{align}\label{ab ineq}
\nonumber \|\phi_{a,b} (\cA)  t^{-\theta} & \psi(t\cA) x \|_\cX \\
& \leq c_{\beta}c_{\psi, \alpha} \left(\int_0^\infty \|\widetilde{\psi}(t\cA) \widetilde{\alpha}(s\cA)\|_{\cL(\cX)} \|s^{-\theta} \gamma(s\cA) x\|^p_\cX \frac{\ud s}{s} \right)^{1/p}
\end{align} (with modification when $p=\infty$), 
where
\[
\phi_{a,b} (\cA) := \int_a^b \phi(s\cA) \frac{\ud s}{s}, \quad \widetilde{\psi}(z):= z^{-\theta} \psi(z),\quad \widetilde{\alpha}(z):= z^{\theta} \alpha(z)\quad (z\in S_\omega)
\] and 
\[
c_{\beta}: = \sup_{s>0} \|\beta(s\cA)\|_{\cL(\cX)},   \quad c_{\psi, \alpha}: = \sup_{t>0} \left(\int_0^\infty \|\widetilde{\psi}(t\cA) \widetilde{\alpha}(s\cA)\|_{\cL(\cX)} \frac{\ud s}{s}\right)^{1/p'}.
\]

Since $\beta  \in H^\infty_0(S_\omega)$, $c_\beta$ is finite; see, e.g. \cite[Theorem 5.2.2a)]{Ha06}. 
Since $\widetilde\psi$ and $\widetilde\alpha$ are in $H^\infty_0(S_\omega)$, by \cite[Theorem 5.2.2b)]{Ha06}, $c_{\psi, \alpha}$ is also finite. Taking the $L^p(\R_+, \ud t/t)$ norm of the right-hand side of \eqref{ab ineq}, by Fubini's theorem and, again, \cite[Theorem 5.2.2b)]{Ha06}, yields the quantity less than $c_\beta c_{\psi, \alpha} c_{\alpha, \psi} [x]_{\Phi, \gamma, \cA}$ with $c_{\alpha, \psi}<\infty$. 
Therefore, with the help of Lebesgue's domination theorem, it suffices to show that for all $x\in \cX_{\Phi, \gamma}$ and $t>0$ we have:
\[
\phi_{a,b} (\cA) \psi(t\cA) x \rightarrow \psi(t\cA) x \qquad \textrm{as } a\rightarrow 0\plus, \, \, b\rightarrow \infty.   
\]
If we know that $\psi\in  H^\infty_0(S_\omega)$, then the conclusion follows from \cite[Proposition 5.2.4c)]{Ha06} directly. In the case of $\psi$ satisfying \eqref{McI cond}, by McIntosh's approximation result, see \cite[Theorem 5.2.6]{Ha06}, it is sufficient that for all $x \in \cX_{\Phi, \gamma}$, $\psi(t\cA)x\in \overline{D(\cA)\cap \cR(\cA)}$.
Recall that for every $y\in \overline{D(\cA)}$, $\tau(\tau+\cA)^{-1}y \rightarrow y$ as $\tau\rightarrow \infty$.
Since $\psi\in \cE(S_\omega)$, $\overline{D(\cA)}$ is invariant with respect to $\psi(t\cA)$ and by our assumption for every $x\in \overline{D(\cA)}$ we have $\psi(t\cA)x \in \cR(\cA)$, Hence 
\[
D(\cA)\cap \cR(\cA) \ni \psi(t\cA)\tau(\tau+\cA)^{-1}x = \tau(\tau+\cA)^{-1} \psi(t\cA)x \rightarrow \psi(t\cA)x
\]
for all $x\in \overline{D(\cA)}$ as $\tau \rightarrow \infty$. Thus, we get 
$[\cdot]_{\Phi, \psi}\leq c_\beta c_{\psi, \alpha}c_{\alpha, \psi} [\cdot]_{\Phi, \gamma}$ and $\cX_{\Phi, \gamma}\subset \cX_{\Phi, \psi}$. 
 Interchanging $\gamma$ with $\psi$ in the above reasoning completes the proof of (i).

For (ii), note that all functions $\psi_j$ ($j=1,2,3$) satisfy the conditions stated in (A) above. Moreover, the Hardy operator $P$ is bounded on $\Phi$; see Lemma \ref{P ineq eqiv} (in fact, $Q$ too). Therefore, (ii) follows from the result stated on the beginning of this section.   
However, it is more instructive to give direct arguments for some equivalences based on the statement (i). First note that $\psi_3, (\cdot)^{-\theta} \psi_3 \in H^\infty_0(S_\omega)$.   
Further, the proof of $\cX_{\Phi, \psi_j}\subset \overline{D(\cA)}$ for $j=1,2$, does not use the boundedness of $P$ on $\Phi$ and follows directly from the estimates:  for all $x\in \cX$ and $t>0$
\[
\inf_{z\in D(\cA)} \|x-z\|_\cX \leq \|x - e^{-t\cA}x\|_\cX = \|\psi_1(t\cA)x\|_\cX,
\]
\[
\inf_{z\in D(\cA)} \|x-z\|_\cX \leq \|x - (1+t\cA)^{-1}x\|_\cX = \|\psi_2(t\cA)x\|_\cX.
\] Indeed, since $\Phi\subset L^1_{loc}$, the function $\R_+\ni t\mapsto \frac{1}{t}$ is not in $\Phi$; see also Remark \ref{Haase' proof} below. 
Moreover, for $\psi(z) := ze^{z}$, as in the proof of Proposition \ref{Komastu rep}, we can write for all $x\in \cX$ and $t>0$:
\[
\inf_{z\in D(\cA)} \|x-z\|_\cX \leq \|x - e^{-t\cA}x\|_\cX = \Big\|\cA \int_0^t \cT(s)x \ud s \Big\|_\cX  \leq    \int_0^t \Big\|\frac{1}{s}\psi(s\cA)x\Big\|_\cX \ud s.
\] 
Therefore, the boundedness of $P$ on $\Phi$ ensures that $x\in  \overline{D(\cA)}$ if $x\in \cX_{\Phi, \psi}$. 
Furthermore, for all $j = 1,2$ we have that $\psi_j(\cA) \cX \subset \cR(\cA)$. Indeed, for $j=1$, by Lemma \ref{lem on density}, $\psi_1(t\cA) = -\cA\int_0^t\cT(s) \ud s$, $t>0$, and the case $j=2$ is obvious. Note also that  $\psi(\cA) \cX = \cA\cT(1) \cX = \cT(1/2)\cA\cT(1/2) \cX \subset D(\cA) \cap \cR(\cA)$ and $K_\Phi(\cX, D_\cA) = \cX_{\Phi, \psi}$ (as the sets) by Proposition \ref{Komastu rep}. 
Thus, all functions $\psi_j$, $j=1,2,3$ and $\psi$ satisfy the assumption of (i), which implies the first statement of (ii). For the second one, if $\Phi = L^\infty(t)$, then for all $x\in \cX$ we have that 
$[x]_{\Phi, \cA} \leq \sup_{t>0}\|t\cA\cT(t)\|_{\cL(\cX)} \|x\|_\cX$ and 
$[x]_{\Phi, \psi_2} \leq \sup_{t>0}\|I - t(t+\cA)^{-1}\|_{\cL(\cX)} \|x\|_\cX$.
 It easily gives the desired claim.

For (iii), since $\psi$ and $\psi_j$, $j=1,3$, are in $H^\infty_0(\Sigma_\omega)$ and $\theta = 0$, by (i), their corresponding semi-norms are mutually equivalent. 
Suppose that there exists $y\in D(\cA)$ such that $\cA y \neq 0 \in \cX$. Let $ \tilde{y} := \la (\la + \cA)^{-1} y$ for some $\la >0$. Of course, $x:= \cA  \tilde{y} \in D(\cA) \cap \cR(\cA)$ and, since $\cA y \neq 0$,  $x = \la (\la +\cA)^{-1} \cA y \neq 0$ too. For some $n\in \N$ with $n>p$, where $p\in [1,\infty)$ is fixed,  and all $\la >0$ set 
\[
x_{\la}: = \cA^n (\la + \cA)^{-n} (1+\la \cA)^{-n} x
\]
Since $x\in  \cR(\cA)$, we have that $x_{\la} \rightarrow x $ as $\la \rightarrow 0^+$. Hence, for some $\la_0>0$ small enough $x_{\la_0} =: x_0 \neq 0$.

It is readily seen that $x_0 \in \cX_{\Phi, \psi}$, where $\Phi = L^p(\R_+, t^{p-1} \ud t)$, $p \in [1,\infty)$. Indeed, denoting $y_0:= (\la_0 + \cA)^{-n}(1+\la_0 \cA)^{-n} x_0$, i.e. $x_0 = \cA^n y_0$, we get 
\begin{align*}
[x_0]_{\Phi, \psi}^p  &: = \int_0^\infty \|\cA \cT (t)x_0\|_\cX^p\, t^{p-1} \ud t\\
& = \int_0^1 \|\cT(t) \cA x_0\|^p_\cX \, t^{p-1} \ud t + \int_1^\infty \|\cA^{n+1} \cT(t) y_0 \|_\cX \, t^{p-1} \ud t.   
\end{align*}
Since $\sup_{t>0} \| t^n \cA^{n} \cT\c(t)\|_{\cL(\cX)} <\infty $ and $n> p$, we have that $x_0 \in \cX_{\Phi,\psi}$. We show that $x\notin \cX_{\Phi, \psi_2}$. 
In fact, if $x_0 \in   \cX_{\Phi,\psi_2}$, then 
\[
[x_0]_{\Phi,\psi_2}^p = \int_0^\infty \|\cA (1+ t\cA)^{-1}x_0 \|_\cX^p\, t^{p-1} \ud t
= \int_0^\infty\|\cA(t+\cA)^{-1} x_0 \|_\cX^p  \, \frac{\ud t}{t} < \infty.
\]
In particular, 
\[
\liminf_{t\rightarrow 0\plus}\|\cA(t+\cA)^{-1} x_0\|_\cX = 0
\]

However, since for all $\tilde{x} \in \overline{\cR(\cA)}$ we have $\cA(t +\cA)^{-1}\tilde{x} \rightarrow \tilde{x} $ as $t\rightarrow 0\plus$, and since $x_0\in \cR(\cA)$, we get $x_0 = 0$, which contradicts our assumption. It completes the proof of (iii).  \end{proof}

\begin{remark}\label{Haase' proof}
(a) We first remark on the condition \eqref{McI cond} of Proposition \ref{Haase equi}(i). In general, for $\Phi\in L^1_{loc}$, if $\psi\in \cE(S_\omega)$ is such that $\phi : = \psi + 1$ is holomorphic at $0$ and decays at $\infty$ as $|z|^{-\epsilon}$ ($z\in S_\omega$) for some $\epsilon>0$,  then $\cX_{\Phi, \psi} \subset  \overline{D(\cA)}$. 
Indeed, then $\psi(tA) = I - \phi(t\cA)$ and $\phi(t\cA)\cX \subset  \overline{D(\cA)}$, since $\phi(t\cA)$ is represented by the absolutely convergent Cauchy integral $\frac{1}{2\pi i}\int_\Gamma \phi(z) R(z,\cA) \ud z$ with $\Gamma := \partial(S_\omega \cup B)$, where $B$ is a ball with the center at $0$. 

(b) Note also that if $\cA$ is densely defined and $z^{-1} \psi$, $z^{-1} \gamma \in \cE(S_\omega)$, then \eqref{McI cond} holds for any $\psi, \gamma \in \cE(S_\omega)$. Indeed, recall that $\psi(\cA) = \cA (z^{-1}\psi)(\cA)$ and similarly for $\gamma$; see e.g. \cite[Theorem 1.3.2]{Ha06}. Therefore, in such case,  Proposition \ref{Haase equi}(i) does not require the boundedness of Hardy's operator $P$ on $\Phi$ as the assumptions of \cite[Theorems 3.1 and 3.4]{ChKr16}; see (A) and (B) at the beginning of this subsection.

\end{remark}

We conclude this subsection with a noteworthy implication of Proposition \ref{Haase equi} to the Kalton-Portal characterization of the $L^1$-m.r. and $L^\infty$-m.r. on $\R_+$ \cite[Theorems 3.6 and 3.7]{KaPo08}; see also Theorems \ref{weighted L1 mr} and \ref{charact}. 
The presented below (equivalent) characterization conditions are expressed in {\it a priori} terms of a generator, which may be more useful in applications than those involving semigroup operators. 

\begin{proposition}\label{equiv for l1} Let $-A$ be the generator of a bounded analytic semigroup on a Banach space $X$. 
Let $0\neq \gamma \in H^\infty_0(S_\omega)$ for some $\omega \in (\omega_A, \frac{\pi}{2})$, where $\omega_A$ denotes the sectoriality angle of $A$.

Then, the following assertions are true:
\begin{itemize}
\item [(i)] The operator $A$ has $L^1$-m.r. on $\R_+$ if and only if there exists a constant $C$  such that for every $x\in X$ 
\begin{equation}\label{cond equiv}
\int_0^\infty \Big\|\frac{1}{t}\gamma(tA) x\Big\|_X \ud t \leq C \|x\|_X.
\end{equation}
 In particular, as $\gamma$ one can take $\gamma(z) : = z(1+z)^{-1-\epsilon}$ for any~$\epsilon>0$. 
 
Moreover, if there exists a constant $C$ such that for every $x\in X$ 
\[
\int_0^\infty \|A(1+tA)^{-1}x\|_X \ud t \leq C \|x\|_X,
\] then $A = 0\in \cL(X)$.

\item [(ii)] The operator $A$ has $L^\infty$-m.r. on $\R_+$ if and only if there exists a constant $C$ such that for every $x\in X$  

\[
\|x\|_X\leq C \sup_{t>0} \|\gamma(tA) x\|_X + \limsup_{t\rightarrow \infty}\|\cT(t)x\|_X. 
\]
\end{itemize}
\end{proposition}
\begin{proof}  Let $\psi(z) := ze^{-z}$ and $\eta(z):= z(1+z)^{-1-\epsilon}$ for $z\in S_\omega$ with $\omega>\omega_A$. Obviously, $\psi, \eta, \gamma \in H^\infty_0(S_\omega)$. Therefore, taking $\Phi := L^1$ (resp. $\Phi := L^\infty(t)$), which corresponds to $\theta = 0$, by Proposition \ref{Haase equi}(i) and \cite[Theorems 3.6 and 3.7]{KaPo08}, we get the desired claim of (i) (resp. of (ii)).

The proof of the additional statement of (i) follows the idea of the proof of Proposition \ref{Haase equi}(iii) above. Indeed, following \eqref{cX spaces} and \eqref{psi funct},  note that $X_{ \Phi, \psi_2, A} = X$ with $\Phi = L^1$. For all $x\in X$ we have
\begin{align*}
\int_0^\infty\|A(1+tA)^{-1}x\|_X \ud t & = \int_0^\infty\|A(t+A)^{-1}x\|_X \frac{\ud t}{t} \\
& = \int_0^\infty \|x - t(t+A)^{-1}x\|_X \frac{\ud t}{t}<\infty. 
\end{align*}
In particular, for all $x\in X$ 
\[
\liminf_{t\rightarrow 0\plus}\|A(t+A)^{-1}x\|_\cX = 
\liminf_{t\rightarrow \infty}\|x - t(t+A)^{-1}x\|_\cX = 0.
\]
It shows that $\overline{D(A)} = X$ and, in particular, for any $x\in \cR(A) \subset X$, $\liminf_{t\rightarrow 0\plus}\|A(t+A)^{-1}x\|_X = 0$. But, for $x = Ay$ with $y \in D(A)$, we have 
\[
A(t+ A)^{-1} x = x + t y + t^2 (t+A)^{-1} y  \rightarrow x \textrm{ as } t\rightarrow 0\plus.
\]
Therefore $\cR(A) = \{0\}$, i.e. $A = 0$. The proof is complete. 
\end{proof}

We conclude with a result, see Theorem \ref{optimal for l1} below, which should be carefully confronted with Theorem \ref{thm main}.  We start with an auxiliary lemma. 

\begin{lemma}\label{equi norms}
Let $-\cA$ be the generator of a bound analytic semigroup on a Banach space $\cX$. 
Let $\theta \in \R$ and $\omega \in (\omega_\cA, \frac{\pi}{2})$.
 Let $\psi \in \cE(S_\omega)$ with $(\cdot)^{-\theta} \psi \in H^\infty_0(S_\omega)$ satisfy $\cX_{\Phi, \psi} \subset \overline{D(\cA)}$ and $\psi (\cA)\overline{D(\cA)} \subset \cR(\cA)$ or $\psi\in H^\infty_0(S_\omega)$.
Let $\Phi := L^1(\R_+, t^{-\theta} \ud t)$. 
Finally,  let $X$ stand for $\cX_{\Phi,\psi}$ equipped with  $\|x\|_X: = \|x\|_{\cX} + [x]_{\Phi,\psi}$. 

Then, the following statements are true.
\begin{itemize} 
\item [(i)] The space $(X, \|\cdot\|_X)$ is a Banach space and the part of $-\cA$ in $X$ is the generator of an analytic semigroup. 

\item [(ii)] For all $x\in \cX_{\Phi,\psi}$
\[
\int^\infty_0 [\cA \cT(t) x]_{\Phi,\psi} \ud t \lesssim [x]_{\Phi, \psi}. 
\] 

\end{itemize}
\end{lemma}

\begin{proof}
The statement (i) follows by standard arguments. We omit its proof. 

For (ii), let $\gamma = \phi \psi$ for some  $\phi \in H^\infty_0(S_\omega)$ and let $\eta(z): = ze^{-z}$ $(z\in \C)$. 
Of course, $\gamma \in H^\infty_0(S_\omega)$ and  $(\cdot)^{-\theta} \gamma \in H^\infty_0(S_\omega)$. 
By Proposition \ref{Haase equi}, the semi-norms $[\cdot]_{\Phi, \psi}$ and $[\cdot]_{\Phi,\gamma}$ are equivalent on $\cX$.

Moreover, since $\gamma(t\cA) = \phi(t\cA) \psi (t\cA)$ and $\phi, \eta \in H^\infty_0(S_\omega)$, by \cite[Theorem 5.2.2b)]{Ha06}, we get 
\[
\kappa:= \sup_{s>0} \int_{\R_+}\|\phi(s\cA) \eta(t\cA)\|_{\cL(\cX)} \frac{\ud t}{t}< \infty.
\]
Therefore, writing $\frac{1}{t}\eta(t\cA) = \cA \cT(t)$, we have that
\begin{align*}
\int_{\R_+} [{t}^{-1}\eta(t\cA) x]_{\Phi,\gamma} \ud t & = \int_{\R_+}\int_{\R_+} \|{s}^{-1}\gamma(s\cA) {t}^{-1}\eta(t\cA) x\|_\cX s^{-\theta} \ud s \ud t\\
& \leq  \int_{\R_+}\int_{\R_+} \|\phi(s\cA) {t}^{-1}\eta(t\cA)\|_{\cL(\cX)}  \|{s}^{-1} \psi (s\cA)x\|_\cX s^{-\theta} \ud s \ud t\\
& \leq \kappa \int_{\R_+} \|{s}^{-1} \psi (s\cA)x\|_\cX s^{-\theta} \ud s \\
& \simeq \, \kappa \,[x]_{\Phi,\psi}.
\end{align*}
By Proposition \ref{Haase equi}(i), $[\cdot]_{\Phi, \psi}\lesssim [\cdot]_{\Phi,\gamma}$ on $X_{\Phi,\psi} = X_{\Phi,\gamma}$. It finishes the proof of~(ii). 
 
\end{proof}

\begin{theorem}\label{optimal for l1} 
Let $-\cA$ be the generator of a bounded, analytic semigroup on a Banach space $\cX$. 
Let  $0 \neq \psi \in H^\infty_0(S_\omega)$, where $\omega \in (\omega_\cA,\frac{\pi}{2})$.

 Let 
\[
X := \cX_{1,\cA}: = \Big\{ x \in \cX: \int_0^\infty \|{t^{-1}}\psi(t\cA) x \|_\cX \ud t < \infty \Big\}
\]
be equipped with the norm 
\[
\|x\|_X: = \|x\|_{\cX} + \int_0^\infty \|{t^{-1}}\psi(t\cA) x \|_\cX \ud t.
\]
  Then, the part of $\cA$ in $X$ has $L^1$-m.r. on $\R_+$. 
  
  Moreover, $D(\cA)\cap \cR(\cA)\subset X$. In particular, if $D(\cA)$ and $\cR(\cA)$ are dense in $\cX$, then $\cX_{1,\cA}$ is. 
  
\end{theorem}

\begin{proof} (i) By Lemma \ref{equi norms}(i) we already know that $X$ is a Banach space and the part of $-\cA$ in $X$ is the generator of a bounded, analytic semigroup $T$ on $X$. Therefore, it suffices to show that $A$ satisfies 
\eqref{cond equiv} of Proposition \ref{equiv for l1}. 

Since $\eta(z):= ze^{-z}$ $(z\in \C)$ is in $H^\infty_0(S_\omega)$, by \cite[Theorem 6.4.2]{Ha06}, the semi-norms $[\cdot]_{\Phi, \eta}$ and $[\cdot]_{\Phi, \psi}$ are equivalent. Here  $\Phi := L^1$ (i.e. $\theta = 0$). Therefore, by Lemma \ref{equi norms}(ii),  for all $x\in X$
\begin{align*}
\int_0^\infty \|AT(t)x\|_X \ud t & = \int_0^\infty \|AT(t)x\|_\cX \ud t  + \int_0^\infty [AT(t)x]_{\Phi, \psi} \ud t\\
& \lesssim   [x]_{\Phi,\eta} + [x]_{\Phi,\psi} \lesssim \|x\|_X. 
\end{align*}

For the last statement note that for $x\in D(\cA)$ such that $x = \cA y$ for some $y\in D(\cA)$, we have for all $t>0$ that 
\[
\|\cA\cT(t)x\|_\cX \leq \min \left(\|\cT(t)\cA x\|_\cX, \|\cA^2 \cT(t) y\|_\cX \right) \lesssim \min(1, t^{-2}).
\]
Therefore, 
 $D(\cA)\cap \cR(\cA)\subset \cX_{\Phi, \eta}=\cX_{\Phi,\psi} = X$. Finally, recall that the density of $D(A)$ and $\cR(A)$ implies the density of $D(\cA)\cap \cR(\cA)$; see, e.g. \cite[Theorem 3.1.2(iv), p. 91]{PrSi16}. It completes the proof.  
\end{proof}

\begin{remark}\label{rem on X1}
(a) One can show that $\cX_{1,\cA} \subset J_{0,1}(\cX, D_\cA)$. Since $J_{0,1}(\cX, D_\cA)$ coincides with the closure of $D(\cA)$ in $\cX$ (cf. \cite[Theorem 3.4.2]{BeLo76}), in general, this inclusion is proper. 
Furthermore, the following example shows that, in general, $(\cX, D_\cA)_{\theta, 1} \nsubseteq  \cX_{1,\cA}$ for all $\theta \in (0,1)$.

Let $\cX  := L^1(\R)$ and $\cA := - \Delta$ be the negative of the distributional realization of the Laplace operator on $\cX$. In this case, the space $\cX_{1,\cA}$  is contained in the subspace $\{h \in L^1(\R) : \cF{h}(0) = 0\}$ of $L^1(\R)$, where $\cF h$ denotes the Fourier transform of $h$. 
Indeed, suppose that $h \in L^1(\R)$ is such that $\cF{h}(0) \neq 0$. Since $\cF{h}$ is continuous at zero, there exists $\delta > 0$  such that $|\cF{h}(t)| \geq \frac{1}{2} |\cF{h}(0)|$ for all $t \in (0,\delta)$. Hence, the integral $\int_0^{\infty} \frac{|\cF{h}(t)|}{t}\, \ud t$ diverges. We will show that this implies $\int_0^{\infty} \|\cA \cT(t)h\|_{\cX} \ud t = \infty$, where $\cT$ stands for a semigroup generated by $-\cA$. Note that 
\[
\cF(\cA \cT(t) h) (\xi) = |\xi|^2 e^{-t|\xi|^2} \cF h(\xi), \qquad \xi \in \R.
\]
Therefore, we get
\begin{align*}
\int_0^{\infty} \|\cA \cT (t)h\|_{\cX} \ud t & \geq \int_0^{\infty} \sup_{\xi \in \R} |\xi|^2 e^{-t|\xi|^2} |\cF{h}(\xi)|  \ud t\\
& \geq \int_0^{\infty} \frac{1}{t} \left|\cF{h}\left(\frac{1}{\sqrt{t}}\right)\right|  \ud t\\
& \approx \int_0^{\infty} \frac{|\cF{h}(t)|}{t}  \ud t.
\end{align*}
It proves our claim.

(b) As regarding the additional statement of Theorem \ref{optimal for l1}, recall that for densely defined, injective sectorial operators $\cA$ on a reflexive Banach space $\cX$, $D(A) \cap \cR(\cA)$ is dense in $\cX$; see, e.g. \cite[Theorem 3.1.2(iii), p.91]{PrSi16}.
\end{remark}

\subsection{The density of the domain in interpolation spaces}\label{subsect denisty}

Our next result clarifies the density of the domain of the part $A$ of $\cA$ in $K_\Phi(\cX, D_\cA)$ for some parameters $\Phi$.
  
\begin{lemma}\label{dense dom} Let $\cA$, $\cT$ and $\Phi$ be as specified in the previous lemma.

Then, the part $A$ of $\cA$ in $K_\Phi(\cX,D_\cA)$ is again the generator of a bounded analytic semigroup $T$ on $K_\Phi(\cX,D_\cA)$, which is the restriction of $\cT$ to $K_\Phi(\cX,D_\cA)$. Moreover, the following assertions hold.
\begin{itemize}
\item [(i)] If, in addition,  $\Phi$ has absolutely continuous norm, then the domain of $A$ is dense in $K_\Phi(\cX, D_\cA)$, i.e. the semigroup $T$ generated by $A$ is strongly continuous at $0$. 

In particular, it applies to $\Phi = L^q(\R_+, t^{q(1-\theta)-1}\ud t)$ with $\theta\in (0,1)$ and $q\in [1,\infty)$, which corresponds to the real interpolation space $K_\Phi(\cX,D_\cA) = (\cX,D_\cA)_{\theta, q}$.\\

\item [(ii)] If, in addition, $\cA$ is unbounded on $\cX$ and $\Phi = L^\infty(t^{-\theta})$ with $\theta \in (0,1)$, then its part $A$ in $K_\Phi(\cX, D_\cA) = (\cX, D_\cA)_{\theta, \infty}$  
is not densely defined. 
\end{itemize}
\end{lemma}

\begin{proof} Let $X: = K_\Phi(\cX, D_\cA)$. 
Since $\rho(\cA)\subset \rho(A)$ with $R(\lambda, A) = R(\lambda, \cA)_{|X}$ $(\lambda \in  \rho(\cA))$, by the representation of the space $K_\Phi(\cX, D_\cA)$ stated in Proposition \ref{Komastu rep}, $-A$ is the generator of the semigroup $T(t): = \cT(t)_{|X}$ ($t\geq 0)$ on $X$.\\

Further note that our assumption on the boundedness of $P$ on $\Phi$ implies that $\Phi\subset L^1_{loc}$. It gives that $K_\Phi(\cX,D_\cA)  =  K_\Phi(\cY,D_\cA)$, where $\cY :=  \overline{D(\cA)}^{\cX}$. Indeed, if there exists $x\in \cX \setminus \cY$, then $\frac{1}{t} K(t, x, \cX,D_\cA) \geq \frac{1}{t}\textrm{dist}(x,\cY)$ for $t\in (0,1)$. Moreover, for each $x\in \cY$, $K(t,x, \cX ,D_\cA) = K(t,x, \cY ,D_\cA)$, $t\geq 0$. Consequently, $X\subset \cY$.  

We show that for every $x\in X$, the vectors $x_n:=n(n+\cA)^{-1}x\in D(\cA)$, $n\in \N$, are in the domain of $A$ and approximate $x$ in the norm of $X$.
For the first statement, just note that $\cA x_n = nx -n^2(n+\cA)^{-1}x \in X$. For the second one, note that for all $t>0$ and $n\in \N$:
\[
\|\cA\cT(t)(x_n - x)\|_\cX \leq \sup_{n\in \N} \|n(n+\cA)^{-1} - I\|_{\cL(\cX)}  \|\cA\cT(t) x\|_\cX.  
\]
Hence, since $x\in X \subset \cY$,  $x_n - x = -\cA (n+\cA)^{-1} x$ converges to $0$ in $\cX$ as $n\rightarrow \infty$, by \eqref{equi norms} and the absolute continuity of the norm of $\Phi$, $x_n$ approximate $x$ in the norm of $X$. It completes the proof of (i).

For (ii), recall that for any Banach couple $(\cY, \mathcal Z)$ such that $\mathcal Z$ is continuously embedded in $\cY$, but not closed in $\cY$, the real interpolation spaces 
$(\cY, \mathcal Z)_{\theta, \infty}$ are not {\it regular}, that is,  $\mathcal Z\subset (\cY, \mathcal Z)_{\theta, \infty}$ is not a dense subset of  $(\cY, \mathcal{Z})_{\theta, \infty}$ ($\theta \in (0,1)$), cf. for instance  \cite[Theorem 3.4.2(c)]{BeLo76}. 

Specifying it to our context, the set $D(\cA)$ (hence also $D(A)$) is not dense in $(\cX, D_\cA)_{\theta, \infty}$.
\end{proof}

\section{Extension of Da Prato - Grisvard theorem}\label{sect DaP-G}

The main results of this section, Theorems \ref{thm main} and \ref{thm main 2}, identify those Banach function spaces $\Phi$ over $\R_+$ (in other words, the parameters of the $K$-method) such that, for any generator $-\cA$ of a bounded analytic semigroup on a Banach space $\cX$, the part $A$ of $\cA$ in the corresponding interpolation space $X = K_\Phi(\cX, D_\cA)$ satisfies the $L^1$- or $L^\infty$-m.r. type of estimates, including weighted ones. Even for $\Phi$ corresponding to the classical real interpolation spaces $(\cX, D_\cA)_{\theta, r}$ ($\theta \in (0,1)$, $r=1, \infty$), our results extend, in several ways, Da Prato-Grisvard type results from the literature. In particular, our assumptions on the force term $f$ in $(CP)_{A,f,x}$ to derive the homogeneous estimates \eqref{hom est1} or \eqref{hom est oo} or \eqref{E mr est} should be compared with those in \cite[Proposition 2.14]{DaHiMuTo}; see also Section \ref{sec app} for further connections to related results in the literature. Moreover, in \ref{thm main}(iii) we provide a natural example of a non-invertible generator of a bounded, analytic semigroup on a real interpolation space, which does not have $L^1$-m.r. on $\R_+$. 

\subsection{An extension of the Da Prato-Grisvard theorem for $L^1$-m.r.}

Recall that $[\cdot]_{\Phi, \cA}$ stands for the homogenous part of the norm of the interpolation space $K_\Phi(\cX, D_\cA)$; see \eqref{homoge part}.

\begin{theorem}\label{thm main}
Let $-\cA$ be the generator of a bounded analytic semigroup $\cT$ on a Banach space $\cX$.  Let $w$ be a weight function on $(0,\infty)$ such 
that $[w]_{P+Q, L^1} < \infty$ $($see, \eqref{P+Q on L1}$)$. 

Denote by $A$ the part of $\cA$ in $X:= K_\Phi(\cX, D_\cA )$,  where $\Phi := L^1(\R_+, w \ud t)$.

Then, $-A$ is the generator of a bounded analytic $C_0$-semigroup $T$ on $X$ and the following assertions hold.

\begin{itemize}

\item [{\rm(i)}]  For any $f\in L^1_{loc}([0,\infty); X)$ and $x\in X$ the function 
\begin{equation}\label{u funct}
u (t):= T(t)x +  \int_0^tT(t-s) f(s) \ud s,  \qquad t\geq 0,
\end{equation}
is a unique strong solution to $(CP)_{A,f,x}$.\\ 
 
\item [{\rm(ii)}]  The operator $A$ on $X$ has $L^1$-m.r.\,\,on finite intervals. Moreover, for any $\tau\in (0,\infty)$ there exists a constant $C_\tau>0$   such that for any $f \in L^1_{loc}([0,\infty); X)$ and $x\in X$ the corresponding strong solution $u$ to $(CP)_{A,f,x}$ satisfies 
\begin{equation}\label{inhom l1}
\| u'\|_{L^1(0,\tau;X)} + \| A u  \|_{L^1(0,\tau;X)} + \| u\|_{{BUC}([0,\tau]; X)} \leq C_\tau \left(  \| f \|_{L^1(0,\tau;X)} + \|x\|_X\right). 
\end{equation}

 Furthermore, the following homogenous estimates hold on $\R_+${\rm:}
there exists a constant $C>0$ such that for every $f\in L^1_{loc}([0,\infty); X)$ with $[f(\cdot)]_{\Phi, \cA} \in L^1(\R_+)$ and every $x\in X$, the  strong solution $u$ to $(CP)_{A,f,x}$ satisfies 
\begin{align}\label{hom est1}
\nonumber \|\,[u']_{\Phi, \cA}\|_{L^1(\R_+)} + \|\, [Au]_{\Phi, \cA} \|_{L^1(\R_+)} & + \|[u]_{\Phi, \cA}\|_{{BUC}(\overline{\R}_+) }\\
& \leq C \left( \|\, [f]_{\Phi,\cA}\|_{L^1(\R_+)} + [x]_{\Phi,\cA} \right), 
\end{align} 
where $[\cdot]_{\Phi, \cA}$ stands for the homogeneous part of the norm $\|\cdot\|_X$ of $X$.

\item [{\rm(iii)}] 
In general, $A$ does not have $L^1$-m.r.\,on $\R_+$. More precisely,
 the distributional  realization of the negative Laplacian $- \Delta$ on the classical Besov spaces $X = B^{2\theta}_{1,1}(\R^n) = (L^1(\R^n), W^{2,1}(\R^n))_{\theta,1}$, $\theta \in (0,1)$, does not have $L^1$-m.r.\,on $\R_+$. 

\item [{\rm(iv)}] If, in addition, $0\in \rho(\cA)$, 
then $A$ on $X$ has $L^1$-m.r.\,on $\R_+$ and, furthermore, the solution operator $U$ defined in \eqref{op U} is bounded on $L^1(\R_+; X)$.

 Consequently, for every $f\in L^1(\R_+;X)$ and $x\in X$ the corresponding solution $u$ satisfies \eqref{hom est1} with the norm $\|\cdot\|_X$ instead of the seminorm $[\cdot]_{\Phi, \cA}$ and, in addition, 
\[
 \| u \|_{L^1(\R_+;X)}\lesssim \|f \|_{L^1(\R_+;X)} + \|x\|_X. 
\]

\item [{\rm(v)}] The above statements $(ii)$ and $(iv)$ $($respectively, $(iii))$ is true if $($unweighted$)$ $L^1$-m.r. is replaced by 
$($weighted$)$ $L^1_v$-m.r. for every non-increasing weight $v$ $($respectively, for every $v$
 which satisfies the assumptions of Proposition $\ref{L1 extrapol var 2})$.

\item [{\rm(vi)}] If $v(t): = e^{-t}$, $t>0$, then $A$ on $X$ has $L^1_v$-m.r. on $\R_+$. 
\end{itemize}
\end{theorem}

\begin{proof}
Combining Lemma \ref{P+Q ineq}(i), Proposition \ref{Komastu rep} and Lemma \ref{dense dom}(i) we get that 
$-A$ is the generator of a bounded analytic $C_0$-semigroup $T$ on $X$ and  $T(t)$ is the restriction of $\cT(t)$ to $X$ ($t\geq0$).  

We begin with the proof of the statement (iv), which  we then apply, by the a rescaling argument, in the proofs of (i) and (ii).

For (iv), by \eqref{P on l1} and Lemma \ref{P+Q ineq}(i), in particular,  the Hardy operator $P$ is bounded on $\Phi$, and Proposition \ref{Komastu rep} shows that,  up to the equivalence of the norm, we have the following Komatsu type representation of $X:= K_\Phi(\cX, D_\cA )$:
$$
X = \left \{ x\in \cX:   \left[ (0,\infty)\ni t\mapsto \|\cA \cT(t) x \|_{\cX} \right]\in \Phi \right\}
$$
and 
\[
\|x\|_X \simeq [x]_{\Phi,\cA} :=  \int_0^\infty \|\cA \cT(s)  x \|_\cX w(s) \ud s, \qquad x\in X.
\]
Moreover, there exists a constant $C_\cA$ such that for all $t,s>0$ we have that
\[
\| \cA \cT(t/2) \cT(s) \|_{\cL(\cX)} = \|\cA \cT((t+s)/{2}) \cT({s}/{2})\|_{\cL(\cX)}  \leq  \frac{C_\cA}{t+s}. 
\] 

It gives that for all $x\in X$ 
\begin{align*}
\int_0^\infty [ A T(t) x]_{\Phi, \cA} \ud t & = \int_0^\infty \int_0^\infty \left\|\cA^2 \cT(s+t) x\right\|_\cX w(s) \ud s \ud t \\
& \leq C_\cA 
 \int_0^\infty \int_0^\infty \frac{1}{t+s}  \|\cA \cT({t}/{2}) x\|_\cX w(s) \ud s \ud t \\
 &  =  C_\cA  \int_0^\infty \left[\frac{1}{w(t/2)}\int_0^\infty \frac{1}{t+s} w(s) \ud s \right] \|\cA \cT(t/2) x\|_\cX w(t/2) \ud t\\
 & \leq 2 C_\cA [w]_{P+Q, L^1}\, [x]_{\Phi, \cA}.
\end{align*} For the further references we point out here that we do not make use of the invertibility of $\cA$ in the above estimate. 
The fact that $0\in \rho(\cA)$ yields $\|\cdot\|_X \simeq [\cdot]_{\Phi,\cA}$, and  the Kalton-Portal characterization, see Theorem \ref{weighted L1 mr}(with $v\equiv 1$), implies that $A$ has $L^1$-m.r. on $\R_+$. The boundedness of the corresponding solution operator $U = U_A$ follows from Dore's result \cite[Theorem 5.2]{Do00}; see also Remark \ref{Dore results}(c). 
It completes the proof of the point (iv). 

For (ii), note that for all $\lambda>0$, the operator $\cA+\lambda$ is the generator of a bounded analytic semigroup and $0\in \rho(\cA+\lambda)$. Moreover, 
since $\|x\|_{D_\cA} \simeq \|x\|_{D_{\cA+\lambda}}$ for $x\in D(\cA) = D(\cA+\lambda)$, we have that $X := K_\Phi(\cX,D_\cA) = K_\Phi(\cX,D_{\cA+\lambda})$ with the equivalence of the norms. By (iv) we know that $(\cA + \lambda)_{|X} = \cA_{|X} + \lambda$ has $L^1$-m.r. on $\R_+$, which easily gives that $A$ has $L^1$-m.r. on finite intervals; cf. the proof of Lemma \ref{lem on hom est}(ii). In particular, the condition \eqref{loc est} with $p=1$ holds; cf. Remark \ref{Dore results}(b). Therefore, since $D(A)$ is dense in $X$ (see Lemma \ref{dense dom}(i)), by Lemma \ref{regularity of u}(ii) and (iii), we get the statement (i).  The estimates \eqref{inhom l1} of (ii) follow easily.

 For the proof of the homogenous estimates \eqref{hom est1}, as has been stressed in the proof of (iv) above, for all $x \in X$ we have  
\begin{equation}\label{l1 homo}
\int_0^\infty [A T(t) x]_{\Phi,\cA} \ud t  \leq C [x]_{\Phi,\cA}
\end{equation}
 for a constant $C$ independent of $x$. 
Fix $f\in L^1_{loc}([0,\infty);X)$ with $[f]_{\Phi, \cA} \in L^1$. Since $A$ has $L^1$-m.r. on finite intervals, 
Lemma \ref{lem on hom est}(ii) shows that for a.e. $t>0$ the function $(0,t) \ni s\mapsto A T(t-s) f(s)\in X$ is  Bochner integrable. 
  For such $t$ one can find a sequence of simple functions $g_n: (0, t)\rightarrow X$, $n\in \N$, such that $g_n(s)\rightarrow A T(t-s) f(s)$ as $n\rightarrow \infty$ and 
\begin{equation*}
\|g_n(s)\|_X\leq 2 \|A T(t-s) f(s)\|_X \qquad  \textrm{for a.e. } s \in (0,t).
\end{equation*}
In particular, by the (scalar) Lebesgue's dominated convergence theorem,
\[
\int_0^t \|g_n(s) -  A T(t-s) f(s)\ud s\|_X \ud s \rightarrow 0  \qquad \textrm{ as } n\rightarrow \infty.
\]
In addition, combining again Lebesgue's dominated theorem with Hille's thoerem, $\|\int_0^t g_n(s) \ud s - A U f(t)\|_X\rightarrow 0$ as $n\rightarrow \infty$. 
 Since for every simple function $g_n$,  $\bigl[\int_0^t g_n(s) \ud s\bigr]_{\Phi, \cA} = \int_0^t [g(s)]_{\Phi, \cA} \ud s$,  for a.e. $t>0$ we get that  
\begin{align*}
[AU f(t)]_{\Phi, \cA} &\leq \bigl[ A \int_0^t T(t-s) f(s)\ud s - \int_0^t g_n(s) \ud s \bigr]_{\Phi, \cA} + \bigl[\int_0^t g_n(s) \ud s \bigr]_{\Phi, \cA}\\
& =  \bigl[ A \int_0^t T(t-s) f(s)\ud s - \int_0^t g_n(s) \ud s \bigr]_{\Phi, \cA} + \int_0^t [g_n(s)]_{\Phi, \cA} \ud s \\
& \leq \bigl[ A \int_0^t T(t-s) f(s)\ud s - \int_0^t g_n(s) \ud s \bigr]_{\Phi, \cA}\\
& \quad  + \int_0^t [g_n(s) -  A T(t-s) f(s)\ud s]_{\Phi, \cA} \ud s
  +  \int_0^t [A T(t-s) f(s)]_{\Phi, \cA}\ud s\\
\end{align*}
Therefore, letting $n\rightarrow \infty$, we obtain  
\begin{equation}\label{hom est2}
[AU f(t)]_{\Phi, \cA} \leq \int_0^t [AT(t-s) f(s)]_{\Phi, \cA} \ud s \qquad (\textrm{ a.e. } t>0).
\end{equation}
Now, combining the (scalar) Fubini theorem with \eqref{l1 homo} and \eqref{hom est2}, we get  
\begin{align*}
\int_{\R_+} [AU f(t)]_{\Phi, \cA} \ud t & \leq  \int_{\R_+} \int_0^t [A T(t-s) f(s)]_{\Phi, \cA} \ud s \ud t\\
& = \int_{\R_+} \int_s^\infty [AT(t-s) f(s)]_{\Phi, \cA} \ud t \ud s\\
& =  \int_{\R_+} \int_0^\infty [AT(t) f(s)]_{\Phi, \cA} \ud t \ud s\\
& \leq C  \int_{\R_+}  [f(s)]_{\Phi, \cA} \ud s.
\end{align*}
Finally, to show that $[U f]_{\Phi, \cA}$ is bounded, uniformly continuous on $\R_+$ with the bound on its sup norm as in \eqref{hom est1}, 
 first note that for all $t>s>0$ we have the integral representation formula: 
\[
U f(t) - U f(s) = \int_s^t (U f)'(r)\ud r,
\]
where the integral exists in $X$. Since $(U f)' = f - AU f$ is locally Bochner integrable on $[0,\infty)$, the same argument that we used to get \eqref{hom est2}, shows that for all $t>s>0$
\[
\big|\, [U f(t)]_{\Phi, \cA} - [U f(s)]_{\Phi, \cA} \,\big| \leq [ U f(t) - U f(s)]_{\Phi, \cA} \leq \int_s^t [(U f)'(r)]_{\Phi, \cA}\ud r.  
\]

Therefore, since $[(U f)']_{\Phi, \cA} \leq  [f]_{\Phi, \cA} + [AU f]_{\Phi, \cA}\in L^1$ and $U f (0) = 0$, we get that $U f$ is bounded and uniformly continuous with respect to the semi-norm $[\cdot]_{\Phi, \cA}$, with 
$\sup_{t>0} [ U f(t)]_{\Phi, \cA} \leq C  \|[f]_{\Phi, \cA}\|_{L^1(\R_+)}.$
It completes the proof of \eqref{hom est1} for the initial value $x=0$.

For a non-zero initial value $x\in X$, if $u$ is the corresponding strong solution to $(CP)_{A,f,x}$, then  $Au(t) = AT(t) x + AU f(t)$ for a.e. $t>0$. By \eqref{l1 homo} we get $[Au]_{\Phi, \cA} \in L^1$. Since $u' = f - Au$ a.e. on $\R_+$, the same argument to the above gives \eqref{hom est1}. 
It completes the proof of (ii).

For (iii), consider a function $\phi(\xi) := e^{\frac{-|\xi|^2}{2}}$ $(\xi \in \R^n)$.  Since $\phi$ is a Schwartz function,  $\phi \in B^{2\theta}_{1,1}(\R^n)$. To prove the lack of $L^1$-m.r. on $\R_+$ for the negative Laplacian on $B^{2\theta}_{1,1}(\R^n)$, by the Kalton-Portal characterization, it is sufficient to show that $\int_0^{\infty} \|\Delta e^{t\Delta}\phi\|_{B^{2\theta}_{1,1}(\R^n)} \ud t = \infty$.  First, note that for $\xi \in \R^n$ and $t \in (0,\infty)$ we have:
\[
(e^{t\Delta}\phi)(\xi) = \frac{1}{(t+1)^{\frac{n}{2}}} e^{\frac{-|\xi|^2}{2(t+1)}}.
\]
Since $e^{t\Delta}\phi$ is a radial function, we have $\Delta e^{t\Delta} \phi (\xi) = \frac{1}{r^{n-1}} \frac{\ud}{\ud r}\left((\cdot)^{n-1}(e^{t\Delta}\phi)'\right)$ with $r = |\xi|$. Therefore
\[
(\Delta e^{t\Delta}\phi)(\xi) = \frac{1}{(t+1)^{\frac{n}{2}+1}} e^{\frac{-r^2}{2(t+1)}} \left(1-\frac{r^2}{t+1}\right),  \qquad r = |\xi|.
\]
Finally, we get 
\begin{align*}
\int_0^{\infty} \|\Delta e^{t\Delta}\phi\|_{B^{2\theta}_{1,1}(\R^n)} \ud t & \geq \int_0^{\infty} \|\Delta e^{t\Delta}\phi \|_{L^1(\R^n)} \ud t\\
& \gtrsim \int_{1}^{\infty} \frac{1}{t^{\frac{n}{2}+1}} \int_{0}^{\infty} \left|r^{n-1} e^{\frac{-r^2}{2t}}\left(1-\frac{r^2}{t}\right)\right| \ud r \ud t \\
& = \int_{1}^{\infty} \frac{t^{\frac{n-1}{2}+\frac{1}{2}}}{t^{\frac{n}{2}+1}} \int_{0}^{\infty} |e^{-r^2}(1-r^2)| \ud r \ud t = \infty.
\end{align*}
It completes the proof of (iii).

(v) 
 It follows from Corollary \ref{L1 for decr}(i) and  Proposition \ref{L1 extrapol var 2}, respectively.
 
Finally, for (vi), by Theorem \ref{weighted L1 mr}, it is sufficient to show that there exists a constant $C$ such that for any $x\in X$ we have
\begin{equation}\label{mr on R+}
\int_0^{\infty} \|AT(t)x\|_\cX  e^{-t} \ud t + \int_0^{\infty}  [AT(t)x]_{\Phi,\cA} e^{-t} \ud t \leq C (\|x\|_\cX + [x]_{\Phi,\cA}).
\end{equation}

By Proposition \ref{KJ}, up to the change of the weight function, we can assume that $w(t) = \phi(1,1/t)$, $t>0$ with $\phi(1,\cdot)$ is quasi-concave. In particular, 
$\sup_{t>0} e^{-t}/w(t) \leq 1/\phi(1, 1)$.  It gives that 
\[
\int_0^{\infty} \|AT(t)x\|_\cX  e^{-t} \ud t \leq \frac{1}{\phi(1, 1)} \int_0^{\infty} \|AT(t)x\|_\cX  w(t) \ud t = \frac{1}{\phi(1, 1)} [x]_{\Phi, \cA}.
\]
For the second integral in \eqref{mr on R+}, we can apply \eqref{l1 homo} already proved in (iv). It completes the proof of (vi).
\end{proof}

\begin{remark}\label{rem on l1}
(a) By Corollary \ref{P+Q for powers}, the conclusions of Theorem \ref{thm main} hold, in particular, for all power weights $w(t):= t^{-\theta}$ with $\theta \in (0,1)$, that is, for the parts of $\cA$ in the classical interpolation spaces $(\cX,D_\cA)_{\theta, 1}$ with $\theta \in (0,1)$.

(b) Regarding the point (iii) of Theorem \ref{thm main}. By Kalton and Lancien \cite[Theorem 3.3]{KaLa00} (see also Fackler \cite{Fa16}) and Weis' characterization of the $L^q$-m.r. on $\R_+$ (for $q\in (1,\infty)$ and the underlying space with the $U\!M\!D$-property) for any $r\in (1,\infty)\setminus \{2\}$ on $L^r(0,1)=:\cX$ there exists a bounded, injective, sectorial operator $\cA$ with the sectorial angle $\omega_\cA$ equal to $0$ and the dense range, which does not have $L^q$-m.r.\,on $\R_+$ for any $q\in (1,\infty)$. Of course, $\cA$ has $L^q$-m.r.\,on any finite interval. By Dore \cite[Theorem 7.1]{Do00} (or \cite[Theorem 17.2.26, p.598]{HNVW24}), $\cA$ does not have $L^1$-m.r.\,on $\R_+$ (and $L^\infty$-m.r.\,on $\R_+$ as well). 

Since in this case  $D_\cA \simeq \cX$, so $X := (\cX,D_\cA)_{\theta, q} \simeq X$, the part $A$ of $\cA$ in $X$ is just $\cA$. Therefore, this somehow artificial example shows that the assumption that $0\in \rho(\cA)$ cannot be dropped in Theorem \ref{thm main}(iv) even in the case of $\cX$ being a reflexive Banach space and $\cA$ being a densely defined, injective generator of a bounded analytic semigroup, which has dense range. 

In the context of the example provided in (iii) of Theorem \ref{thm main}, since the range of the Laplacian on $L^1(\R^n)$ is a subset of the closed subspace $\{f\in L^1(\R^n): \int_{\R^n} f \ud t = 0\}\subset L^1(\R^n)$, its part in $B^{2\theta}_{1,1}(\R^n)$, $\theta \in (0,1)$, does not have dense range. 

Theorem \ref{thm main} should be read in connection to Theorem \ref{optimal for l1}.  Note that the function $\phi$ which we make use of in the proof of (iii) is not in the range of the Laplacian on $L^1(\R^n)$.

(c) Comparing with a recent result on homogeneous estimates \cite[Proposition 2.14, Theorem 2.20]{DaHiMuTo} (the case $q=1$), in Theorem \ref{thm main}(ii) we relax the assumptions on the force term $f$ in $(CP)_{A,f,x}$ showing that $f$ can be taken from $L^1_{loc} ([0,\infty);X)$ instead of $L^1(\R_+;X)$. 
As the following example shows, our condition  is strictly weaker (for $\theta > \frac{1}{2}$) than that in \cite[Proposition 2.14]{DaHiMuTo}, see also a comment below \cite[Theorem 2.20]{DaHiMuTo} therein. 

 More precisely, setting $\cA := - \Delta$, $\cX: = L^1(\R)$, $\Phi := L^1_{w}$ for $w(t) := t^{-\theta} (t > 0)$, $X: = K_{\Phi}(\cX,D_\cA) = (\cX,D_\cA)_{\theta,1} =  B^{2\theta}_{1,1}(\R)$, there exists $f \in L^1_{loc}(\R_{+}; X)$ such that $[f]_{\Phi, \cA} \in L^1(\R_{+})$, but $f \notin L^1(\R_{+}; X)$. Indeed, let 
\[
f(t) := \chi_{(1,\infty)}(t) \frac{1}{t^2} e^{-\frac{\xi^2}{2t^2}}, \qquad t>0.
\]
The function $f$ is not even in $L^1(\R_{+};L^1(\R)) \supsetneq L^1(\R_{+}; B^{2\theta}_{1,1}(\R))$, because
\[
\int_1^{\infty} \int_{\R} \frac{1}{t^2} e^{-\frac{\xi^2}{2t^2}} \ud \xi \ud t = \sqrt{2\pi} \int_1^{\infty} \frac{1}{t} \ud t = \infty.
\]
However, following calculations done in \cite[Example 17.4.2]{HNVW24} we get
\[
(\Delta e^{s\Delta} f)(\xi) = \frac{1}{t} \frac{1}{(s+t^2)^{\frac{3}{2}}} e^{-\frac{\xi^2}{2(s+t^2)}} (1 - \frac{\xi^2}{s+t^2}).
\]
It gives
\begin{align*}
\int_0^{\infty} \int_0^{\infty} \|\Delta e^{s\Delta} f\|_{L^1(\R)} \frac{1}{s^{\theta}} \ud s \ud t & \approx \int_1^{\infty} \frac{1}{t} \int_0^{\infty} \frac{1}{\sqrt{s}} \frac{1}{s+t^2} \ud s \ud t \\ &= \frac{\pi}{\sin(\pi \theta)} \int_1^{\infty} \frac{1}{t^{1+2\theta}} \ud t \\& = \frac{\pi}{\sin(\pi \theta) 2\theta}.
\end{align*}
Therefore, $[f]_{\Phi, \cA} \in L^1(\R)$. 
\end{remark} 

\subsection{An extension of the Da Prato-Grisvard theorem for $L^\infty$-m.r.}

In this subsection we provide a counterpart of Theorem \ref{thm main} for $L^\infty$-m.r.

\begin{theorem}\label{thm main 2}
 Let $-\cA$ be the generator of a bounded analytic semigroup $\cT$ on a Banach space $\cX$. 

Let $w$ be a weight function on $(0,\infty)$ such that $[w]_{P+Q, L^\infty} < \infty$ $(see, \eqref{P+Q on Linfty})$.

 Denote by $A$ the part of $\cA$ in $X:= K_\Phi(\cX, D_\cA )$,  where $\Phi := L^\infty(w)$.

Then, $-A$ is the generator of a bounded analytic semigroup $T$ on $X$ and the following assertions hold.

\begin{itemize}

\item [{\rm(i)}] For any $q\in (1,\infty)$, if $f\in L^q_{loc}([0,\infty); X)$ and $x\in (X, D_A)_{\frac{1}{q'},q}$, then the function 
\begin{equation}\label{u funct}
u (t):= T(t)x +  \int_0^tT(t-s) f(s) \ud s,  \qquad t>0,
\end{equation}
is a unique strong solution to $(CP)_{A,f,x}$.

\item [{\rm(ii)}] The operator $A$ on $X$ has $L^\infty$-m.r.\,\,on finite intervals. Consequently, for any $\tau\in (0,\infty)$ there exists 
a constant $C_\tau>0$ such that for any $f \in L^\infty_{loc}([0,\infty); X)$ and $x\in D(A)$ the corresponding strong solution $u$ to $(CP)_{A,f,x}$ satisfies 
\item [{\rm(ii)}] The operator $A$ on $X$ has $L^\infty$-m.r.\,\,on finite intervals. Thus, for any $\tau\in (0,\infty)$ there exists a constant $C_\tau>0$ such 
that for any $f \in L^\infty_{loc}([0,\infty); X)$ and $x\in D(A)$ the corresponding strong solution $u$ to $(CP)_{A,f,x}$ satisfies

\begin{align}\label{inhom loo}
\nonumber \| u'\|_{L^\infty(0,\tau; X)}  + \|A u\|_{L^\infty(0,\tau;X)} & + \| u\|_{Lip([0,\tau]; X)} \\ & \leq C_\tau \left( \| f \|_{L^\infty(0,\tau;X)} + \|Ax\|_X\right).
\end{align}

 Furthermore, the following homogenous estimates hold on $\R_+$: 
there exists a constant $C>0$ such that for every $f\in \bigcup_{q>1} L^q_{loc}([0,\infty); X)$ with $[f(\cdot)]_{\Phi, \cA} \in L^\infty(\R_+)$ and every $x\in D(A)$, the  strong solution $u$ to $(CP)_{A,f,x}$ satisfies 
\begin{equation}\label{hom est oo}
\|\,[u']_{\Phi, \cA}\|_{L^\infty(\R_+)} + \|\, [Au]_{\Phi, \cA} \|_{L^\infty(\R_+)}  \leq C \left( \|\, [f]_{\Phi,\cA}\|_{L^\infty(\R_+)} + [Ax]_{\Phi,\cA} \right) 
\end{equation} 
and for all $t,s>0$
\[
[u(t) - u(s)]_{\Phi, \cA}\leq C |t-s| \left( \|\, [f]_{\Phi,\cA}\|_{L^\infty(\R_+)} + [Ax]_{\Phi,\cA} \right),
\]
where $[\cdot]_{\Phi, \cA}$ stands for the homogeneous part of the norm $\|\cdot\|_X$ of $X$.

\item [{\rm(iii)}] If, in addition, $0\in \rho(\cA)$, 
then $A$ on $X$ has $L^\infty$-m.r.\,on $\R_+$ and, furthermore, the solution operator $U$ is bounded on $L^\infty(\R_+; X)$.

 Consequently, for every $f\in L^\infty(\R_+;X)$ and $x\in D(A)$ the corresponding strong solution $u$ satisfies \eqref{hom est oo} with the norm $\|\cdot\|_X$ instead of the seminorm $[\cdot]_{\Phi, \cA}$ and, in addition, 
\[
 \| u \|_{L^\infty(\R_+;X)}\lesssim \| f \|_{L^\infty(\R_+;X)} + \|x\|_X. 
\]

\item [{\rm(iv)}] The above statements $(ii)$ and $(iii)$ remain true if (unweighted) $L^\infty$-m.r. is replaced with (weighted) $L^\infty_v$-m.r. for every power weight $v(t):=t^\mu$ with  $\mu \in (0,1)$. 
\end{itemize}
\end{theorem}

\begin{proof} The proof follows the lines of the proof of Theorem \ref{thm main}. We provide only main supplementary observations which should be made.

Combining Lemma \ref{P+Q ineq}(ii), Proposition \ref{Komastu rep} and Lemma \ref{dense dom}(ii) we get that $-A$ is the generator of a bounded analytic semigroup $T$ on $X$ and $T(t)$ is the restriction of $\cT(t)$ to $X$ ($t\geq0$).  If $\cA$ is not bounded operator, then the semigroup $T$ is not $C_0$-semigroup on $X$, that is, $D(A)$ is not dense in $X$.

For (iii), by Lemma \ref{P+Q ineq}(ii), the operator $P$ (and $Q$) is bounded on $\Phi$. Consequently,  by Proposition \ref{Komastu rep},  up to the equivalence of the norm,  the space $X:= K_\Phi(\cX, D_\cA)$ can be represented as follows: 
\[
X = \big\{ x\in \cX: \, \esssup_{t>0} w(t) \|\cA \cT(t) x\|_{\cX} < \infty \big\}
\]
and 
\[
\| x\|_X \simeq \|x\|_\cX + [x]_{\Phi, \cA} =: \|x\|_\cX +  \esssup_{t>0} w(t) \|\cA \cT(t) x\|_{\cX}, \qquad x\in \cX.
\]
 
Since $\frac{d}{\ud t} ( \cA \cT(t)) = \cA^2 \cT(t)$ ($t>0$) and  $\|\cA^2 \cT(s)\|_{\cL(\cX)} \lesssim s^{-2}$ $(s>0)$, for every $t>0$ and $x\in \cX$ we obtain that  
\begin{equation}\label{cond oo}
-\cA \cT(t) x = \int_t^\infty \cA^2 \cT(s) x \ud s = 2\int_{t/2}^\infty \cA^2 \cT(2s) x \ud s. 
\end{equation}

Therefore, combining \eqref{cond oo}  with \eqref{Q on l infty}, for all $t>0$ and $x\in X$:
\begin{align*}
w(t) \|A T(t) x\|_\cX & \leq 2 w(t) \int_{t/2}^\infty \|\cA^2 \cT(2s) x\|_{\cX} \ud s\\
& = 2w(t) \int_{t/2}^\infty \frac{\ud s}{s w(s)} \esssup_{s>t/2} w(s) s \|\cA^2 \cT(2s) x\|_{\cX}\\
& \leq 6 [w]_{P+Q,L^\infty} \esssup_{s>t/2} w(s) s \|\cA^2 \cT(2s) x\|_{\cX}. 
\end{align*}
Moreover, note that 
\begin{align*}
\esssup_{s>t/2} w(s) s \|\cA^2 \cT(2s) x\|_{\cX} & \leq \sup_{s>0} \esssup_{t>0} w(s) t \|\cA^2 \cT(s+t)x\|_{\cX} \\
& = \sup_{t>0} \|t \cA \cT(t) x\|_{X}
\end{align*}
Therefore, for every $x \in X$ we get 
\begin{equation}\label{charact oo}
[x]_{\Phi, \cA} \leq 6[w]_{P+Q,L^\infty} \sup_{t>0} [t A T(t) x]_{\Phi, \cA}.
\end{equation}

For further references, we point out that to get \eqref{charact oo} we are not making use of the assumption that $0\in \rho(A)$. Assuming it, we get that $[\cdot]_{\Phi, \cA}$ is equivalent to $\|\cdot\|_X$ on $X$, see Proposition \ref{Komastu rep}, so the Kalton-Portal characterization (see Lemma \ref{charact}), shows that $A$ has $L^\infty$-m.r.\,on $\R_+$. The boundedness of $U$ is a direct consequence of Dore \cite[Theorem 5.2]{Do00}; cf. Remark \ref{Dore results}. It completes the statement (iii).

Now we can proceed to the proofs of (i) and (ii). The same rescaling argument as that applied in the proof of Theorem \ref{thm main}(ii) shows that $A$ has $L^\infty$-m.r.\,on finite intervals. It gives the first assertion of (ii).
Furthermore, by Dore's extrapolation result, \cite[Theorem 7.1]{Do00}, for every $q\in (1,\infty)$, $A$ has $L^q$-m.r.\,on finite intervals. 

For (i), let $f\in \bigcup_{q>1} L^q_{loc}([0,\infty);X)$ and $x\in (X, D_A)_{\frac{1}{q'}, q}$. Since  the function $\R_+\ni t\mapsto T(t)x\in X$ is  bounded and analytic, and its derivative is in $L^1_{loc}([0,\infty);X)$,  $u$ given by \eqref{u funct} is a strong solution to $(CP)_{A,f,x}$. It completes the proof of (i).

For the second assertion of (ii), fix $f\in \bigcup_{q>1} L^q_{loc}([0,\infty);X)$ with $[f]_{\Phi, \cA}\in L^\infty(\R_+)$ and let $x = 0$. We already know that  $U f(t) \in D(A)$ for a.e. $t\in \R_+$; cf. Remark \ref{Dore results}\eqref{loc est}. For such  $t\in \R_+$, by \eqref{charact oo}, we have that
\[
[AU f(t)]_{\Phi, \cA} \leq 2 [w]_{P+Q, L^\infty} \sup_{s>0} [sAT(s) A U f(t)]_{\Phi, \cA}. 
\]
Set $C_\cA:=\sup_{t>0} \|t\cA\cT(t)\|_{\cL(\cX)}$. Then, 
\begin{align*}
\sup_{s>0} [sAT(s) A   &    U f(t)]_{\Phi, \cA}\\
& = \sup_{s>0} \esssup_{\tau>0} w(\tau) \Bigl\|\int_0^t s\cA^2\cT(s+t-r) \cA \cT(\tau) f(r) \ud r\Bigr\|_{\cX}\\
& \leq \sup_{s>0} C_\cA^2  \int_0^t  \frac{4s}{(s+t-r)^2} \esssup_{\tau>0} w(\tau) \|\cA \cT(\tau) f(r)\|_\cX \ud r \\
& = 4  C_\cA^2\; \esssup_{0<r<t}[f(r)]_{\Phi, \cA}. 
\end{align*}
Therefore, 
\begin{equation}\label{oo est}
\| [AU f]_{\Phi, \cA}\|_{L^\infty(\R_+)} \leq 24[w]_{P+Q,L^\infty} C^2_\cA 
\,\| [f]_{\Phi, \cA}\|_{L^\infty(\R_+)}.
\end{equation}

The similar argument to that of the proof of Theorem \ref{thm main}(ii) shows that $U f$ satisfies Lipschitz' condition on $\R_+$ with a constant less or equal to $\|\,[f - AU f]_{\Phi, \cA}\|_{L^\infty(\R_+)}$. 

For a non-zero initial value $x \in D(A)$, 
if $u$ denotes the corresponding strong solution to $(CP)_{A,f,x}$, then 
for a.e. $t>0$ we have
\begin{align*}
[Au(t)]_{\Phi, \cA}  & \leq \esssup_{s>0} w(s) \|\cA \cT(s) AT(t) x\|_\cX + [AU f(t)]_{\Phi, \cA}\\
&\leq \sup_{t>0} \|\cT(t)\|_{\cL(\cX)} [A x ]_{\Phi, \cA} + [AU f(t)]_{\Phi, \cA}.
\end{align*}
Thus, by \eqref{oo est},
\[
\|\,[u']_{\Phi, \cA}\|_{L^\infty(\R_+)} \lesssim [A x ]_{\Phi, \cA} + \|\,[f]_{\Phi, \cA}\|_{L^\infty(\R_+)}.
\]
 It readily gives the desired estimates. Therefore, the proof of (ii) is complete. Finally, the statement (iv) is now a direct consequence of \cite[Proposition 17.2.36]{HNVW24}.
\end{proof}

\begin{remark}
(a) By Corollary \ref{P+Q for powers}, the conclusions of Theorem \ref{thm main 2} hold for all power weights $w(t):= t^{1-\theta}$ with $\theta \in (0,1)$, that is, for the parts of $\cA$ in the classical interpolation spaces $(\cX,D_\cA)_{\theta, \infty}$ with $\theta \in (0,1)$.

(b)  The estimate \eqref{hom est oo} in which the seminorm $[\cdot]_{\Phi, \cA}$ is replaced with the norm $\|\cdot\|_{X} = \|\cdot\|_\cX + [\cdot]_{\Phi, \cA}$, in general, does not hold. More precisely, for any $p\in (1,\infty)\setminus \{2\}$ there is a bounded, injective operator $\cA$ on $L^p(0,1)=:\cX$ with dense range such that for all $\theta \in (0,1)$ there exists a function $f\in L^\infty(\R_+; \cX)$, $AU f\notin  L^\infty(\R_+; \cX)$, where $
U f(t) : = \int_0^tT(t-s)f(s) \ud s$  $(t\geq 0)$. See Remark \ref{rem on l1}(b). 

(c) We do not know whether for the negative Laplacian $\cA: = -\Delta$ on $\cX = C_0(\R^n)$ its part in $X:=(\cX,D_\cA)_{\theta, \infty}$, $\theta \in (0,1)$, has $L^\infty$-m.r. on the whole $\R_+$. 
\end{remark}

\subsection{Maximal regularity estimates via completion and extrapolation}
In this subsection we provide abstract results which in particular allow to {\it extrapolate} the $L^p$-m.r.\,estimates with respect to the homogeneous part $[\cdot]_{\Phi, \cA}$ of interpolation norm \eqref{equi norm}. Moreover, it throws a light on some relevant results from the literature; see Section \ref{sec app}. 
Roughly, the $L^p$-norms with respect to which we measure the time regularity of $[A u]_{\Phi, \cA}$ in (ii) of Theorems \ref{thm main} and \ref{thm main 2} can be replaced by the norm corresponding to a much larger class of Banach function spaces over $(\R_+, \ud t)$ than the Lebesgue spaces $L^p$. It provides a closer correspondence between the time regularity of $[A u]_{\Phi, \cA}$ and $[f]_{\Phi,\cA}$.
We make this statement more precise in the following results, relying on an extrapolation result \cite[Theorem 5.1]{ChKr18}; cf. also \cite[Theorem 5.1 and Corollary 5.1]{Kr23}.

For the definitions of the {\it weighted} variant $E_v$ of an arbitrary rearrangement-invariant Banach function space $E$ over $\R_+$ and its Boyd's indices $p_E$ and $q_E$, which are involved in the formulation of such extensions, we refer the reader to \cite[Section 4]{ChKr18}. We point out that Boyd's indices can be computed explicitly for many examples of concrete Banach function spaces; see e.g. [7, Chapter 4].  Moreover, a  weight $v$ on $(0,\infty)$ belongs to  the class $A^-_p(\R_+)$, $p\in (1,\infty)$, if and only if 
\begin{equation}\label{Ap- weight}
\sup_{0\leq a<b<c} \frac{1}{(c-a)^p} \int_b^c w\, \ud t \left( \int_a^b w^{1-p'} \ud t \right)^{p-1} <\infty
\end{equation}
(see  \cite[Definition 2.1(d)]{ChKr18}; cf. also \cite[Theorem 2.3]{ChKr18}). It is readily seen that  each class $A^-_p(\R_+)$, $p\in (1,\infty)$, contains all non-increasing, positive functions on $(0,\infty)$. Note that for $v\equiv 1$ on $(0,\infty)$, $E_v = E$ with equal norms.

We begin with an observation, which can be easily derived from Lemma \ref{lem on density},  and is used in formulating the following result. Namely, for any generator $-A$ of the semigroup $T$ on a Banach space $X$ and for any $X$-valued step function $f$, $Uf(t) := \int_0^t T(t-s)f(s) \ud s \in D(A)$ for all $t\geq 0$.

\begin{proposition}\label{hom ext p}
Let $p\in [1,\infty)$ and let $-A$ be the injective generator of a $C_0$-semigroup $T$ on a Banach space $(X, \|\cdot\|_X)$. Suppose that $\|\cdot\|_*$ is another, weaker norm on $X$ $($not necessarily complete$)$, such that for some constant $M$ and all $x\in X$ and $t>0$ we have $\|T(t)x\|_*\leq M\|x\|_*$. 

Suppose that $A$ satisfies the following $L^p$-m.r.~estimate on $\R_+$ with respect to $\|\cdot\|_*$: for any step function $f$ the mild solution $u:=Uf$ to $(CP)_{A,f,0}$ satisfies{\rm:}
\begin{equation}\label{estimate on test functions}
\bigl\|\,\|Au\|_*\bigr\|_{L^p(\R_+)}\leq C \bigl\|\,\|f\|_*\bigr\|_{L^p(\R_+)} 
\end{equation}
for some constant $C>0$ independent on $f$.

Then, the following statements hold{\rm:}

{\rm(a)} The canonical extension of $A$ on a Banach completion $\widetilde X$ of $(X,\|\cdot\|_*)$ has $L^p$-m.r.\,on $\R_+$. 

{\rm(b)} If, in addition, the operator $A$ has $L^p$-m.r. on finite intervals, then for any $f \in L^1_{loc}([0,\infty);X)$ the corresponding mild solution $u := Uf$ to $(CP)_{A,f,0}$ 
has the right-hand derivative $u'_{+}(t)$ in $X$ for a.e. $t>0$ and for a.e. $t>0$

\[
u(t) \in D(A) \quad \textrm{ and }\quad u'_+(t)+ Au(t) = f(t).
\]

Moreover, if $f \in \bigcup_{r>1}L_{loc}^r([0,\infty);X)$ $($or even $f \in L^1_{loc}([0,\infty);X)$ when $p=1)$, then $u$ is a strong solution to $(CP)_{A,f,0}$ and for every rearrangement invariant Banach function space $E$ with non-trivial Boyd's indices $p_E, q_E \in (1,\infty)$ and for every weight function $v\in A_{p_E}^-(\R_+)$ there exists a constant $C_{E, v}$ such that 
\begin{equation}\label{E mr est general}
\bigl\|\,\|u'\|_* \bigr\|_{E_v(\R_+)} + \bigl\|\,\|Au\|_* \bigr\|_{E_v(\R_+)} \leq C_{E,v} \bigl\|\,\,\|f\|_*\bigr\|_{E_v(\R_+)}.
\end{equation}
\end{proposition}

\begin{proof} Let $(\wX, \|\cdot\|_{\wX} )$ be the standard Banach completion of $(X, \|\cdot\|_*)$, that is, $\wX$ is a space of  all equivalence classes $[\bfx]$ of Cauchy sequences $\bfx := (x_n)_{n\in\N} \in X^\N$ with respect to $\|\cdot\|_*$, and $\|[\bfx]\|_{\wX} : =\lim_{n\rightarrow \infty}\|x_n\|_*$. Recall that the map $J:X\ni x \mapsto [(x)_{n\in \N}] \in \wX$ is an isometric isomorphism of $(X, \|\cdot\|_*)$ onto a dense subspace of $\wX$. 
  
From the boundedness of $T$ on $(X, \|\cdot\|_*)$, it is easily seen that for all $t>0$ the formula  
\[
\wT (t)[(x_n)] := [(T(t)x_n)],  \qquad  [(x_n)] \in \wX, 
\]
defines a family of uniformly bounded operators $\wT (t)$, $t>0$,  on $\wX$, which satisfies the semigroup property. Since $\wT (t)Jx = J T(t)x$ for all $t>0$ and $x\in X$, the density $JX$ in $\wX$ gives that $\wT$ is a $C_0$-semigroup on $\wX$. We denote the negative generator of $\wT$ by $\wA$. Therefore, Lemma \ref{regularity of u} applies to $\wA$. For the simplicity, we set $\widetilde{U}$  instead of $U_{\wA}$ (resp. $U$ instead of $U_A$) to denote the solution operator associated with $\wA$ on $\wX$ (resp. $A$ on $X$). 

It is clear that the function $Jf$ is in $L^p(\R_+; \wX)$ and $JU f(t) = \widetilde{U} J f (t)$ for all $t>0$. Moreover, for any $x\in X$ and for $\lambda>0$ large enough, we have $J(\lambda + A)^{-1}x = (\lambda + \wA)^{-1} Jx$. Hence, for any $x\in D(A)$, taking $J$ on the both sides of the equation 
\[
\lambda(\lambda+ A)^{-1} x -x = (\lambda + A)^{-1}Ax
\] 
gives 
\[
\lambda(\lambda+ \wA)^{-1} Jx - Jx = (\lambda + \wA)^{-1} JAx,
\] 
which is equivalent to $Jx\in D(\wA)$ and $\wA J x = JAx$. Applying it to $x = U f(t) \in D(A)$, we get that $JU f(t) = \widetilde{U} Jf (t) \in D(\wA)$ and $\wA \widetilde{U} Jf(t) = J A U f(t)$ for a.e. $t>0$. 
Therefore, for any step function $f : \R_+ \to X$ we have
\[
\widetilde{U} Jf (t) \in D(\wA) \textrm{ for a.e. } t>0 \quad \textrm{ and }\quad \|\wA\widetilde{U} Jf \|_{L^p(\R_+; {\wX})}\leq C \| Jf \|_{L^p(\R_+; {\wX})}. 
\]
We show that it holds for all functions $\tilde{f} \in L^p(\R_+;\wX)$. Then, by Lemma \ref{regularity of u}(ii), $\wA$ on $\wX$ has $L^p$-m.r. on $\R_+$. Note that the set 
\[
\bigl\{ \tilde{f}\in L^p(\R_+;\wX): \widetilde{U} \tilde{f}(t)\in D(\wA) \textrm{ a.e. } t>0 \textrm{ and } \|\wA \widetilde{U} \tilde{f}\|_{L^p(\R_+, \wX)}\leq C\|\tilde{f}\|_{L^p(\R_+;\wX)}\bigr\}
\]
is closed in $L^p(\R_+;\wX)$.
Indeed, suppose that a function $g\in L^p(\R_+;\wX)$ can be approximated in the norm of $L^p(\R_+;\wX)$ by functions $g_n$ from this set. Then, $\wA \widetilde{U} g_n$, $n\in \N$, is a Cauchy sequence in $L^p(\R_+;\wX)$. This implies that $\wA \widetilde{U} g_n(t)$ converges, modulo taking a subsequence, almost everywhere on $\R_+$. In particular, by the closedness of $\wA$, we get that $\widetilde{U} g(t) \in D(\wA)$ and $\wA\widetilde{U} g(t) = \lim_{n\in \N} \wA \widetilde{U} g_n(t)$ for a.e. $t>0$. By Fatou's property of $L^p$, we get $\|\wA \widetilde{U} g\|_{L^p(\R_+;\wX)} \leq C \|g\|_{L^p(\R_+;\wX)}$. It shows the desired closedness. 
Since $JX$ is a dense subset of $\wX$, it completes the proof of (a).

The first assertion of (b) directly follows from Proposition \ref{weak L1} and the observation provided in Remark \ref{D-Y-S}. 
For the second one, first note that by Dore's extrapolation result \cite[Theorem 7.1]{Do00} (resp. combined with (a)), the operator $A$ on $X$  (resp. $\widetilde A$ on $\wX$) has $L^r$-m.r. on finite intervals (resp. on $\R_+$) for every $r\in \{p\}\cup (1,\infty)$. 
In particular, if $f \in \bigcup_{r>1}L_{loc}^r([0,\infty);X)$ (or even $f \in L^1_{loc}([0,\infty);X)$ when $p=1$), $Uf$ is a strong solution to $(CP)_{A,f,0}$. Consequently, $\widetilde U Jf = JU f$ is a strong solution to $(CP)_{\widetilde A, Jf,0}$ and, since the norm $\|\cdot\|_*$ is weaker than $\|\cdot\|_X$, $(\widetilde U Jf)' = J(U f)'$ and $J A U f = \wA \widetilde{U} Jf$. Therefore, applying \cite[Theorem 5.1]{ChKr18} (see also \cite[Theorem 4.3]{ChKr18}) to $\widetilde A$ on $\wX$, if (in addition) $Jf \in E_v(\R_+; \wX)$, i.e. $\|Jf\|_{\wX} = \|f\|_* \in E_v(\R_+)$, then $ \wA \widetilde{U} Jf$ and $(\widetilde U Jf)'$ are in $E_v(\R_+;\wX)$ and \eqref{E mr est general} holds. In the other case, that is, when $f \notin E_v(\R_+;\wX)$, $\|\,\|f\|_*\|_{E_v(\R_+)} = \infty$, the estimate \eqref{E mr est general} is obviously valid. Thus, the proof is complete. 
\end{proof}

It is not clear whether the direct counterpart of the above result holds for the case of $p=\infty$. We prove the variant for the $L^\infty$-m.r.\,estimates under slightly modified assumptions, relaying on Lemma \ref{charact}. The assumptions stated in the following result can be verified in the context we are interested in.
Recall that $S_\omega: = \{ z \in \C\setminus \{0\}: |\arg z|<\omega\}$; see Subsection \ref{subs H}.

\begin{proposition}\label{hom ext infty}
Let $-A$ be the injective generator of a $C_0$-semigroup $T$ on a Banach space $(X, \|\cdot\|_X)$, which admits a bounded holomorphic extension on a sector $S_\omega$ for some $\omega>0$. Suppose that $\|\cdot\|_*$ is another, weaker norm on $X$ $($not necessarily complete$)$, such that for some constant $M$ and all $x\in X$ and $z\in S_\omega$ we have $\|T(z)x\|_*\leq M\|x\|_*$. Let $\widetilde X$ denote the Banach completion of $(X,\|\cdot\|_*)$.

Suppose that there exists a constant $C>0$ such that for every $x \in X$:
\[
\|x\|_* \leq C \sup_{t>0} \|tAT(t)x\|_* + \limsup_{t\rightarrow \infty} \|T(t)x\|_*.
\]
Then, the canonical extension of $A$ on $\widetilde X$ has $L^{\infty}$-m.r. on $\R_+$.

In addition, if $A$ on $X$ has $L^\infty$-m.r.\,on finite intervals, then the statement $($b$)$~~of Proposition \ref{hom ext p} holds. 
\end{proposition}

\begin{proof} As in the proof of Proposition \ref{hom ext p} one can show that for all $z\in S:=S_\omega$ the formula  
\[
\wT (z)[(x_n)] := [(T(z)x_n)], \quad  [(x_n)] \in \wX, 
\]
defines a family of uniformly bounded operators $\wT (z)$, $z\in S$,  on $\wX$, which satisfies the semigroup property. Since $JX$ is dense in $\wX$ and $\wT (z)Jx = J T(z)x$ for all $z\in S$ and $x\in X$, by Vitali's theorem \cite[Theorem A.5, p.458]{ABHN01}, for all $\bfx\in \wX$ the function $S\ni z \mapsto \wT(z)\bfx \in \wX$ is holomorphic. Therefore, by \cite[Proposition A.3]{ABHN01}, we get that $\wT$ is bounded holomorphic $C_0$-semigroup on $\wX$. We denote the negative generator of $\wT$ by $\wA$. 
In particular, $\widetilde M: =  \sup_{t>0}\|t\wA \wT(t)\|_{\cL(\wX)}<\infty$.
As in the proof of the case when $p<\infty$, for any $x\in D(A)$ and  $\lambda>0$ large enough, we have $Jx\in D(\wA)$ and $\wA J x = JAx$. Take any $\bfx \in \widetilde X$ and $x\in X$. Then,
\begin{align*}
\|\bfx\|_{\widetilde X} &\leq  \|\bfx - Jx\|_{\wX} + \|x\|_{*} \\
& \leq \|\bfx - Jx\|_{\widetilde X} + C \sup_{t>0}\|tAT(t)x\|_{*}+ \limsup_{t\rightarrow \infty}\|T(t)x\|_*\\
& = \|\bfx - Jx\|_{\widetilde X} + C \sup_{t>0}\|t\wA \wT (t)x\|_{\wX}+ \limsup_{t\rightarrow \infty}\|\wT(t) Jx\|_{\wX}\\
 & \leq \|\bfx - Jx\|_{\wX} + C\sup_{t>0}\left( \|t\wA \widetilde{T}(t)(\bfx-Jx)\|_{\wX} + \|t \wA\wT(t) x\|_{\wX}\right)\\
&\quad  + \limsup_{t\rightarrow\infty} \left(\|\wT(t)(Jx - \bfx)\|_{\wX} + \|\wT(t)\bfx\|_{\wX} \right) \\ 
& \leq (1+ C \widetilde{M} + M)\|\bfx-Jx\|_{\widetilde X} + C\sup_{t>0}\|t\widetilde A \widetilde{T}(t)\bfx\|_{\widetilde X} + \limsup_{t\rightarrow\infty} \|\wT(t)\bfx\|_{\wX}.
\end{align*}

Hence, since $JX$ is dense in $\wX$, by Lemma \ref{charact}, $\widetilde A$ has $L^{\infty}$-m.r. on $\R_+$. 
Finally, the proof of the additional statement follows the same arguments used in the proof of Proposition \ref{hom ext p}(b).
\end{proof}

\begin{remark}\label{force term}
The natural question arises whether the force term $f$ in Proposition \ref{hom ext p}(b) (in the case $p \neq 1$) can be taken from $L^1_{loc}([0,\infty);X)$ to obtain the corresponding conclusions on the homogeneous estimates therein.  
This question connects with Proposition \ref{weak L1} and Remark \ref{D-Y-S}. We do not know if the mild solution $Uf$ with $f \in L^1_{loc}([0,\infty);X) \setminus \bigcup_{r>1}L_{loc}^r([0,\infty);X)$ has the full strong derivative a.e. in $X$ even if $A$ is densely defined. In the setting of Theorem \ref{thm main 2}(ii), the operator $A$ is not densely defined; see Lemma \ref{dense dom}(ii). However, Proposition \ref{hom ext infty} remains applicable to the part of $A$ in the closure of its domain. 
\end{remark}

\begin{corollary}\label{hom ext 1-p} Let $-\cA$ be an injective generator of a bounded analytic semigroup on $\cX$. Let $p \in [1,\infty)$ and $\theta \in (0,1)$. Let $A$ denote the part of $\cA$ in $X := K_\Phi(\cX,D_{\cA})$  with $\Phi:= L^p(\R_+, t^{-p(1-\theta) -1}\ud t)$.    

Then, for every rearrangement invariant Banach function space $E$ with non-trivial Boyd's indices $p_E, q_E \in (1,\infty)$ and for every weight function $v\in A_{p_E}^-(\R_+)$ there exists a constant $C_{E, v}$ such that for any $f \in \bigcup_{r>1}L_{loc}^r([0,\infty);X)$ $($or even $f\in L^1_{loc}(\R_+;X)$ when $p=1)$, the corresponding strong solution $u$ to $(CP)_{A,f,0}$ satisfies
\begin{equation}\label{E mr est}
\|\,[u']_{\Phi, \cA}\|_{E_v(\R_+)} + \|\,[Au]_{\Phi, \cA}\|_{E_v(\R_+)} \leq C_{E,v} \|\,[f]_{\Phi, \cA}\|_{E_v(\R_+)}.
\end{equation}
\end{corollary}
A similar result holds for $p=\infty$ as well. We leave the detailed formulation to the reader; cf. Remark \ref{force term}.

\begin{proof}
Since $\cA$ is injective, $\|\cdot\|_* := [\cdot]_{\Phi, \cA}$ is a weaker norm on $(X, \|\cdot\|_X)$. By Lemma \ref{dense dom}, the semigroup $T$ generated by the part $A$ of $\cA$ in $X$ is a $C_0$-semigroup on $X$. Moreover, for all $x\in X$ and $t>0$
\[
\|T(t)x\|_* := \left\|\,\|\cA\cT(\cdot)T(t)x\|_\cX \right\|_\Phi \leq  \|\cT(t)\|_{\cL(\cX)} \|x\|_*.
\] 
The proof of the classical Da Prato-Grisvard theorem, as presented, e.g. in [45, Theorem 3.5.8, p. 146], shows that the estimate \eqref{estimate on test functions}  holds for $p\in[1,\infty)$ (for the case of $p=1$, which is omitted therein, see, e.g. \eqref{l1 homo}). Since $A$ has $L^p$-m.r. on finite intervals, by Proposition \ref{hom ext p}(b), we obtain the desired claim.
\end{proof}

Our next result sheds light on the Banach completion of the space $X$ equipped with the homogeneous part of the corresponding interpolation norm $[\cdot]_{\Phi, \cA}$, which is involved in the formulation of the corollary above.
It also completes Proposition~\ref{Komastu rep} (see the statements (i) and (ii)) and we utilize it in the remaining part of this section (see Proposition \ref{XD}), as well as in Section 5 (see Proposition \ref{equivalent norm on besov}). All the completions considered in the statement (iii) below are standard Banach completions.

\begin{lemma}\label{hom isom}
 Let $-\cA$ be the generator of a bounded semigroup $\cT$ on a Banach space $\cX$. 
Let $\dot{D}_\cA$ denote the domain of $\cA$ equipped with the seminorm $\|x\|_{\dot{D}_\cA} :=\|\cA x\|_{\cX}$, $x\in D(\cA)$.  For all $x\in \cX$ and $t>0$ let
\[
K(t,x, \cX, \dot{D}_\cA) = \inf \{\|y\|_\cX + t\|z\|_{\dot{D}_{\cA}}: y\in \cX,\, z\in D(\cA),\, x=y+z \}.
\]
Then the following assertions hold.

\begin{itemize}
\item [(i)]  
For any $x\in \cX$ and $t>0$ 
\begin{equation}\label{equi for K}
C^{-1} \|\cT(t)x - x\|_{\cX} \leq K(t,x, \cX, \dot{D}_\cA) \leq \Big\|\frac{1}{t}\int_0^t (\cT(s)x -  x)\ud s\Big\|_\cX +   \|\cT(t) x -x\|_\cX,
\end{equation}
where $C: = \sup_{s\geq 0 }\|\cT(s)\|_{\cL(\cX)} + 1$.
In particular, 
\[
\omega_{\cA}(t, x) := \sup_{0<s\leq t} \|\cT(s)x - x\|_{\cX} \simeq K(t,x, \cX, \dot{D}_\cA), \qquad  t>0, \,x\in \cX.
\]
If, in addition, $\cA$ is injective, then  the pair $(\cX, \dot{D}_\cA)$ forms an interpolation couple of normed spaces.

\item [(ii)] Let $\Phi$ be a parameter of the $K$-method. Then 
\[
K_\Phi(\cX, \dot{D}_\cA):= \bigl\{ x \in \cX: (\cdot)^{-1}K(\cdot,x; \cX, \dot{D}_\cA) \in \Phi\bigr\} 
= K_\Phi(\cX, D_\cA)
\] 
$($as sets$)$. If, in addition, the Hardy operator $P$ is bounded on $\Phi$, then for all $x\in \cX$:
\[
\bigl\|\,(\cdot)^{-1}\cT(\cdot)x -x\|_\cX \bigr\|_\Phi \simeq \bigl\|\,(\cdot)^{-1}\omega_{\cA}(\cdot,x)\bigr\|_{\Phi}\simeq \bigl\|\, \|\cA \cT(\cdot)x\|_{\cX}\bigr\|_{\Phi}.
\]

\item [(iii)] Let, in addition, $\cA$ be injective and $\Phi$ be a Banach function space with an absolutely continuous norm.  
Suppose that there exists a Hausdorff topological vector spaces $Y$ and two injective linear maps $J:\cX \rightarrow Y$ and $\kappa : \widetilde{\dot D_\cA}\rightarrow Y$ such that for any $x\in D(\cA)$ we have $\kappa \bfx  = Jx$, where $\bfx  = [(x)_{n\in \N}]$. 

Let $J\cX$ and $\cZ := \kappa\, \widetilde{\dot D_\cA}$ be equipped with the image norms, that is, $\|Jx\|_{J\cX} := \|x\|_\cX$ for all $x\in \cX$ and $\|\kappa \bfx\|_{\cZ} := \bigl\| \bfx \bigr\|_{\widetilde{\dot D_\cA}}$ for all $\bfx \in \widetilde{\dot D_\cA}$.   

Suppose that $J\cX \cap \mathcal{Z} = JD(\cA)$. Then, the completion of $K_\Phi(\cX, \dot{D}_\cA)$ is topologically isomorphic to  $K_\Phi(J\cX^0,\mathcal{Z})$, where $\cX^0$ is the closure of $D(\cA)$ in $\cX$.
\end{itemize}

\end{lemma}

\begin{proof} (i) 
Let $x\in \cX$. Take $y\in \cX$ and  $z\in D(\cA)$ with $x= y + z$. By Lemma \ref{lem on density} and Hille's theorem we get for all $t>0$ that
\[
z - \cT(t) z =  \cA \int^t_0 \cT(s) z \ud s = \int^t_0 \cT(s) \cA z \ud s.
\]
Therefore, for all $t>0$,
\begin{align}\label{equi for K l}
\|\cT(t) x - x\|_{\cX} & \leq \|\cT(t) - I\|_{\cL(\cX)}\|y\|_\cX + t \sup_{0<s\leq t} \|\cT(s)\|_{\cL(\cX)} \|\cA z\|_{\cX}\\
\nonumber & \leq \bigl(1+ \sup_{0<s\leq t} \|\cT(s)\|_{\cL(\cX)} \bigr) K(t,x, \cX, \dot{D}_\cA).
\end{align} It yields the left hand-side of \eqref{equi for K}. 
For the right hand-side, again by Lemma \ref{lem on density}, for any $x\in \cX$ and $t>0$, taking $y := \frac{1}{t}\int_0^t (x - \cT(s) x ) \ud s \in \cX$ and $z:= \frac{1}{t}\int_0^t \cT(s) x \ud s \in D(\cA)$, we easily get the desired estimate in \eqref{equi for K}. Further, since $K(\cdot, x, \cX, \dot D_{\cA})$ is non-decreasing on $\R_+$, \eqref{equi for K} gives the statement about $\omega_\cA(\cdot, \cdot)$.

The last statement of (i) follows from  Proposition \ref{Cauchy}(i) below, but the following direct argument also can be applied. Let $N(x): = K(1, x, \cX, \dot{D}_\cA)$, $x\in X$.
To show that $N$ is a norm on $\cX$ it suffices to deduce that for all $x\in \cX$, $N(x) = 0$ implies $x= 0$. Then, $\cX$ equipped with $N$ is a normed space such that $\cX$ and $\dot{D}_\cA$ are continuously embedded in it, hence  $(\cX, \dot{D}_\cA)$ is an interpolation couple. By \eqref{equi for K} we get for all $t>0$ that 
\begin{align}\label{equi for K l}
\|\cT(t) x - x\|_{\cX} \leq  (1+t)\bigl(1+ \sup_{0<s\leq t} \|\cT(s)\|_{\cL(\cX)} \bigr)N(x) = 0.
\end{align}
Therefore, $\cT(t)x = x$ for all $t>0$. Consequently, $x\in D(\cA)$ with $\cA x =0$, and, by the injectivity of $\cA$, we get $x=0$. It completes the proof of (i).

For (ii), note that for all $t>0$
\[
\Big\|\frac{1}{t}\int_0^t (\cT(s)x -  x)\ud s\Big\|_\cX\leq t P\bigl({(\cdot)^{-1}}\|\cT(\cdot)x- x\|_\cX \bigr) (t). 
\]
Therefore, (ii) follows directly from (i) and the presentation of the proof of Proposition \ref{Komastu rep}, which shows that for all $t>0$ and $x\in \cX$ we have
\begin{equation}
\widetilde{C}^{-1} \|\cA \cT(t)\|_{\cX} \leq  \frac{1}{t}   K(t,x, \cX, \dot{D}_{\cA}) \leq \frac{1}{t}\int_0^t \|\cA \cT(s) x\|_{\cX} \ud s + \|\cA \cT(t) x\|_{\cX}, 
\end{equation}
where $\widetilde{C}: = \max(\sup_{s>0} \|s\cA \cT(s)\|_{\cL(\cX)}, \sup_{s>0}\|\cT(s)\|_{\cL(\cX)})$; cf. Subsection \ref{subs H}. 
 (iii) First, from the assumption on $\Phi$, $D(\cA)$ is a dense subset of $K_\Phi(\cX, \dot D_\cA) = K_\Phi(\cX^0, \dot D_\cA)$ and $\mathcal{Z}\cap J\cX^0$ is a dense subset of $K_\Phi(J\cX^0,\mathcal{Z})$; see the proof of Proposition \ref{dense dom}.   
Since the set $\{ [(x)_{n\in \N}] : x\in D(\A) \}$ is a dense subset of $\widetilde{\dot D_\cA}$ and the isomorphism $\kappa$ between this completion and $\cZ$ is compatible with $J$, we get that $JD(\cA)$ is dense in $\mathcal{Z}$. Of course, $JD(\cA)$ is also dense in $J\cX^0$, and consequently $JD(\cA) \subset JK_\Phi(\cX, \dot D_\cA)$ is a dense subset of a Banach space $K_\Phi(J\cX, \mathcal{Z})$. 
Therefore, by Lemma \ref{isomoprhism of completions}, it is sufficient to show that 
there exists a constant $c>0$ such that for any $x\in D(\cA)$:
\[
{c}^{-1}\|x\|_{K_\Phi(\cX, \dot D_\cA)} \leq  \|J x\|_{K_\Phi(J\cX, \mathcal{Z})} \leq c \|x\|_{K_\Phi(\cX, \dot D_\cA)}.
\]
By our assumption  $J\cX \cap \mathcal{Z} = JD(\cA)$, one can show even more, namely, for any $x\in D(\cA)$ and all $t>0$:
\[
K(t, x,\cX, \dot{D}_\cA) =  K(t, x, J\cX, \mathcal{Z}).
\]
Indeed, for any $x\in D(\cA)$, if $x^0 \in \cX^0$ and $z\in \mathcal{Z}$ are such that $Jx = Jx^0 + z$, then, by our assumptions, there exists $x^1\in D(\cA)$ with $z = Jx^1$ and $\|z\|_{\mathcal{Z}} = \|x^1\|_{\dot D_{\cA}}$. It easily leads to the desired equality for the corresponding $K$-functionals, and completes the proof of (iii).
\end{proof}
\begin{remark} The point (iii) of Lemma \ref{hom isom} shows that our approach to studying global-in-time estimates for the strong solutions of $(CP)_{A,f,x}$, with respect to the homogeneous parts of interpolation norms, is an alternative to that proposed in \cite[Section 2]{DaHiMuTo}. 
We proceed to the completion after interpolation. In fact, in \cite{DaHiMuTo} the procedure is the opposite; in a sense, the normed space $Y$ considered therein can naturally be taken as a subspace of the completion of $D(\cA)$ in the norm $\|\cA \cdot\|_\cX$.
Following  the setting of Proposition \ref{hom ext p}, if the force term $f$ satisfies, e.g. the assumptions of \cite[Proposition 2.14]{DaHiMuTo}, this is, $f\in L^p(\R_+; X)$ with $[f]_{\Phi, \cA} \in L^p(\R_+)$ $(p\in [1,\infty])$, then $J f \in L^p(\R_+, \widetilde X)$. Hence, knowing that $\widetilde A$ on $\widetilde X$ has $L^p$-maximal regularity on $\R_+$, we get that $(U_{\widetilde A} Jf)'$ is in  $L^p(\R_+, \widetilde X)$. Since $u: = U_\cA f$ is a strong solution to $(CP)_{A,f,0}$ with $u, u' \in L^p_{loc}(\R_+; X)$, we get that $[u' (t)]_{\Phi, \cA} = \|J(U_A f)' (t)\|_{\wX} = \|(U_{\widetilde A} Jf)'(t)\|_{\widetilde X}$ for a.e. $t>0$, that is, $ [u']_{\Phi, \cA} \in L^p(\R_+)$, which is a desired conclusion therein. 
\end{remark}

In the proofs of the estimates in \eqref{E mr est}, we use the Banach completion as a tool. To make statement (a) of Proposition \ref{hom ext p} more adequate for application, it is useful to ensure that if a normed space $X$ is a subspace of a Hausdorff vector space $Y$ then there exists a completion of $X$ in $Y$. That is, 
a completion $\widetilde{X}$ such that $\widetilde{X} \subset Y$, where the inclusion map is continuous (denoted $\widetilde{X} \hookrightarrow Y$). If such a~completion of $X$ in $Y$ exists 
 we also say that $X$ admits a completion in $Y$.
 Following \cite{ArGa65}, a {\it completion} of $X$ is a Banach space $\widetilde{X}$ containing  an isometric copy of $X$ as a dense subspace and satisfying 
$\|x\|_{\widetilde{X}} = \|x\|_X$ for all $x \in X$. 
The problem of the existence of a completion of $X$ in $Y$ is resolved by the following result. 

\begin{theorem}[{\cite[Theorem 3.1]{ArGa65}}] \label{CauchyComp} 
If a normed space $X$ is a subspace of a Hausdorff vector space $Y$, then $X$ admits a completion in $Y$ if and only if the following conditions are satisfied{\rm:}
\begin{itemize}
\item[{\rm (i)}] $X \hookrightarrow Y$ and every Cauchy sequence in $X$ converges in $Y$.
\item[{\rm (ii)}] every Cauchy sequence in $X$ convergent to $0$ in $Y$ is also convergent to $0$ in~~$X$. 
\end{itemize}
\end{theorem}

Note that Theorem \ref{CauchyComp} is stated in \cite{ArGa65} without a proof. As we require certain details, we outline the essential steps of the proof briefly in the following lemma leaving the remaining details to the reader. Let $X$ be a normed space and $Y=(Y, \tau)$ a Hausdorff vector space such that $X\hookrightarrow Y$. 
 The \textit{Cauchy completion} $X^{ac, Y}$ (abbreviated as $X^{ac}$) is defined as the space of all $x \in Y$ for which there exists a Cauchy sequence $(x_n)_{n \in \mathbb{N}}$ in $X$ that converges to $x$ in $Y$. It should be noted that $X^{ac}$  is, in general, not a completion of $X$ in the classical sense. It is straightforward to verify that $X^{ac}$ is a normed space, equipped with the norm
\[
\|x\|_{X^{ac}} := \inf_{(x_n)} \sup_{n \geq 1} \|x_n\|_{X},
\]
where the infimum is taken over all sequences $(x_n)$ satisfying the aforementioned conditions. 

\begin{lemma}  \label{AC}
Let $X$ be a normed space and $Y = (Y, \tau)$ a Hausdorff vector space. If $X \hookrightarrow Y$ and every Cauchy sequence in $X$ converges in $Y$, then $X \hookrightarrow X^{ac} \hookrightarrow Y$, and the inclusion map ${\rm{id}} \colon X \to X^{ac}$ has norm at most $1$. Moreover, $X^{ac}$ is a Banach space. If, in addition, every Cauchy sequence in $X$ that converges to $0$ in $Y$ also converges to $0$ in $X$, then $X^{ac}$ is the completion of $X$ in $Y$.
\end{lemma}

We also require the following result. For clarity, all Cauchy completions below are considered in $Y$. Additionally, we use the symbol '$\overset{1}{\hookrightarrow}$' to indicate that the inclusion map has norm less than or equal to $1$.

\begin{lemma} \label{sumAC}
Let $X_0$ and $X_1$ be normed spaces, and let $Y = (Y, \tau)$ be a Hausdorff vector space such that $X_j \hookrightarrow Y$ and every Cauchy sequence in $X_j$ converges in $Y$ for $j = 0, 1$. Then, the pair $(X_0^{ac}, X_1^{ac})$ forms a Banach couple such that $X_0^{ac} + X_1^{ac} = (X_0 + X_1)^{ac}$ with equal norms.
\end{lemma}
\begin{proof} 
From Lemma \ref{AC}, it follows that $X_j^{ac}$ are Banach spaces and $X_j^{ac} \hookrightarrow Y$ ($j=0, 1$). 
Thus, $(X_0^{ac}, X_1^{ac})$ forms a Banach couple. To show the required isometric formula, we start with an auxiliary observation. Suppose that $(Z_0, Z_1)$ is a couple of normed spaces which satisfies the same conditions as $(X_0,X_1)$ and, in addition, $Z_0\overset{1}{\hookrightarrow} Z_1$.     
We claim that 
\begin{equation}\label{inclusion ac}
Z_0^{ac} \overset{1}{\hookrightarrow} Z_1^{ac}\,.
\end{equation}
Clearly, we have $
Z_0 \overset{1}{\hookrightarrow} Z_1 \overset{1}{\hookrightarrow} Z_1^{ac}\,$. For a given $x \in Z_0^{ac}$, choose an arbitrary Cauchy sequence $(x_n)$ in $Z_0$ such that $x_n \to x$ in $Y$. 
By the completeness of $Z_1^{ac}$, it follows that $x_n \to y$ in $Z_1^{ac}$ and so $x_n \to y$ in $Y$. Hence $y = x$ and 
\[
\|x\|_{Z_1^{ac}} = \lim_{n \to \infty} \|x_n\|_{Z_1^{ac}} \leq \sup_{n \geq 1} \|x_n\|_{Z_0}\,.
\]
This estimate, combined with the fact that $(x_n)$ is an arbitrary sequence in $Z_0$ with the mentioned properties, implies that 
$x \in Z_1^{ac}$ with 
\[
\|x\|_{Z_1^{ac}} \leq \|x\|_{Z_0^{ac}}\,,
\]
and so the claim follows.

 To establish $X_0^{ac} + X_1^{ac} = (X_0 + X_1)^{ac}$ (equal norms) 
  observe that $X_j \overset{1}{\hookrightarrow} X_j^{ac}$ ($j=0,1$) yields
\begin{equation}\label{inclusion}
X_0 + X_1 \overset{1}{\hookrightarrow} X_0^{ac} + X_1^{ac}\,.
\end{equation}
Since $X := X_0^{ac} + X_1^{ac}$ is a Banach space, its Cauchy completion in $Y$ coincides isometrically with $X$. 
Combining this with \eqref{inclusion}, we obtain 
\[
(X_0 + X_1)^{ac} \overset{1}{\hookrightarrow} X_0^{ac} + X_1^{ac}\,.
\]
Now, observe that $X_0 + X_1$ is a normed space and $X_j \overset{1}{\hookrightarrow} X_0 + X_1$ ($j=0,1$). Then, by applying the first part of the proof to $(Z_0, Z_1) = (X_j, X_0 + X_1)$, $j=0,1$,  
we obtain $X_j^{ac} \overset{1}{\hookrightarrow} (X_0 + X_1)^{ac}$. Consequently, this yields the reverse inclusion 
\[
X_0^{ac} + X_1^{ac} \overset{1}{\hookrightarrow} (X_0 + X_1)^{ac}\,,
\]
which completes the proof. 
\end{proof}

As an application, we obtain the following corollary.

\begin{proposition} \label{K-space}
Let $\xo = (X_0, X_1)$ be a normed couple and let $Y = (Y, \tau)$ be a sequentially complete Hausdorff vector space such that 
$X_j \hookrightarrow Y$ for $j = 0, 1$. If $(X_0 + X_1)^{ac}$  is a completion of $X_0+X_1$ in $Y$, then for 
any parameter $\Phi$ of the $K$-method, the following statements hold{\rm:}
\begin{itemize} 
\item[{\rm(i)}] The normed space $K_\Phi(\xo)$ admits a completion in $Y$. More precisely, the space $K_\Phi(\xo)^{ac} \subset Y$ is a completion of $K_\Phi(\xo)$. 
\item[{\rm(ii)}] The inclusion map from ${K_\Phi(\xo)}^{ac}$ into $K_{\Phi}(X_0^{ac}, X_1^{ac})$ is an isometry. 
If, in addition, 
\[
 K_\Phi(\xo) \text{ is dense in } K_\Phi(X_0^{ac}, X_1^{ac}),
\]
then 
\begin{equation}\label{tilde vs functor}
K_\Phi(\xo)^{ac} = K_\Phi(X_0^{ac}, X_1^{ac})
\end{equation} 
with equal norms. In particular, if $K_\Phi$ is a regular functor and $X_0^{ac}\cap X_1^{ac} = (X_0\cap X_1)^{ac}$, then \eqref{tilde vs functor} is true.
\end{itemize} 
\end{proposition}

Note that the condition $X_0^{ac}\cap X_1^{ac} = (X_0\cap X_1)^{ac}$ in (ii) above holds, e.g. when  $X_0\hookrightarrow X_1$.

\begin{proof}
(i) Fix a parameter $\Phi$ of the $K$-method. Observe that 
\begin{equation}\label{incl for comp}
K_\Phi(X_0, X_1) \hookrightarrow X_0 + X_1 \hookrightarrow  (X_0 + X_1)^{ac} \hookrightarrow Y.
\end{equation}
Therefore, since $Y$ is sequentially complete, it suffices to verify that the condition (ii) of Theorem~\ref{CauchyComp} is satisfied.
Let $(x_n)$ be a Cauchy sequence in $K_\Phi(X_0, X_1)$ such that $x_n \to 0$ in $Y$. Clearly, for each $n, m \in \mathbb{N}$ 
and all $t > 0$, we have
\[
\big|K(t, x_n; \xo) - K(t, x_m; \xo)\big| \leq K(t, x_n - x_m; \xo)\,.
\]
Thus, $(K(\cdot, x_n; \xo))$ is a Cauchy sequence in the Banach space $\Phi$. Therefore, there exists $f \in \Phi$ such that
\[
\|K(\cdot, x_n; \xo) - f\|_\Phi \to 0\,.
\]
Since $\Phi \hookrightarrow L^0(\mathbb{R}_+, \ud t)$, it follows from the Riesz theorem that there exists a subsequence $(x_{n_k})$ of $(x_n)$ satisfying
\[
K(t, x_{n_k}; \xo) \to f(t) \quad \text{for a.e. $t > 0$}\,.
\]

By \eqref{incl for comp}, $x_n \to 0$ in $X_0 + X_1$ and, consequently, for every $t > 0$,
\[
K(t, x_n; \xo) \to 0 \quad \text{as $n \to \infty$}\,.
\]
Thus,  $f(t) = 0$ for a.e. $t > 0$, that is, $x_n \to 0$ in $K_\Phi(X_0, X_1)$ as desired. 

(ii) By Lemma \ref{sumAC}, for every $t >0$, we have 
\[
({X_0 + t X_1})^{ac} = X_0^{ac} + t X_1^{ac} 
\]
with equal norms, where $X_0 + t X_1$ denotes $X_0 + X_1$ equipped with the norm $K(t, \cdot\,; \vec{X})$. 
Hence, for any $x \in X_0 + X_1$, we have
\[
K(t, x; X_0, X_1) = \|x\|_{(X_0 + t X_1)^{ac}} = K(t, x; X_0^{ac}, X_1^{ac}), \quad\, t > 0\,.
\]

In particular, this implies that for all $x \in K_\Phi(X_0, X_1)$, we have
\[
\|x\|_{K_{\Phi}(X_0, X_1)} = \|x\|_{K_\Phi(X_0^{ac}, X_1^{ac})}\,.
\]
Since $K_\Phi(\xo) \hookrightarrow K_\Phi(X_0^{ac}, X_1^{ac})$ and $K_\Phi(\xo)$ is a dense subset of its completion, 
we obtain the desired claim, as well as the first part of the additional statement. For the second part, recall that the 
regularity of $K_\Phi$ implies that $X_0 \cap X_1$ and $X_0^{ac} \cap X_1^{ac}$ are dense in $K_\Phi(\xo)$ and 
$K_\Phi(X_0^{ac}, X_1^{ac})$, respectively. Moreover, the assumption that $X_0^{ac} \cap X_1^{ac} = (X_0 \cap X_1)^{ac}$ 
ensures that $X_0 \cap X_1$ is dense in $X_0^{ac} \cap X_1^{ac}$. Therefore, $X_0 \cap X_1$ is dense in 
$K_\Phi(X_0^{ac}, X_1^{ac})$, as required. This completes the proof.
\end{proof}

We apply the result above to the context considered in this paper and formulate the {\it a priori} conditions in terms of the 
injective operator $\cA$ on $\cX$, in particular, which guarantee that  the spaces $X = K_\Phi(\cX,D_\cA)$ from the formulation of 
Theorems \ref{thm main} and \ref{thm main 2} equipped with the homogeneous part of the corresponding interpolation norms admit completions in a given vector space.

In what follows if $\cA$ is an injective linear operator on a Banach space $\cX$, then $\dot{D}_\cA$  denotes the domain $D(\cA)$ of $\cA$ equipped with the 
norm $\|x\|_{\dot{D}_\cA}:= \|\cA x\|_{\cX}$. 
Recall that a linear operator $\cA$ on a topological vector space $Y$ is called {\it closable in $Y$} if admits a closed extension. One can show that it holds if and only if for every sequence $(x_n)$ in $D(\cA)$ such that $x_n\rightarrow 0$ and $\cA x_n \rightarrow y \in Y$ one has $y = 0$. 

\begin{proposition} \label{Cauchy}
Let $\cX$ be a Banach space and let $Y = (Y, \tau)$ be a sequentially complete Hausdorff vector space.  Let $\cA$ be an injective operator on $\cX$. Then the following statements hold{\rm:}
\begin{itemize}
\item[{\rm (i)}] If $\cA$ is closable in $\cX$, then the pair $(\cX, \dot{D}_\cA)$ is a couple of normed spaces.
\item[{\rm (ii)}] If $\cX \hookrightarrow Y$, $\dot{D}_\cA \hookrightarrow Y$ and $\cA$ is closable in $Y$, then $\dot D_\cA$ has a~completion~in~$Y$.
\item[{\rm (iii)}] If  
$\cX \hookrightarrow Y$, $\dot{D}_\cA \hookrightarrow Y$,  $\cA$ is closable in $Y$, and $\cX + \dot D_\cA$ admits a completion in $Y$, then for any parameter $\Phi$ of the $K$-method, 
the normed space $K_\Phi(\cX, \dot D_\cA)$ also admits a completion in $Y$.
\end{itemize}
\end{proposition}

\begin{proof}
(i) It is easy to verify (cf.~\cite[Proposition 2.1.7, p.~94]{BrKr91} for the case of Banach spaces) that if $X_0$ and $X_1$ are normed spaces linearly embedded into a vector space 
$\mathcal{V}$, and the following consistency condition holds: whenever a sequence $(x_n)$ in $X_0 \cap X_1$ converges to $y_j$ in the normed space $X_j$ for $j = 0, 1$, 
we have $y_0 = y_1$, then $(X_0, X_1)$ forms a couple. 
Since $\cA$ is a closable operator, it is clear that the normed spaces $\dot{{D}}_\mathcal{A}$ and $\cX$ satisfy the consistency condition. Thus, it follows that $(\cX, \dot D_\cA)$ 
forms a couple of normed spaces.

For (ii), let $(x_n)$ be a Cauchy sequence in $\dot D_\cA$ such that $x_n \to 0$ in $Y$. Since $(\cA x_n)$ is a Cauchy
sequence in $\cX$, there exists $y \in \cX$ such that
\[
\|\cA x_n - y \|_{\cX} \to 0\,.
\]
Hence $\cA x_n \to y$ in $Y$ by $\cX \hookrightarrow Y$. Since $x_n \to 0$ in $Y$ and $\cA$ is closable in $Y$, it follows that $y = 0$. Therefore, we have
\[
\|x_n\|_{\dot D_\cA} = \|\cA x_n \|_{\cX} \to 0\,.
\]
Thus, the result follows from the Theorem \ref{CauchyComp}. Finally, the statement (iii) is a~~direct consequence of Proposition \ref{K-space}. The proof is complete.
\end{proof}

We conclude with a result that demonstrates that under certain conditions, the normed space $\cX + \dot D_\cA$ has a completion in a given space $Y$.

\begin{proposition} \label{XD}
Let $\cX$ be a Banach space, $(Y, \tau)$ a sequentially complete Hausdorff vector space, and $\cA$ an injective operator on $\cX$ such that $\cX \hookrightarrow Y$ 
and $\dot D_\cA \hookrightarrow Y$. Assume that $\cA$ is closable in $Y$ and generates a bounded semigroup $\cT$ on $\cX$ such that for every sequence $(x_n)\subset \cX$ such that $x_n\rightarrow 0$ in $Y$, $\cT(t) x_n \rightarrow 0$ in $Y$ for all $t\in (0,1]$. 
Then, the space $\cX + \dot D_\cA$ admits a completion in $Y$. 
In particular, the conclusion of Proposition $\ref{Cauchy}\,(iii)$ holds.
\end{proposition}

\begin{proof}
Since $\cX \hookrightarrow Y$ and $\cA$ is closable in $Y$, $\cA$ is closable in $\cX$. Therefore, it follows from Proposition \ref{Cauchy}(i) that $(\cX, \dot D_\cA)$ 
forms a couple of normed spaces. This implies that $\cX + \dot\cD_\cA$ is a normed space, and we have
\[
\cX + \dot D_\cA \hookrightarrow Y\,.
\]
From Lemma \ref{hom isom}(ii), for every $x \in \cX + \dot D_\cA$, we have
\begin{equation}\label{repr of sum norm}
\|x\|_{\cX+\dot D_\cA} = K(1, x; \cX, \dot D_\cA) \simeq \sup_{0 < t \leq 1} \|\cT(t)x - x\|_{\cX}.
\end{equation}
Let $(x_n)$ be a Cauchy sequence in $\cX + \dot \cD_\cA$ such that $x_n \to 0$ in $Y$. From the above equivalence,  
it follows that $(\cT(\cdot) x_n - x_n)$ is a Cauchy sequence in the Banach space $C_b((0, 1]; \cX)$. Thus, there exists $g \in C_b((0, 1]; \cX)$ such that
\begin{equation}\label{conv 1}
\sup_{0<t \le 1} \|\cT(t) x_n - x_n - g(t)\|_{\cX} \to 0  \quad\, \text{as $n \to \infty$}\,.
\end{equation}
Therefore, by our assumptions, for all $t\in (0, 1]$ we get that $\cT(t) x_n - x_n - g(t)\rightarrow 0$ and  $\cT(t)x_n\rightarrow 0$ in $Y$. 
Consequently, $g(t) = 0$ for every $t \in (0, 1]$, and combining \eqref{repr of sum norm} with \eqref{conv 1} yields $\|x_n\|_{\cX + \dot D_\cA} \rightarrow 0$. 
Hence, by Theorem \ref{CauchyComp}, we get that $\cX+\dot D_\cA$ admits a completion in $Y$, which finishes the proof.
\end{proof}

Note that by combining Propositions \ref{XD} and \ref{Cauchy}, and Corollary \ref{K-space}, a conclusion similar to that of Lemma \ref{hom isom}(iii) holds, 
provided $Y$, $\cX$, $\cA$, and $\cT$ satisfy the assumptions of Proposition \ref{XD}. More precisely, for any $\Phi$ as in (iii) of Lemma \ref{hom isom}, the 
Cauchy completion $K_\Phi(\cX, \dot D_\cA)^{ac}$ of $K_\Phi(\cX, \dot D_\cA)$ coincides with $K_\Phi(\cX, (\dot{D}_\cA)^{ac}) \subset Y$.

We conclude with two observations regarding Theorems \ref{thm main} and \ref{thm main 2}. From the perspective of interpolation theory, Theorems 
\ref{thm main}, \ref{thm main 2}, and \ref{optimal for l1} can be considered as an extension of the Levy-type result describing those parameters $\Phi$ of the $K$-method 
for which the corresponding interpolation spaces $K_\Phi(\cX,\cY)$ contain an isomorphic copy of $\ell_1$ or $c_0$; see e.g. \cite[Theorems 4.6.25 and 4.6.28]{BrKr91}.

\medskip

\noindent 
In general, one cannot expect that the conclusions of Theorem \ref{thm main} hold for other parameters $\Phi$ of the $K$-method. In fact, for some parameters $\Phi$ (e.g.
those corresponding to the real interpolation method, i.e. $\Phi = L^q(\R_+, t^{\mu}\ud t)$ with $\mu\in (-1, q-1)$ and $q\in (1,\infty)$), the $U\!M\!D$ property of the 
underlying space $\cX$ is inherited by $K_\Phi(\cX, D_{\cA}) = (\cX, D_\cA)_{1-\frac{\mu+1}{q}, q}$; see Rubio de Francia \cite{RdeF86}.

\section{The interpolation approach}\label{interp mr}

In this section we provide a useful interpolation framework for the study of the maximal regularity estimates in other interpolation spaces than those considered 
in Theorems \ref{thm main} and \ref{thm main 2}. 
     
Let $(X_0, X_1)$ be a Banach couple and let $A_j$ be  a closed linear operator on $X_j$ ($j=0, 1$). Let $D_{A_j}$ for $j=0,1$ denote the domain $D(A_j)$ of $A_j$ 
equipped with the graph norm. 
Note that $(D_{A_0}, D_{A_1})$ forms a Banach couple. 
Suppose $A_0 x = A_1 x$ for all $x \in D(A_0) \cap D(A_1)$. We define 
\begin{equation}\label{def of A}
A x:= A_0 x_0 + A_1 x_1
\end{equation}
for all $x= x_0 + x_1$ with $x_j \in D(A_j)$ for $j=0, 1$. It is easy to see that $A$ is the unique linear mapping with domain 
$D(A) = D(A_0) + D(A_1) \subset X_0 + X_1$ such that  $A|_{D(A_j)} = A_j$  for $j=0,1$.  Given an  interpolation functor $F$,  let $A_{F}$ denote the part of $A$ 
in the corresponding interpolation space $F(X_0, X_1)$.  In the first result we identify the domain of the part $A_F$ of $A$ in $F(X_0,X_1)$ for an arbitrary 
interpolation functor $F$. 

\begin{proposition} \label{intDom}
Let $(X_0, X_1)$ be a Banach couple, and let $A_j$ be a closed operator on $X_j$ for $j = 0, 1$. Assume that $A_0 x = A_1 x$ 
for all $x \in D(A_0) \cap D(A_1)$, and suppose there exists $\lambda \in \rho(A_0) \cap \rho(A_1)$ such that $(\lambda - A_0)^{-1}$ and 
$(\lambda - A_1)^{-1}$ coincide on $X_0 \cap X_1$. ~Then, $\la \in \rho(A)$ and for any interpolation functor ${F}$ the part $A_{{F}}$ of $A$ in ${F}(X_0, X_1)$ 
is a closed operator and 
\[
{F}(D_{A_0}, D_{A_1}) = D_{A_{{F}}}
\]
up to equivalence of norms. 
\end{proposition}

\begin{proof}
Since $(\lambda - A_0)^{-1}$ and $(\lambda - A_1)^{-1}$ agree on $X_0 \cap X_1$, the map $S$ given by 
\[
S x :=  (\lambda - A_0)^{-1}x_0  + (\lambda - A_1)^{-1} x_1 
\]
for all $x = x_0 + x_1$ with $x_0 \in X_0$ and $x_1 \in X_1$ is a (well-defined) linear operator on $X_0+X_1$. 
Clearly, $S|_{X_j} = (\lambda - A_j)^{-1}$ for $j=0,1$ and $S(X_0 + X_1) = D(A_0) + D(A_1)$. 
In particular, $S$ is an admissible with respect to the couples  $(X_0, X_1)$ and $(D_{A_0}, D_{A_1})$.  In particular, it gives that $S$ is bounded on $X_0+X_1$ and its restriction $S_|$ to $F(X_0, X_1)$ is a bounded operator from $F(X_0, X_1)$ to  $F(D_{A_0}, D_{A_1})$. We show that $S_|$ is invertible. Indeed, since $(\lambda - A_0)x = (\lambda - A_1)x$ for all $x \in D(A_0) \cap D(A_1)$, $S$ is injective. It is readily seen that $S^{-1} = \lambda - A$, where $A$ is given by \eqref{def of A} above. It shows that $\lambda \in \rho(A)$ and 
\[
S^{-1}|_{D(A_j)} = \lambda - A_{j}, \quad\, j = 0, 1\,.
\]
Therefore,  the operator $S^{-1}$ is admissible with respect to the couples $(D_{A_0}, D_{A_1})$ and $(X_0, X_1)$. 
Further, for a given $y \in F(D_{A_0}, D_{A_1}) \subset D(A_0) + D(A_1)$, there exists $x \in X_0 + X_1$ such that $S x = y$. By the interpolation property of $F$, we conclude that $x = S^{-1}y \in F(X_0, X_1)$. It yields the desired claim on $S_|$. 
To show that $A_{F}$ is a closed operator on  $F(X_0, X_1)$ let $(x_n)$ be a sequence in $D(A_{F})$ 
with $x_n \to x$ in $F(X_0, X_1)$ and $A_{F} x_n = A x_n \to y$ in $F(X_0, X_1)$. Clearly, for each $n \geq 1$, we have
\[
\lambda (\lambda - A)^{-1} x_n - (\lambda - A)^{-1} A x_n  = x_n\,.
\]
Since $S_|=(\lambda - A)_|^{-1} \colon F(X_0, X_1) \to F(D_{A_0}, D_{A_1})$ is continuous, it follows that 
\[
\lambda (\lambda - A)^{-1} x - (\lambda - A)^{-1} y  = x\,.
\]
Hence, $x\in D(A)\cap F(X_0,X_1)$ and $(\lambda - A)^{-1} (\lambda x - y) = x$ yields $\lambda x - y = (\lambda - A) x$, so $A x = y$ as required. Further, combining $D_{A_j} \hookrightarrow X_j$ ($j = 0, 1$) with the interpolation property of $F$ we get  
\[
F(D_{A_0}, D_{A_1}) \hookrightarrow F(X_0, X_1)\,.
\] 
Note that $D(\cA_F) = S(F(X_0, X_1)) = F(D_{A_0}, D_{A_1})$, and for all $x \in F(D_{A_0}, D_{A_1})$, one has
\[
\gamma \|x\|_{F(D_{A_0}, D_{A_1})} \leq \|x - A x\|_{F(X_0, X_1)} \leq \|x\|_{F(X_0, X_1)} + \|A x\|_{F(X_0, X_1)}\,, 
\]
where $\gamma =  1/\|S_|\,\|_{\cL(F(X_0, X_1),F(D_{A_0}, D_{A_1}))}$. 
Combining this with the above inclusion, we conclude that
\[
F(D_{A_0}, D_{A_1}) = D_{A_{F}}
\]
up to equivalence of norms, and this completes the proof.
\end{proof}

\begin{theorem}\label{interp result} Let $(X_0, X_1)$ be a Banach couple.
Let $-A_0$ and $-A_1$ be the generators of the bounded holomorphic semigroups $T_0$ and $T_1$ on $X_0$ and $X_1$, respectively, which are consistent in the sense that  
$T_0(t)x = T_1(t)x$ for all $x\in X_0\cap X_1$ and $t>0$. Suppose that $A_0$ and $A_1$ are consistent on $D(A_0) \cap D(A_1)$ and have $L^p$-m.r. on $(0,\tau)$ for some $\tau\in (0,\infty]$ and $p\in [1,\infty)$. 

Then, for any interpolation functor $F$ such that, up to the equivalence of the norms,
\begin{equation}\label{spliting}
F(\ell^p(X_0), \ell^p(X_1)) = \ell^p(F(X_0, X_1)),
\end{equation}
the part of the operator $A_0 + A_1$ in $F(X_0, X_1)$ has $L^p$-m.r. on $(0,\tau)$. 
\end{theorem}

Below, for any bounded linear operator $S$ on a Banach space $X$ and any sequence $\mathbf{x}=(x_n)_{n\in \N}$ in $X$ we set
\begin{equation}\label{def of K}
\bK_S \mathbf{x}:= (y_n)_{n\in\N}, \qquad y_n := \sum_{k=1}^n S^{n-k}(I-S) x_k,  \qquad n\in \N. 
\end{equation}

\begin{proof} 
Let  $X_0+X_1$ be equipped with the standard sum norm $\|\cdot\|_{X_0+X_1}$.
On $X_0+X_1$ we define the family of operators $(T(t))_{t>0}$  as follows: for $x = x_0+x_1$ with $x_0\in X_0$ and $x_1\in X_1$ set 
\[
T(t)x := T_0(t)x_0 + T_1(t)x_1. 
\]
By the consistency of semigroups $T_0$ and $T_1$, one can check that $T$ is a well-defined family of bounded operators on $X_0+X_1$, satisfying the semigroup property and having a bounded analytic extension to some sector around positive real axis (note that $T$ is a $C_0$-semigroup if $T_0$ and $T_1$ are). For $T$ to be a semigroup (according to the definition in \cite[p.~128]{ABHN01}; see Remark \ref{Dore results}), we need to check that for $x \in X$ if $T(t)x = 0$ for all $t>0$ then $x=0$. Recall that the resolvent of $-A_i$ is represented by the Laplace transform of $T_i$, that is, for all $\la$ large enough, say $\la>\omega>0$, we have 
\begin{equation}\label{resolv of Ai}
(\lambda + A_i)^{-1}x = \int_0^{\infty} e^{-\lambda t}T_i(t)x \ud t, \qquad  x \in X_i\,,
\end{equation}
with convergence in the norm of $X_i$ ($i=0, 1$).
Clearly, for all $\la >\omega$ the operators $(\lambda + A_0)^{-1}$ and $(\lambda + A_1)^{-1}$ coincide on $X_0 \cap X_1$. Since $A_0$ and $A_1$ are consistent on $D(A_0) \cap D(A_1)$, Proposition \ref{intDom} yields that $(\omega, \infty)\subset \rho(A)$, where $A$ is given by \eqref{def of A}.

Since $X_0$ and $X_1$ are continuously embedded in $X_0 + X_1$, by \eqref{resolv of Ai}, for all $\lambda > \omega$ and for all $x = x_0 + x_1 \in X_0 + X_1$, we get
\[
(\lambda + A)^{-1}x  = (\lambda +A_0)^{-1}x_0 + (\lambda +A_1)^{-1}x_1 = \int_0^\infty e^{-\lambda t} T(t) x \ud t,
\]
where the integral is convergent in the norm of $X_0 + X_1$.  According to \cite[Definition 3.2.5]{ABHN01},  $T$ is a semigroup on $X_0+X_1$ with the generator $-A$. 
Let for some $p\in [1,\infty)$, $A_i$ has $L^p$-m.r. on $(0,\tau)$. By the rescaling argument (see the proof of Lemma \ref{lem on hom est}(ii)), we can consider only the case $\tau = \infty$. 
By the Kalton-Portal result \cite[Proposition 2.2]{KaPo08}, 
the operators $\bK_{T_i(t)}$, $t\in (0,1)$, are uniformly bounded on $l^p(X_i)$  $(i=0,1)$.   

Let $F$ be an interpolation functor. Set $X_F := F(X_0,X_1)$ and let $T_F(t):=T(t)_{|X_F}$, $t\geq 0$. Of course, $T_F$ is a semigroup on $X_F$. Denote its generator by $-B$.  
We show that $B$ is equal to the part $A_F$ of $A$ in $X_F$. Indeed, since 
$X_F$ is continuously embedded in $X_0+X_1$, by the same argument as above, 
\[
(\lambda + B)^{-1} y = \int^\infty_0 e^{-\lambda t} T_F(t) y \ud t = \int^\infty_0 e^{-\lambda t} T(t) y \ud t = 
(\lambda + A)^{-1}y = x.
\] for each $x\in X_F$ and  all $\lambda$ large enough. In particular, it gives that for $x\in D(B)\subset X_F$ and $y\in X_F$ such that $x = (\lambda + B)^{-1}y$, we have $x\in D(A)$ and $Ax = y - \lambda x \in X_F$ and $B x = A x$. Hence, $D(B)\subset D(A_F)$ and $B\subset A_F$. 
Conversely,  if $x\in D(A_F)$, i.e. $x\in X_F \cap D(A)$ with $Ax \in X_F$, then $y:= \lambda x + A x \in X_F$. Therefore, $(\lambda + B)^{-1} y  = (\lambda + A)^{-1}y = x$.
Consequently, $x\in D(B)$ and $Bx = A_Fx$. It yields that $B = A_{F}$.

Since for each $t>0$ the operators $\bK_{T_0(t)}$ and $\bK_{T_1(t)}$ are consistent on $\ell^p(X_0)\cap \ell^p(X_1)$, $\bK_{T(t)}$ is an admissible operator on a Banach couple  $(\ell^p(X_0),\ell^p(X_1))$. Consequently, 
its restriction to $F(\ell^p(X_0),\ell^p(X_1))$, which, of course, coincides with the operator $\bK_{T_F(t)}$,  is a bounded operator on $F(\ell^p(X_0),\ell^p(X_1))$.
Since we suppose that $F$ \emph{splits} $\ell^p$, i.e. $ 
F(\ell^p(X_0),\ell^p(X_1)) = \ell^p(X_F)$ (equivalent norms), again by \cite[Proposition 2.2]{KaPo08}, $A_F$ has $L^p$-m.r. on $\R_+$ in $X_F$.    
\end{proof}

We conclude with a simple application of Proposition \ref{interp result}. Let $[\,\cdot\,]_{\theta}$, $\theta \in (0,1)$, stands for the functor of the complex method; see, e.g. \cite[Chapter 4]{BeLo76}. Recall that those functors split $\ell^p$ spaces, that is, \eqref{spliting} holds; see \cite[Theorem 5.1.2, p. 107]{BeLo76}.

\begin{corollary}\label{cor1} Let $\cA$ be a linear operator on a Banach space $\cX$ with the $L^p$-m.r. on $\R_+$ for some $p\in [1,\infty)$. 

Then,  for all interpolation functors $F$, which splits $\ell^p$, the part of $\cA$ in $F(\cX,D_\cA)$ has $L^p$-m.r. on $\R_+$. 
In particular, In particular, this is the case when 
\[
F\in \bigl\{ [\,\cdot\,]_{\theta},\, (\,\cdot\,)_{\theta,q}\, : \theta \in (0,1), q\in (1,\infty) \bigr\}. 
\]
\end{corollary}
\begin{proof} Recall that for any $\lambda \in \rho(\cA)$, the resolvent operator $R(\lambda,A)$ is an isomorphism between $(\cX,\|\cdot\|_\cX)$ and $(D_\cA, \|\cdot\|_{D_\cA})$. It is straightforward to see, that the $L^p$-m.r. on $\R_+$ of the part of $\cA$ in $D_\cA$ follows from that one of $\cA$ in $\cX$. Therefore, taking $A_0 := \cA$ on $X_0 := \cX$ and $A_1$ as the part of $\cA$ on $X_1 := D_\cA$, the proof follows from Theorem \ref{interp result}. 
\end{proof}

Note that a similar consequence of Theorem \ref{interp result}, analogous to that of Corollary~\ref{cor1}, can be formulated in the context considered 
in Proposition \ref{hom ext p}, namely by taking $X_0$ as $(X, \|\cdot\|_X)$, $X_1$ as the completion $\wX$ of $(X, \|\cdot\|_*)$, and $(A_0, A_1) := (A, \wA)$, 
provided that $(X, \wX)$ forms a Banach couple. We do not elaborate further on this topic here.

We illustrate Corollary \ref{cor1} by applying it to the Stokes systems with Dirichlet boundary conditions. 
There is an extensive literature on the maximal regularity of the Stokes operator on  various domains $\Omega \subset \R^n$ under several boundary conditions. See the expository article by Hieber and Saal \cite{HiSa18} (which also includes results by Grubb and Solonnikov \cite{GrSo91}).

For Dirichlet boundary conditions, the most general result, to our knowledge, is provided in \cite[Corollary 2.2]{GeHeHiSa12}, which states that if the domain $\Omega$ has a {\it uniform} $C^3$-boundary and the Helmholtz decomposition exists for $L^q(\Omega)^n$, then the corresponding Stokes operator has $L^p$-m.r. on finite intervals for any $1 < p < \infty$ (see also \cite{GeKu15}). 
It appears that this result does not determine whether, for such general domains $\Omega$, the corresponding Stokes operator has $L^p$-m.r. on $\R_+$. However, this is the case for a large class of domains $\Omega$. For instance, it was proven in \cite[Theorem 7.6]{DeHiPr01}, that for $\Omega = \R^n_+,\R^n$, the corresponding Stokes operator admits a bounded $H^{\infty}$ calculus. Similarly, for bounded domains, exterior domains, or bent half-spaces (whose boundary is of class $C^3$), analogous results were proven in \cite{NoSa03}. Consequently, by the Dore-Venni theorem (see a variant in \cite[Theorem 4]{PrSo90}), the corresponding Stokes operator has $L^p$-m.r. on $\R_+$ for any $1 < p < \infty$. 

The description of the corresponding interpolation spaces between the solenoidal part $\cX:= L^q_\sigma(\Omega)$ of $L^q(\Omega)^n$ and the domain $D_\cA$ of the corresponding Stokes operator $\cA$ was given by Amann in \cite{Am00} (for so-called {\it standard} domains). Let $\Omega \subset \R^n$ denote any of the following domains: $\R^n$,  $\R^n_+$, exterior domain, bent half-space or bounded domain whose boundary is of class $C^3$.  

Recall that by \cite[Lemma 3.2 and Theorem 3.4]{Am00} for every $\theta \in (0,1)$ and $q\in (1,\infty)$
\[
[\cX, D_\cA]_\theta = [L^q(\Omega)^n, W^{2,q}_0(\Omega)^n]_{\theta} \cap L^q_{\sigma}(\Omega)  =: H^{2\theta,q}_{0,\sigma}(\Omega)
\]
\[
(\cX, D_\cA)_{\theta,r} = (L^q(\Omega)^n, W^{2,q}_0(\Omega)^n)_{\theta,r} \cap L^q_{\sigma}(\Omega)  =: B^{2\theta}_{q,r,0,\sigma}(\Omega) 
\]
For the exact description of those spaces (how it depends on the relation between parameters $p$ and $\theta$) we refer the reader to \cite[(2.19), p. 40]{Am00}
By means of Corollary \ref{cor1}, for any $p, q, r \in (1,\infty)$ and $s \in (0,2)$,  the Stokes operator on the resulting interpolation spaces $X:= H^{s,q}_{0,\sigma}(\Omega)$ and $X = B^{s}_{q,r,0,\sigma}(\Omega)$ has $L^p$-m.r. on $\R_+$.

\section{Illustrative examples}\label{sec app}

In this section, we illustrate selected results from previous sections by considering the negative Laplacian $\cA:=-\Delta$ on  $\cX:= L^p(\R^n)$ with $p\in [1,\infty)$ and $\cX = C_0(\R^n)$. Our main aim is to provide an intrinsic description of the homogeneous part $[\cdot]_{\Phi, \cA}$ of the norm of the real interpolation spaces $X = (\cX, D_\cA)_{\theta, q}$ with $\theta \in (0,1)$ and $q\in [1,\infty]$.  
In particular, it allows to identify the completion $\wX$ of $X$ with respect to  $[\cdot]_{\Phi, \cA}$  and 
show that Theorems \ref{thm main}, \ref{thm main 2}, Propositions \ref{hom ext p}, \ref{hom ext infty} and Corollary \ref{hom ext 1-p}, extend and unify relevant results from the literature; see Corollary \ref{final application}.

For $p\in [1,\infty]$, let $\cA = \cA_p$  denote the realization of the negative Laplacian $-\Delta$ on $\cX= \cX_p$, where $\cX_p:= L^p(\R^n)$ for $p\in [1,\infty)$ and $\cX_{\infty} := C_0(\R^n)$. Recall that $D_{\cA} = \{ f \in \cX_p ; \Delta f \in \cX_p\}$. Moreover,  for each $\theta \in (0,1)$, if $\Phi: = L^q(\R_+; t^{q(1-\theta) - 1} \ud t)$ for $q\in [1,\infty)$, or $\Phi: = L^\infty(t^{1-\theta})$, then 
\begin{equation}\label{X space}
X : = K_\Phi(\cX, D_\cA) = (\cX, D_\cA)_{\theta, q} = B^{2\theta}_{p,q}(\R^n),
\end{equation}
 see, e.g. \cite[Theorem 6.7.4, p.160]{BeLo76}. 
The representation of the (inhomogeneous) Besov spaces $X = B^{2\theta}_{p,q}(\R^n)$ stated in Proposition \ref{Komastu rep} is a description via {\it thermic extension}; see e.g. \cite[Section 2.6.4]{Tri92}. Since $\cA$ is injective, the corresponding homogeneous part 
\[
[\phi]_{\Phi, \cA}: = \bigl\| \Delta e^{\cdot\Delta}\phi \bigr\|_{\Phi(\cX)} , \qquad \phi \in  B^{2\theta}_{p,q}(\R^n),
\]
of the interpolation norm $\|\cdot\|_X$ is, in fact, the norm on $X$. 
We expect that the completion $\widetilde X$ of $(X, [\cdot]_{\Phi, \cA})$, which we discuss in Proposition \ref{hom ext p}, can be  identified with the {\it homogeneous} Besov space $\dot B^{2\theta}_{p,q}(\R^n)$. We do not know any references for it. Since it is crucial below, we provide a rigorous argument for it.

There are several definitions of homogeneous Besov spaces available in the literature and each of them offers advantages in specific situations. For the classical definition, we refer the reader to \cite{Tr83} (see also \cite{BeLo76}, \cite{Gr14}, \cite{Pe76}), while a more recent framework is introduced in \cite{BaChDa11}. In the classical approach, we deal with equivalence classes of distributions modulo polynomials, which presents some difficulties when it comes to solving nonlinear PDEs (for example, the lack of product rules). However, this approach guarantees the completeness of these spaces. In contrast, for the second definition, we only have completeness for a range of parameters, but the spaces are more suitable for applications. Let us state these definitions, along with a few properties and the relationships between them.

Let $(\phi_k)_{k \in \mathbb{Z}}$ be the Littlewood-Paley resolution of the identity on $\R^n \setminus \{0\}$, that is, $\phi$ is a smooth radial function such that $\phi \geq 0$, $\supp \phi \subset \{\xi : 2^{-1} \leq |\xi| \leq 2\}$, $ \phi_k : = \phi(2^{-k}\cdot)$ ($k\in \ZZ$) and $\sum_{k\in \ZZ} \phi_k(\xi) = 1$ for $\xi \neq 0$. Let $s\in \R$ and $p,q \in [1,\infty]$. The classical homogeneous Besov space $\dot{B}^s_{p,q,1}(\R^n)$ is defined as those tempered distributions $f\in \cS'(\R^n)$  modulo polynomials $\cP(\R^n)$ such that

\[
\|f\|_{\dot{B}^s_{p,q}} := \bigg( \sum_{k \in \mathbb{Z}} 2^{sjk} \|\cF^{-1} (\phi_k \cdot \cF f)\|^q_p \bigg)^{\frac{1}{q}} < \infty
\](with modification for $q=\infty$).
The more recent approach is presented in \cite{BaChDa11}. Let $\cS'_h(\R^n)$ denote the space of tempered distributions $f$ such that for any $\Theta \in C^{\infty}_0(\R^n)$ 
\[
\lim_{\lambda \to \infty} \|\cF^{-1} [\Theta (\lambda \cdot) \cF f]\|_{\infty} = 0. 
\]
The homogeneous Besov space $\dot{B}^s_{p,q,2}(\R^n)$ consists of those distributions $f$ in $\cS'_h(\R^n)$ such that $\|f\|_{\dot{B}^s_{p,q}} < \infty$. 
Let $\cS_\infty(\R^n)$ denote the space of all Schwartz functions which annihilate polynomials $\cP$. For $q=\infty$ we set:
 \[
\dot{b}^s_{p,1}(\R^n) : = \overline{\cS_\infty(\R^n)}^{\dot{B}^s_{p,\infty,1}}
\]
and
\[
\dot{b}^s_{p,2}(\R^n) : = \overline{\cS_\infty(\R^n)}^{\dot{B}^s_{p,\infty,2}}.
\]
We also refer the reader, for instance, to Bourdaud \cite{Bourd13} and Moussai \cite{Mous18} for a discussion on further {\it realizations} of homogeneous Besov spaces.
  
Below, we omit '$(\R^n)$' in the symbols of the Besov spaces and other spaces over $\R^n$. This should not result in any contradictions. 
Recall that,  the spaces $\dot{B}^s_{p,q,1}$ and $\dot{B}^s_{p,q,2}$ are well-defined (they do not depend on the choice of the sequence $(\phi_k)_{k \in \mathbb{Z}}$) and normed. The space $\dot{B}^s_{p,q,1}$ is complete for all $s \in \R$ (see \cite[Theorem 5.1.5, p. 240]{Tr83}), and the space $\dot{B}^s_{p,q,2}$ is complete, when $s < \frac{n}{p}$, or $s=\frac{n}{p}$ and $q=1$ (see \cite[Theorem 2.25, p.67]{BaChDa11}). 

To demonstrate that both classes, up to completion and isomorphism, contain the same objects, we employ the following standard set-theoretical argument. We skip the straightforward proof of this.

\begin{lemma}\label{isomoprhism of completions}
{\it Let $(B, \|\cdot\|_B)$ and  $(\dot B, \|\cdot\|_{\dot B})$ be normed spaces. Suppose that there exists a linear injection $j: B \rightarrow \dot B$ such that for some $c>0$ for all $b\in B$
\begin{equation}\label{J isom}
c^{-1} \|b\|_B\leq \|jb\|_{\dot B} \leq c \|b\|_B 
\end{equation}
and $jB$ is a dense subset of $\dot B$. Then the completions of $B$ and $\dot B$ are topologically isomorphic.}
\end{lemma}

Let $J : \cS' \to \cS' / \cP$ be the  quotient map, that is,  given by 
\begin{equation}\label{Definition of J}
Jf := f+ \cP,  \qquad f \in \cS',
\end{equation}
where $\cP$ is the collection of all polynomials on $\R^n$.

Note that $J$ is an isometry from $\dot{B}^s_{p,q,2}$ into $\dot{B}^s_{p,q,1}$.  
 For $q \neq \infty$, since $J\cS_\infty$ is dense in $\dot{B}^s_{p,q,1}$ and $\dot{B}^s_{p,q,1}$ is complete (\cite[Theorem 5.1.5, p.240]{Tr83}), applying Lemma \ref{isomoprhism of completions} we see that the completion of $\dot{B}^s_{p,q,2}$ is isometrically isomorphic to $\dot{B}^s_{p,q,1}$. Therefore, the spaces $\dot{B}^s_{p,q,1}$ and $\dot{B}^s_{p,q,2}$ are isomorphic whenever the former one is complete. Analogous statements hold for $\dot{b}^s_{p,2}$ and $\dot{b}^s_{p,1}$.
 From now on let
\[
\dot{B}^s_{p,q} := \dot{B}^s_{p,q,1} \simeq \widetilde{\dot{B}^s_{p,q,2}} , \qquad q\neq \infty,
\]
and
\[
\dot{b}^s_p := \dot{b}^s_{p,1} \simeq \widetilde{\dot{b}^s_{p,2}}.
\]

We apply the same argument to show that $\widetilde{X}$ and $\dot{B}^{2\theta}_{p,q}$ for $q \neq \infty$ are isomorphic, where $X$ is given by \eqref{X space}. Analogously, $\widetilde{X}$ and $\dot{b}^{2\theta}_p$ for $q=\infty$ are isomorphic. 
 Since 
$\cS_{\infty} \subset \cS \subset X$, by  Lemma \ref{isomoprhism of completions} above, it is sufficient to show that for all $f \in \cS_{\infty}$
\begin{equation}\label{equivalence of besov}
[f]_{\Phi,\cA_p} := \bigl\|\,\|\Delta e^{(\cdot)\Delta}f\|_{\cX_p}\bigr\|_{\Phi} \approx \|Jf\|_{\dot{B}^s_{p,q}} = \bigg( \sum_{k \in \mathbb{Z}} 2^{sjk} \|\cF^{-1} (\hat\phi_k \cdot \cF f)\|^q_{\cX_p} \bigg)^{\frac{1}{q}},
\end{equation} where $\Phi$, $\cA_p$, $\cX_p$ have the meaning specified above.
This is addressed in the subsequent result, Proposition \ref{equivalent norm on besov}.

We are now in a position to obtain \eqref{equivalence of besov}. The statement could be proved by combining the steps of the proof of \cite[Theorem 6.7.4]{BeLo76} and \cite[Theorem 3.5.3]{BuBe67}. The proof presented here is more direct.
 \begin{proposition} \label{equivalent norm on besov} Let $\cA_p$ be the realization of the negative Laplacian in $\cX_p$, where $p \in [1,\infty]$. Let $\dot D_{\cA_p}$ be the domain of $\cA_p$ equipped with the norm $\|\cA_p \cdot\|_{\cX}$. Let $J$ be as in \eqref{Definition of J}.
Then for every $q \in [1,\infty]$ and $\theta \in (0,1)$ there exists a constant $c>0$, such that for any $f \in \cS_{\infty}(\R^n)$ we have
\begin{equation}\label{equivalent norm besov}
c^{-1} [f]_{\Phi,\cA_p} \leq \|Jf\|_{\dot{B}^{2\theta}_{p,q}} \leq c [f]_{\Phi,\cA_p},
\end{equation}
where $\Phi = L^q(t^{q(1-\theta)-1})$ for $q \neq \infty$ and $\Phi = L^{\infty}(t^{1-\theta})$ for $q=\infty$.
More precisely,
\[
c^{-1} \sum_{k \in \mathbb{Z}} 2^{2qk\theta}\|\widecheck{\phi}_k \ast f\|^q_{\cX_p} \leq \int_0^{\infty} t^{q(1-\theta)-1} \|\Delta e^{-t\Delta}f\|^q_{\cX_p}  dt \leq c \sum_{k \in \mathbb{Z}} 2^{2qk\theta}\|\widecheck{\phi}_k \ast f\|^q_{\cX_p}
\]
in the case $q \neq \infty$ and 

\[
c^{-1} \sup_{k \in \mathbb{Z}} 2^{2k\theta} \|\widecheck{\phi}_k \ast f \|_{\cX_p} \leq \sup_{t>0} t^{1-\theta}  \|\Delta e^{-t\Delta} f\|_{\cX_p}  \leq c \sup_{k \in \mathbb{Z}} 2^{2k\theta}\|\widecheck{\phi}_k \ast f \|_{\cX_p}
\]
when $q =\infty$.

Moreover, the completion of $K_{\Phi}(\cX,\dot{D}_{\cA_{p}})$ is isomorphic to $\dot{B}^{2\theta}_{p,q}$ $($resp. $\dot{b}^{2\theta}_{p})$, when $q \neq \infty$ $($resp. $q = \infty)$.

\end{proposition}
\begin{proof}
First, we consider the case $q \neq \infty$.
For the left inequality in \eqref{equivalent norm besov}, note that for all $f \in \cS_{\infty}$:
\[
\cF f (\xi)=  (e^{t|\xi|^2}-1)^{-1}\cF (e^{-t\Delta}-I)f(\xi), \qquad \xi \in \R^n,
\, t>0.
\]
Taking $t = 2^{-2k}, k \in \mathbb{Z}$, and using the fact that for any $p \in [1,\infty]$, functions 
\[
m_k(\xi) := \phi_k (\xi) (e^{|2^{-k} \xi|^2}-1)^{-1}, \qquad \xi \in \R^n,
\]
are symbols of $L^p$-Fourier multiplier operators, whose norms might be bounded by a constant independent of $k \in \mathbb{Z}$, we get
\[
\|\widecheck{\phi}_k \ast f\|_{\cX_p} \lesssim \|e^{-2^{-2k}\Delta}f - f\|_{\cX_p} \leq \omega_{\cA_{p}}(2^{-2k},f) \quad (k \in \mathbb{Z}).
\]
By Lemma \ref{hom isom}, $\bigl\|\,(\cdot)^{-1}\omega_{\cA_{p}}(\cdot,f)\bigr\|_{\Phi}\simeq \bigl\|\, \|\cA_{p}e^{(\cdot)\cA_{p}}f\|_{\cX_{p}}\bigr\|_{\Phi}$, hence we obtain:
\begin{align*}
\sum_{k \in \mathbb{Z}} 2^{2qk\theta} \| \widecheck{\phi}_k \ast f\|^q_{\cX_p} & \lesssim \sum_{k \in \mathbb{Z}} 2^{2qk\theta} \|e^{-2^{-2k}\Delta}f - f\|^q_{\cX_p} \\ & \lesssim \int_0^{\infty} t^{-q\theta - 1} \omega_{\cA_{p}}(t,f)^q dt  \\ & \lesssim \int_0^{\infty} t^{q(1-\theta)-1}\|\Delta e^{-t\Delta} f\|^q_{\cX_p} dt.
\end{align*}
For the second inequality, note that
\[
\|\Delta e^{-t\Delta}f\|^q_{\cX_p} \lesssim \sum_{k \in \mathbb{Z}} 2^{2kq} \|\widecheck{\phi}_k \ast f\|^q_{\cX_p},
\]
as in \cite[Theorem 6.7.4]{BeLo76}, and, by the sectoriality of $\cA_{p}$,
\[
\|\Delta e^{-t\Delta}f\|^q_{\cX_p} \lesssim \sum_{k \in \mathbb{Z}} \frac{1}{t^q} \|\widecheck{\phi}_k \ast f\|^q_{\cX_p}.
\]
Both inequalities are justified, since $\lim_{j \to \infty} \sum_{|k|<j}\widecheck{\phi}_k \ast f = f$ in the $\cS$ topology (see \cite[Proposition 2.11, p.270]{Sawa18}) and therefore also in the norm of $\cX_p$.
This yields
\begin{align*}
\int_0^{\infty} t^{q(1-\theta)-1} \|\Delta e^{-t\Delta}f\|^q_{\cX_p} dt & \lesssim \sum_{k \in \mathbb{Z}} \Big( \int_0^{2^{-2k}} t^{q(1-\theta)-1} 2^{2kq}  \|\widecheck{\phi}_k \ast f\|^q_{\cX_p} \\ & \quad+ \int_{2^{-2k}}^{\infty} t^{-q\theta - 1}  \|\widecheck{\phi}_k \ast f\|^q_{\cX_p} \Big)  \\ & \lesssim \sum_{k \in \mathbb{Z}} 2^{2kq\theta}  \|\widecheck{\phi}_k \ast f\|^q_{\cX_p}.
\end{align*}
For the case $q=\infty$ the proof might be performed analogously, but the left inequality in \eqref{equivalent norm besov} can be proved in a slightly easier manner by a suitable choice of a Fourier multiplier. Indeed, note that
\[
\cF f(\xi) = -t \frac{1}{t |\xi|^2} e^{-t|\xi|^2}\cF \Delta e^{-t\Delta} f(\xi), \qquad  \xi \in \R^n,\, f \in \cS_{\infty}, \, t>0.
\]
Taking $t = 2^{-2k}$ and using the fact that for any $p \in [1,\infty]$ functions 
\[
m_k(\xi) := \phi_k(\xi) \frac{1}{|2^{-k}\xi|^2} e^{|2^{-k}\xi|^2}, \qquad \xi \in \R^n,
\]
are symbols of $L^p$-Fourier multiplier operators, whose norm might be bounded by a constant independent of $k$, we get
\[
\|\widecheck{\phi}_k \ast f\|_{\cX_p} \lesssim 2^{-2k} \|\Delta e^{-2^{-2k}\Delta}f\|_{\cX_p}, \qquad k \in \mathbb{Z}.
\]
Hence,
\[
\sup_{k \in \mathbb{Z}} 2^{2k\theta} \| \widecheck{\phi}_k \ast f\|_{\cX_p} \lesssim \sup_{k \in \mathbb{Z}} 2^{-2k(1-\theta)} \|\Delta e^{-2^{-2k}\Delta} f\|_{\cX_p} \leq \sup_{t>0} t^{1-\theta} \|\Delta e^{-t\Delta} f\|_{\cX_p} .
\]
The completion of $K_{\Phi}(\cX,\dot{D}_{\cA})$ is the same as the completion of $\cS_{\infty}$, so the last statement follows from Lemma \ref{isomoprhism of completions}.
\end{proof}

For $1 < p < \infty$ we can express the homogeneous part $[\cdot]_{\Phi,\cA_p}$ of the norm of $K_{\Phi}(\cX, D_{\cA_p})$ in terms of the modulus of continuity. By $\dot{H}^2_p(\R^n)$
 we denote the homogeneous Sobolev space; see \cite[p.147]{BeLo76}.
\begin{proposition}\label{K for dom} Let $1 <p <\infty$ and let $\cA_p$, $\cX_p$ be as above. Then, the following statements hold: 
\begin{itemize}
\item [(i)] The norms $\|\cdot\|_{\dot{H}^2_p(\R^n)}$ and $\|\cdot\|_{\dot{D}_{\cA_p}}$ are equivalent on $ \dot D_{\cA_p}$.  
\item [(ii)] For any $\phi \in K_{\Phi}(\cX,\dot{D}_{\cA_p})$ we have
\[
K(t, \phi, \cX, \dot D_{\cA_p}) \simeq  
K(t, \phi, \cX, \dot H^2_p(\R^n)) \simeq \omega^2_p(\sqrt{t}, \phi):= \sup_{|h|\leq t}\|\delta_h^2 \phi\|_{L^p(\R^n)}, 
\] where $\delta^2 \phi : = \phi - 2\phi(\cdot -h) + \phi(\cdot - 2h)$ 
 and the equivalence constants are independent on $\phi$.
\end{itemize}
\end{proposition}
\begin{proof}
(i) The inequality $\|\phi\|_{\dot{D}_{\cA_{p}}} \leq \|\phi\|_{\dot{H}^2_p(\R^n)}$ follows readily from the triangle inequality. The reverse inequality follows from the well-known fact that for any $i,j \in \{1,\ldots,n\}$ and $p \in (1,\infty)$ the functions $m(\xi) = \frac{\xi_i \xi_j}{|\xi|^2}$,  $\xi \in \R^n$, are $L^p$-Fourier multiplier symbols.

(ii) The equivalence $K(t, \cdot , \cX, \dot D_{\cA_{p}}) \simeq  
K(t, \cdot, \cX, \dot H^2_p(\R^n))$ follows easily from (i) and $K(t, \cdot, \cX, \dot H^2_p(\R^n)) \simeq \omega^2_p(\sqrt{t}, \cdot)$ is the standard result (see \cite[Theorem 1]{JoSc77})
\end{proof}

We present a consequence of Proposition \ref{XD}, which we find to be of considerable interest.

\begin{corollary}\label{homogeneous besov in sum of lebesgue}
Let $n \geq 3$ and $p \in (1,\frac{n}{2})$. Then, for any parameter $\Phi$ of the $K$-method, the space $K_{\Phi}(\cX_p,\dot{\cD}_{\cA_p})$ admits a completion in $L^p(\mathbb{R}^n) + L^r(\mathbb{R}^n)$, where $r = \frac{np}{n - 2p}$.
\end{corollary}

Indeed, the condition $p \in (1,\frac{n}{2})$ implies, by Proposition \ref{K for dom}(i) and the Sobolev inequality \cite[Corollary 9.13, p.~283]{Br10} (cf. Proposition 9.9 therein), that  
\[
\dot{\mathcal{D}}_{\mathcal{A}_p} \hookrightarrow L^r(\mathbb{R}^n) \hookrightarrow L^p(\mathbb{R}^n) + L^r(\mathbb{R}^n).
\]

The closedness of $\cA_p$ in $L^p(\R^n)+L^r(\R^n)$ and the properties of $\cT_p$ stated as assumptions in Proposition \ref{XD} can be read from the proofs of Proposition \ref{intDom} and Theorem \ref{interp result}. Thus, the statement follows directly from Proposition \ref{XD}. 

As a consequence of Proposition \ref{equivalent norm on besov} and Corollary \ref{homogeneous besov in sum of lebesgue}, 
if $p \in (1,\frac{n}{2})$, then for each $\theta \in (0,1)$ and $q \in [1,\infty)$ the homogeneous Besov space $\dot{B}^{2\theta}_{p,q}$ is isomorphic to a subspace of $L^p(\R^n)+L^r(\R^n)$, where $r = \frac{np}{n-2p}$.

Note that for any $p\in (1,\infty)$, $q\in [1,\infty)$ and $\theta \in (0,1)$, the statement (iii) of Lemma \ref{hom isom} apply to $\cA: = \cA_p$, $\cX = \cX_p$ and $\Phi$ as specified on the beginning of this section. In particular,  the completion of $K_\Phi(\cX, \dot{D}_\cA)$ is isomorphic to  $K_\Phi(J\cX_p,\dot H^2_p(\R^n))$. 
It is sufficient to combine Lemma \ref{hom isom}(iii) with \cite[Theorem 3.4.2, p.47]{BeLo76} and \cite[Theorem 6.3.2, p.148]{BeLo76}. 

The $L^1$-m.r.\,of the negative Laplacian on the homogeneous Besov spaces $\dot{B}^s_{p,1,2}\subset\cS'_h$ is stated in \cite{DaMu12} (for any $s \in \R$ and $p \in [1,\infty]$). Since the corresponding {\it force terms space} $L^1(\R_+;  \dot{B}^s_{p,1,2})$ for the associated Cauchy problem is not complete when $s > \frac{n}{p}$, the statement of that result is somewhat inconsistent with underlying intentions (e.g. the values of resulting solutions do not necessary belong to $\cS'_h$). 

The following corollary provides an extension of that result for the negative {\it homogeneous} Laplacian on the {\it classical} homogeneous Besov spaces $\dot B^{s}_{p,q} = \dot B^{s}_{p,q,1}$ for $s\in \R$, $p\in [1,\infty]$ and $q\in [1,\infty)$.
These spaces are complete, but in general, they possess realizations in $\cS'_\infty$, and consequently the values of solutions corresponding to all admissible force terms do not belong to $\cS'$, that is, somewhat lose their physical interpretation.
Since  $\dot B^{s}_{p,1} \simeq \dot{B}^s_{p,1,2}$ for $s\leq \frac{n}{p}$ and $p\in [1,\infty]$, the corresponding conclusion of this corollary coincides with that from \cite{DaMu12}.

However, for the full range of parameters $p$, $q$ and $s$, our preparatory results presented above (Propositions \ref{equivalent norm on besov} and \ref{K for dom}) allow us to place such studies within the framework established in Sections \ref{sect DaP-G} and 4; see, e.g. Propositions \ref{hom ext p} and \ref{hom ext infty}, and  Corollary \ref{hom ext 1-p}. 

Let $\sigma \in \R$. We define the Bessel potential operator on $\cS' / \cP$ by the formula
\[
\dot{I}^{\sigma}(f + \cP) := J \cF^{-1}|\xi|^{\sigma}\cF f, \qquad f \in \cS'.
\]

This operator is also called the lifting operator due to the fact that it is an isomorphism of $\dot{B}^s_{p,q}$ onto $\dot{B}^{s-\sigma}_{p,q}$ for any $s \in \R$, $p,q \in [1,\infty]$ (see \cite[Theorem 1, Section 5.2.3, p.242]{Tr83}). The homogeneous counterpart $\dot{\Delta}$ of the distributional Laplacian is defined as $\dot{\Delta} := -\dot{I}^2$.

\begin{corollary}\label{final application}
Let $s \in \R$, $p\in[1,\infty]$ and $q \in [1,\infty]$. 

{\rm(i)} The part of the negative homogeneous Laplacian $-\dot{\Delta}$ in $\dot{B}^s_{p,q}$
 has $L^r$-m.r.\,on $\R_+$ for all $r \in (1,\infty) \cup \{q\}$.

{\rm(ii)} Let $\cA_p$ and $\cX_p$  have the meaning specified at the beginning of this section. Let $\mathbb{F}$ be the subspace of $L^1_{loc}([0,\infty); \cX_p)$ given by 
\[
 \mathbb{F}:=  \bigcup_{s>0 \atop r\in (1,\infty)\cup \{q\}} L^r_{loc}([0,\infty); B^{s}_{p,q})
\]
Then for every $f\in \mathbb{F}$ the corresponding mild solution $u$ of the Cauchy problem $($CP$_{\cA_p, f,0})$ belongs to $W^{1,1}_{loc}([0,\infty);\cX_p)$, $u(t) \in D(\cA_p)$ and 
\[
u'(t) - \Delta u(t) = f(t)\qquad \textrm{for a.e. } t>0.
\]
Moreover, for any $q \in [1,\infty]$ and  $\theta \in (0,1)$, if $\| f \|_{\dot B^{2\theta}_{p,q}} \in E_v(\R_+)$, for any $E$ and $v$ which satisfy the conditions specified in Proposition \ref{hom ext p}$($b$)$, then 
\[
\| u \|_{\dot B^{2\theta}_{p,q}}, \, \| u' \|_{\dot B^{2\theta}_{p,q}},\, \| \Delta u \|_{\dot B^{2\theta}_{p,q}}\in E_v(\R_+).
\]
In particular, the last statement holds for $E = L^\alpha$ with $\alpha\in (1,\infty)$ and any non-increasing weight function $v$ on $(0,\infty)$. 
\end{corollary}

\begin{proof} 
(i) First, we prove the statement for any $s \in (0,2)$.
Recall that the part $-\Delta_{B^{s}_{p,q}}$ of $\cA_p = -\Delta$ (on $\cX_p$) in $B^s_{p,q}$ has $L^r$-m.r.\,on finite intervals for any $r \in (1,\infty) \cup \{q\}$. Note that each $g\in J\cX_p$ (see \eqref{Definition of J})  has exactly one representative in $\cX_p$, so the map $J\cX_p \ni g \mapsto \phi \in \cX_p$, where $\phi \in g\cap \cX_p$ is well-defined bijection. Denote its inverse, which of course coincides with $J_{|\cX_p}$, by $j$. 
Let $A:= -j \Delta_{B^s_{p,q}} j^{-1}$ with $D(A): = jD(\Delta_{B^s_{p,q}})$ and let $X: = j B^s_{p,q}$ be equipped with the image norm, i.e. $\|j\phi\|_X: = \|\phi\|_{B^s_{p,q}}$, $\phi\in B^s_{p,q}$. Of course, $A$ inherits the $L^r$-m.r.\,on finite intervals from $-\Delta_{B^{s}_{p,q}}$. 
Set $\|j\phi\|_{*} := [\phi]_{\Phi, \cA_p}$, $j\phi \in X$, where $\Phi$ is the parameter corresponding to the functor $(\,\cdot\,)_{\frac{s}{2}, q}$, that is, $\Phi = L^q(t^{q(1-\frac{s}{2}) - 1} )$ when $q\in [1,\infty)$ and $\Phi = L^\infty(t^{s/2})$ when $q=\infty$.  
Combining Theorem \ref{thm main}(ii) with Proposition \ref{hom ext p}(a) if $q\in [1,\infty)$ and the proof of Theorem \ref{thm main 2} (see \eqref{charact oo}) with Proposition \ref{hom ext infty} if $q=\infty$, we deduce that the canonical extension $\wA$ of $A$ on a completion of $(X, \|\cdot\|_{*})$ has $L^r$-m.r.\,on $\R_+$ for any $r\in (1,\infty)\cup \{q\}$. 

By Proposition \ref{equivalent norm on besov}, we can assume that $\wX = \dot B^{s}_{p,q}$ (as sets) with equivalent norms and $\wA$ is an extension of $A$, i.e. $A\subset \wA$. 
Note that the operator $\wA$ is a closure of $A$, that is, the smallest closed extension of $A$. Indeed, since $D(A)$ is dense in $\wX$ and is invariant under the $C_0$-semigroup $\wT$ generated by $\wA$, the set of vectors $\{ \frac{1}{\epsilon} \int_0^\epsilon \wT (\tau) j\phi\, \ud \tau: \, \epsilon >0, \, j\phi  \in D(A)\} \subset D(A)$ is dense in $D_{\wA}$. 
In particular, for any $\widetilde \phi \in D(\wA)$ there exists a sequence $(j\phi_n)$ in $D(A)$ such that $j\phi_n\rightarrow \widetilde \phi$ and $\wA j \phi_n = -j \Delta \phi_n = \dot I^2 \phi_n \rightarrow \wA \widetilde \phi$ in $\cX$ as $n\rightarrow \infty$. 
Therefore, to prove that the part $\dot A  = \dot  A_{s,p,q}$ of $\dot I^2$ in $\dot B^s_{p,q}$ has $L^r$-m.r.\,on $\R_+$, it suffices to show
that $\dot A$ is closed. However, since $\dot I^2$ is an isomorphism from $\dot B^{s+2}_{p,q}$ onto $\dot B^{s}_{p,q}$,  $D(\dot A) = \dot{B}^s_{p,q} \cap \dot{B}^{s+2}_{p,q}$ equipped with the graph norm of $\dot A$ is a Banach space, it yields the desired claim and completes the proof in the case when $s\in (0,2)$.

For other values of the parameter $s \in \R$, fix any $s_0 \in (0,2)$ and $r \in (1,\infty)\cup\{q\}$. Let $f \in L^r(\R_+;\dot{B}_{p,q}^{s})$. Since $\dot{I}^{s_0-s} : \dot{B}^{s_0}_{p,q} \to \dot{B}^{s}_{p,q}$ is an isomorphism, clearly $\dot{I}^{s-s_0}f \in L^r(\R_+;\dot{B}_{p,q}^{s_0})$. By the $L^r$-m.r.\,of $A_{p,q,s_0} =: \dot A$, there exists $\widetilde{u} \in W^{1,1}_{loc}(\R_+;\dot{B}^{s_0}_{p,q})$ with $\widetilde u (t) \in D(\dot A)$, 
\[
\widetilde{u}'(t) + \dot A \widetilde{u}(t) = \dot{I}^{s-s_0}f(t)\qquad 
\textrm{for a.e. } t>0
\]
and
\[
\|\dot A \tilde{u}\|_{L^r(\R_+;\dot{B}^{s_0}_{p,q})} \leq C\|\dot{I}^{s-s_0}f\|_{L^r(\R_+;\dot{B}^{s_0}_{p,q})}
\]
for some constant $C$ independent on $f$. Since $\dot{I}^{s_0-s} \dot A \widetilde{u}(t) = A_{p,q,s} \dot{I}^{s_0-s} \widetilde{u}(t)$ and $\dot{I}^{s_0-s}\widetilde{u}'(t) = (\dot{I}^{s_0-s}\widetilde{u})'(t)$ for a.e. $t>0$, it follows that the function $u(t) := \dot{I}^{s_0-s} \widetilde{u}(t)$, $t>0$, belongs to $W^{1,1}_{loc}(\R_+;\dot{B}^{s}_{p,q})$, $u(t) \in D(A_{p,q,s})$  and 
\[
u'(t) + \cA_{p,q,s}u(t) = f(t) \qquad 
\textrm{for a.e. } t>0.
\]
Moreover, the following estimate holds
\[
\|\cA_{p,q,s}u\|_{L^r(\R_+;\dot{B}^{s}_{p,q})} \leq C \|\dot{I}^{s_0-s}\|_{\cL(\dot{B}^s_{p,q};\dot{B}^{s_0}_{p,q})}\|\dot{I}^{s-s_0}\|_{\cL(\dot{B}^{s_0}_{p,q};\dot{B}^s_{p,q})}\|f\|_{L^r(\R_+;\dot{B}^{s}_{p,q})},
\]
which completes the proof.

(ii) All the statements in (ii), except for the implication
\[
\| f \|_{\dot{B}^{2\theta}_{p,q}} \in E_v(\R_+) \quad \Rightarrow \quad \| u \|_{\dot{B}^{2\theta}_{p,q}} \in E_v(\R_+)
\] 
can be deduced from Corollary $\ref{hom ext 1-p}$.
We turn to the proof of the above implication.
As the first part of the proof of (i) shows, the operator $\dot A = \dot A_{2\theta, p,q}$ has $L^r$-m.r.\,on $\R_+$ and $0\in \rho(\dot A)$. Therefore, since a semigroup $\dot T$ generated by $\dot A$ on $B^{2\theta}_{p,q}$ is uniformly bounded on $\R_+$ (see \cite[Theorem 4.3]{Do00}) and holomorphic, $\dot T$ has the negative exponential growth bound. In particular, by Young's inequality, the corresponding solution operator $U_{\dot A}$ is bounded on $L^2(\R_+; \dot B^{2\theta}_{p,q})$. Again, the boundedness and analyticity of $\dot T$, gives 
\[
\sup_{t>0,\, l=0,1}\Big \|t^{l+1}\frac{\ud^l }{\ud t^l } \dot T(t)\Big\|_{\cL(\dot B^{2\theta}_{p,q})} = \sup_{t>0,\, l=0,1}\|t^{l+1} \dot A^l \dot T(t)\|_{\cL(\dot B^{2\theta}_{p,q})}<\infty.
\]
Thus, $\dot U$ is a singular operator with the kernel $\dot T$ satisfying the standard Calder\'on-Zygmund condition. 
The same extrapolation argument as that applied in the proof of Proposition \ref{hom ext p}(a), based on \cite[Theorems 5.1 and 4.3]{ChKr18}, gives that $U_{\dot A}$ is bounded on $E_v(\R_+; \dot B^s_{p,q})$. It readily gives the desired implication. The fact, that every positive, non-increasing function on $(0,\infty)$ belongs to $A^-_{\alpha}(\R_+)$ is stated in the remark before Proposition~\ref{hom ext p}. This completes the proof.
\end{proof}

We conclude with the following observation. Since for any $s'>s>0$, $B^{s'}_{p,q} \subset B^{s}_{p,q} \subset \cX_p$, the space $\mathbb{F}$ in Corollary \ref{final application}(ii) could be considered as a {\it limiting} force term space for the $L^q$-m.r.\,property of the negative Laplacian $\cA_p = - \Delta$ on $\cX_p$.  Recall that, e.g. if $q=1$, then $\cA_p$ on $\cX_p$ ($p\in [1,\infty]$) does not have $L^1$-m.r.\,on any finite interval; cf. also Theorem \ref{thm main}(iii). 

\vspace{5 mm}

\noindent\textbf{Data availability:} no data was used for the research described in the article.

\providecommand{\bysame}{\leavevmode\hbox to3em{\hrulefill}\thinspace}
\providecommand{\MR}{\relax\ifhmode\unskip\space\fi MR }
\providecommand{\MRhref}[2]{%
  \href{http://www.ams.org/mathscinet-getitem?mr=#1}{#2}
}
\providecommand{\href}[2]{#2}

\end{document}